\documentclass[a4paper,onecolumn,superscriptaddress,10pt]{compositionalityarticle}
\pdfoutput=1

\usepackage{./sty/leosstylefile}
\usepackage{./sty/tikzit}
\usetikzlibrary{positioning,intersections,quotes}

\tikzstyle{node}=[fill=black, draw=black, shape=circle, scale=0.5]
\tikzstyle{wnode}=[fill=white, draw=black, shape=circle, scale=0.5]
\tikzstyle{textbox}=[inner sep=2pt, shape=rectangle, fill=none]
\tikzstyle{textnode}=[inner sep=0mm, shape=circle, fill=white]
\tikzstyle{gnode}=[inner sep=0mm, minimum size=1mm, fill={rgb,255: red,221; green,221; blue,221}, draw={rgb,255: red,221; green,221; blue,221}, shape=circle]
\tikzstyle{refine}=[fill=black, draw=black, shape=regular polygon, regular polygon sides=3, rotate=180, scale=0.5]
\tikzstyle{coarsen}=[fill=white, draw=black, shape=regular polygon, regular polygon sides=3, scale=0.5]
\tikzstyle{bdytextbox}=[fill=white, draw=black, shape=rectangle]
\tikzstyle{redbox}=[fill=white, draw=red, shape=rectangle, text=red]
\tikzstyle{bluecirc}=[inner sep=1mm, fill=white, draw={rgb,255: red,4; green,51; blue,255}, shape=circle, text={rgb,255: red,4; green,51; blue,255}]
\tikzstyle{rednode}=[fill=red, draw=red, shape=circle, scale=0.5]
\tikzstyle{new style 0}=[fill=white, draw=red, shape=circle, scale=0.5]
\tikzstyle{bluenode}=[fill={rgb,255: red,4; green,51; blue,255}, draw={rgb,255: red,4; green,51; blue,255}, shape=circle, scale=0.5]
\tikzstyle{yellownode}=[fill={rgb,255: red,255; green,210; blue,75}, draw={rgb,255: red,255; green,210; blue,75}, shape=circle, scale=0.5]
\tikzstyle{blacksq}=[fill=black, draw=black, shape=rectangle, scale=0.5]
\tikzstyle{bluetext}=[fill=none, draw=none, shape=rectangle, text={rgb,255: red,4; green,51; blue,255}]
\tikzstyle{reg}=[draw, fill=white, rounded rectangle, rounded rectangle left arc=none, minimum height=1em, minimum width=1em, node font={\scriptsize}]
\tikzstyle{coreg}=[draw, fill=white, rounded rectangle, rounded rectangle right arc=none, minimum height=1em, minimum width=1em, node font={\scriptsize}]
\tikzstyle{turquoisenode}=[fill={rgb,255: red,115; green,255; blue,239}, draw=black, shape=circle, scale=0.5]
\tikzstyle{resistor}=[R]
\tikzstyle{inductor}=[L]
\tikzstyle{capacitor}=[C]
\tikzstyle{voltage-source}=[american voltage source]
\tikzstyle{current-source}=[american current source]
\tikzstyle{togray}=[none]

\tikzstyle{edge}=[-, draw=black]
\tikzstyle{diredge}=[->, draw=black]
\tikzstyle{dashed edge}=[-, draw=black]
\tikzstyle{loosedash}=[-, dash pattern=on 1.5pt off 3pt, draw=black]
\tikzstyle{dirdash}=[->, dashed, dash pattern=on 2pt off 0.5pt, draw=black]
\tikzstyle{mapsto}=[{|->}, draw=black]
\tikzstyle{gray diredge}=[draw={rgb,255: red,221; green,221; blue,221}, ->]
\tikzstyle{dark grey dirdash}=[->, dashed, dash pattern=on 2pt off 0.5pt, draw={rgb,255: red,81; green,81; blue,81}]
\tikzstyle{doubedge}=[-, draw=black, double=none, double distance=3pt, inner sep=0pt, thick]
\tikzstyle{thedge}=[-, line width=1.5pt, draw=black]
\tikzstyle{gray dashed}=[-, dashed, dash pattern=on 1pt off 1.5pt, draw={rgb,255: red,128; green,128; blue,128}]
\tikzstyle{rededge}=[-, draw=red]
\tikzstyle{gray edge}=[-, draw={rgb,255: red,128; green,128; blue,128}]
\tikzstyle{blthedge}=[-, thick, draw={rgb,255: red,4; green,51; blue,255}]
\tikzstyle{blthedge-extend}=[-, thick, draw={rgb,255: red,4; green,51; blue,255}, shorten >= -0.3pt, shorten <= -0.3pt]
\tikzstyle{blthdash}=[-, dashed, dash pattern=on 3pt off 1pt, thick, draw={rgb,255: red,4; green,51; blue,255}]
\tikzstyle{dirrededge}=[draw=red, ->]

\tikzstyle{object}=[inner sep=0mm, shape=circle, fill=none]
\tikzstyle{bullet}=[fill=black, draw=black, shape=circle, scale=0.3]
\tikzstyle{circ}=[fill=white, draw=black, shape=circle, scale=0.3]
\tikzstyle{objectbox}=[inner sep=3pt, shape=rectangle, fill=none]
\tikzstyle{bdyobjectbox}=[fill=white, draw=black, shape=rectangle]

\tikzstyle{morphism}=[->, draw=black]
\tikzstyle{dash morphism}=[->, dashed, dash pattern=on 1.5pt off 1pt, draw=black]
\tikzstyle{mapsto}=[{|->}, draw=black]
\tikzstyle{nat transf}=[-implies, double, double distance=3pt, thick]
\tikzstyle{gray nat transf}=[-implies, draw=gray, double, double distance=3pt, thick]
\tikzstyle{equality}=[-, double, double distance=3pt]
\tikzstyle{squig morphism}=[->, draw=black, line join=round, decorate, decoration={zigzag, segment length=4, amplitude=.9,post=lineto, post length=2pt}]
\tikzstyle{hookarrow}=[right hook->, draw=black]
\tikzstyle{hookarrowmirror}=[left hook->, draw=black]
\tikzstyle{dotted line}=[-, dotted, draw=black]

\newcommand{\authorsforheader}{Lobski and Zanasi}
\newcommand{\compositionalityIssue}{XX}
\newcommand{\compositionalityVolume}{X}
\newcommand{\paperdoinumber}{10.46298/compositionality-\compositionalityVolume-\compositionalityIssue}
\newcommand{\paperdoi}{https://doi.org/\paperdoinumber}
\newcommand{\receiveddate}{yyyy-mm-dd}
\newcommand{\accepteddate}{yyyy-mm-dd}

\begin{document}

\title{Layered Monoidal Theories II: Fibrational Semantics}
\author{Leo Lobski}
\email{leo.lobski.21@ucl.ac.uk}
\orcid{0000-0002-0260-0240}
\author{Fabio Zanasi}
\email{f.zanasi@ucl.ac.uk}
\orcid{0000-0001-6457-1345}
\affiliation{University College London, United Kingdom}
\maketitle

\begin{abstract}
{\em Layered monoidal theories} provide a categorical framework for studying scientific theories at different levels of abstraction, via string diagrammatic algebra. We introduce models for three closely related classes of layered monoidal theories: {\em fibrational}, {\em opfibrational} and {\em deflational} theories. We prove soundness and completeness of these theories for the respective models. Our work reveals connections between layered monoidal theories and well-known categorical structures such as Grothendieck fibrations and displayed categories. 
\end{abstract}

\section{Introduction}\label{sec:intro}
Layered monoidal theories generalise {\em monoidal} theories: their diagrammatic language allows several monoidal theories to be combined, together with functorial translations between them, all within a single string diagram. The motivation for the formalism is that a single system may admit multiple intertranslatable descriptions at different levels of granularity. Examples include chemical reactions, which may be modelled abstractly as reaction networks or more finely as molecular graphs~\cite{chem-trans-motifs}, and computer architecture, where high-level descriptions must ultimately be implemented as microelectronic circuits~\cite{kaye-thesis}. In Part I of this work~\cite{lmt-part1}, we focussed on the \emph{syntax} of layered monoidal theories. In this second part, we introduce their \emph{semantic models} and investigate questions of soundness and completeness. Since we begin by recalling the syntax, the exposition is self-contained.

A layered monoidal theory consists of several monoidal theories, together with monoidal functors between them. A typical term in a layered theory is drawn on the left of Figure~\ref{fig:typical-term}, where $x$ is a term in the monoidal theory $\omega$ (drawn as a box on the left) and $y$ is a term in the monoidal theory $\tau$ (drawn as a box on the right). The dotted lines indicate that $x$ and $y$ sit above $\omega$ and $\tau$, and that the ``functor boundary'' $\refine_f$ sits above $f$. Terms may slide through the functor boundaries from one monoidal theory to another, as shown on the right of Figure~\ref{fig:typical-term}. Albeit highly intuitive, this is a purely syntactic formalism, leaving open how to semantically interpret such structures. In this paper we show that these diagrammatic transformations can be interpreted in a sound way in well-studied mathematical structures. Beyond establishing internal coherence of the formalism, our work reveals a connection between layered monoidal theories and well-known categorical structures, such as fibrations~\cite{grothendieck2004,borceux94-fibred,jacobs-cltt,johnstone02-indexed,loregian-riehl20} and displayed categories~\cite{displayed-categories,benabou}.

\begin{figure}[h]
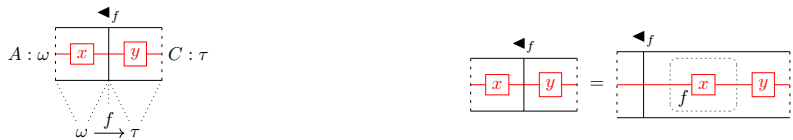

\begin{minipage}{0.45\textwidth}
\begin{equation*}
\scalebox{.7}{\input{journal-figures/term-layered-theory.tikz}}
\end{equation*}
\end{minipage}
\begin{minipage}{0.45\textwidth}
\begin{equation*}
\scalebox{.7}{\input{journal-figures/term-layered-pushing.tikz}}
\end{equation*}
\end{minipage}
\caption{Left: A typical term in a layered theory. Right: An equation for the functor boundary.%
\label{fig:typical-term}}
\end{figure}

As a necessary step towards semantics, we need to refine the classification of layered monoidal theories. The definition of a layered monoidal theory (Definition~\ref{def:layered-theory}) requires specifying a {\em recursive sorting procedure} (Definition~\ref{def:sorting-procedure}), which determines what kinds of terms are available within the theory. The role of such terms is to control what kinds of translations are present {\em between} the layers -- as opposed to the usual monoidal terms within a single layer. Different recursive sorting procedures give rise to different theories, even if the signature is kept fixed\footnote{Note that this is a departure from how monoidal theories are usually developed: in that case, there is a single sorting procedure, and consequently just one class of models -- (strict) monoidal categories.}. We shall consider three different (but related) recursive sorting procedures, corresponding to three classes of models. Namely, we define {\em opfibrational}, {\em fibrational} and {\em deflational} theories. The first two serve as building blocks for the third, which is the most expressive and central to our development.

The distinction between opfibrational, fibrational and deflational theories can be understood in terms of the functor boundaries (labelled as $\refine_f$ in Figure~\ref{fig:typical-term}). The choice of the direction of the functor boundary is, in principle, arbitrary: one may write $\omega$ on the right and $\tau$ on the left, while still depicting the functorial translation $f:\omega\rightarrow\tau$. It is, therefore, desirable to have a formalism where both directions are present, and the interaction between the functor boundaries is axiomatised. While the equations governing the behaviour of the boundaries should be symmetric, it is {\em a priori} not clear how to obtain both directions simultaneously.
In deflational theories, the most expressive case, boundaries indeed come in two forms: each generator between layers gives rise to both a ``forward'' and a ``backward'' boundary.

Opfibrational theories only contain the left-to-right functor boundaries as in Figure~\ref{fig:typical-term}, while fibrational theories reverse this direction. The two are mirror images of one another in all respects, and their corresponding models are categorically dual. Deflational theories are obtained by ``glueing'' the fibrational and the opfibrational version of the same theory along their shared internal morphisms (denoted by $x$ and $y$ in Figure~\ref{fig:typical-term}), thereby producing bidirectional boundaries. As a consequence of the modular construction of the theories, we need three classes of models -- one for each kind of theory mentioned above: {\em opfibrations with indexed monoids}, {\em fibrations with indexed comonoids} and {\em monoidal deflations} are the models of opfibrational, fibrational and deflational theories, respectively.

The paper establishes three soundness and completeness results, each arising from a free-forgetful adjunction. We summarise the results in Table~\ref{tab:soundness-completeness-results}: the first column lists the relevant theory, while the second column indicates the semantic structures with respect to which it is sound and complete. We include the classical result for monoidal theories as a point of reference. As expected from the close connection between the theories, the semantic structures are tightly linked: each (split) opfibration with indexed monoids contains models for monoidal theories as its fibres; each (split) fibration with indexed comonoids is obtained by dualising a (split) opfibration with indexed monoids; and the 1-categorical part of a (split) monoidal deflation is obtained by glueing a (split) opfibration with indexed monoids with its dual (split) fibration with indexed comonoids\footnote{What is meant by ``glueing'' is made precise by Corollary~\ref{cor:decomposition-biretrofunctor-opfibration} and Lemma~\ref{lma:monoidal-deflation-opfibration-indexed-monoids}.}.

\begin{table}
\centering
\renewcommand{\arraystretch}{1.5}
\begin{tabular}{ c c | c c }
theory & definition & semantics & definition \\
\hhline{==|==}
monoidal theory & -- & (strict) monoidal category & -- \\
\hline
opfibrational theory & \ref{def:opfib-layered-theory} & split opfibration with indexed monoids & \ref{def:opfib-indexed-mon} \\
\hline
fibrational theory & Figure~\ref{fig:fibrational-terms} & split fibration with indexed comonoids & -- \\
\hline
deflational theory & \ref{def:deflational-theory} & split monoidal deflation & \ref{def:monoidal-deflation}
\end{tabular}
\renewcommand{\arraystretch}{1}
\caption{The summary of theories and the semantics they characterise.\label{tab:soundness-completeness-results}}
\end{table}

From another perspective, (op)fibrational layered theories provide a string diagrammatic language for (strict) {\em indexed monoidal categories}. Moeller and Vasilakopoulou introduced a {\em monoidal Grothendieck construction}~\cite{monoidal-gro}, which provides an equivalence between (op)indexed monoidal categories and pseudomonoids in the category of (op)fibrations. The present work can be seen as the first step towards a diagrammatic treatment of this result. However, we remark that, unlike in~\cite{monoidal-gro}, all the monoidal categories that we treat as models are {\em strict}; see Subsection~\ref{subsec:opfibrations-indexed-monoidal}.

In light of their connection to indexed monoidal categories, one might reasonably call the theories under study {\em indexed monoidal} rather than {\em layered monoidal}. However, indexed monoidal categories provide a proper subclass of models of layered monoidal theories -- an important class that is not contained in indexed monoidal theories are deflations (Section~\ref{sec:deflations}). We therefore retain the term layered to emphasise the broader scope of the framework.

\paragraph{Structure of the article.} We begin by covering quite some background in Section~\ref{sec:preli}. After briefly defining terms and equations for layered theories in Section~\ref{sec:layered-theories}, we introduce {\em opfibrations with indexed monoids} in Section~\ref{sec:indexed-monoidal} and {\em monoidal deflations} in Section~\ref{sec:deflations}. These will, respectively, provide the models for the {\em opfibrational} and {\em deflational} layered theories. Section~\ref{sec:semantics} contains the main results of the article: we show that terms organise themselves into {\em free models} upon quotienting by appropriate structural equations.

\section{Preliminaries}\label{sec:preli}
Here we cover the preliminaries needed to follow the ensuing discussion, and give references to the relevant literature for more details. While most things are more or less known, we use this section to establish some notational and terminological conventions that will be used throughout the paper. Specifically, we cover (op)fibrations~\ref{sec:fib-opfib}, retrofunctors~\ref{subsec:retrofunctors}, indexed categories~\ref{subsec:indexed-categories}, profunctors~\ref{subsec:profunctors}, profunctor collages~\ref{subsec:collages} and monoidal theories~\ref{subsec:models-monoidal-theories}.

\subsection{Fibrations and opfibrations}\label{sec:fib-opfib}

While (op)fibrations were originally introduced by Grothendieck~\cite{grothendieck2004}, here we mostly follow Jacobs~\cite{jacobs-cltt}. Other expository texts include Borceux~\cite{borceux94-fibred}, Johnstone~\cite{johnstone02-indexed} and Loregian \& Riehl~\cite{loregian-riehl20}.

A functor is usually denoted by a lowercase Latin letter $p:\cat Y\rightarrow\cat X$, or just by an arrow $\cat Y\rightarrow\cat X$ with no symbol.

Given a functor $p:\cat Y\rightarrow\cat X$, we say that an object $y\in\Ob(\cat Y)$ is {\em above} an object $x\in\Ob(\cat X)$ if $py=x$. Similarly, a morphism $g\in\cat Y$ is above a morphism $f\in\cat X$ if $pg=f$. Given a morphism $f\in\cat X$, we denote by $\cat Y_f$ the set of all morphisms above $f$. In the special case when $f$ is an identity on some object $x\in\Ob(\cat X)$, this defines the subcategory of $\cat Y$ known as the {\em fibre} of $x$: the objects are all objects of $\cat Y$ above $x$, the morphisms are all morphisms above $\id_x$. We denote the fibre of $x$ simply by $\cat Y_x$.

Given a function $p:\Ob(\cat Y)\rightarrow\Ob(\cat X)$, a {\em liftable pair} $(f:x\rightarrow y,b)$ consists of a morphism $f\in\cat X$ and an object $b\in\Ob(\cat Y)$ with $pb=y$. Similarly, an {\em opliftable pair} $(a,f:x\rightarrow y)$ consists of a morphism $f\in\cat X$ and an object $a\in\Ob(\cat Y)$ with $pa=x$.

Let us fix a functor $p:\cat Y\rightarrow\cat X$.
\begin{definition}[Opcartesian map]\label{def:opcartesian-map}
A morphism $F:a\rightarrow b$ in $\cat Y$ is called {\em $p$-opcartesian} if for any morphism $G:a\rightarrow c$ in $\cat Y$ and any morphism $h:pb\rightarrow pc$ in $\cat X$ such that $pF;h=pG$ (i.e.~the lower triangle below commutes), there is a unique morphism $H:b\rightarrow c$ above $h$ such that $F;H=G$ (i.e.~the upper triangle below commutes):
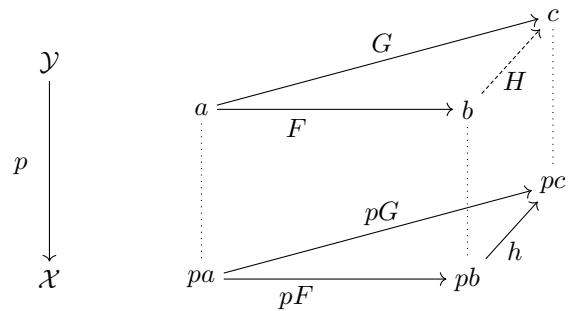
\begin{equation*}
\scalebox{1}{\input{journal-figures/opcartesian-definition.tikz}}.
\end{equation*}
\end{definition}

\begin{remark}[Weakly opcartesian map]\label{rem:weakly-opcartesian-map}
We obtain the notion of a {\em weakly $p$-opcartesian map} upon restricting the quantification to $h$ being the identity map in the above definition: for any morphism $G:a\rightarrow c$ {\em with $pc=pb$} such that $pF=pG$, there is a unique map $H$ above $\id_{pb}$ with $F;H=G$. Clearly, each opcartesian map is, in particular, weakly opcartesian. The converse does not hold in general.
\end{remark}

\begin{definition}[Opfibration]\label{def:opfibration}
A functor $p:\cat Y\rightarrow\cat X$ is an {\em opfibration} if for every opliftable pair $(a,f:x\rightarrow y)$, there is a $p$-opcartesian morphism $F:a\rightarrow b$ above $f$:
\begin{equation*}
\scalebox{1}{\input{journal-figures/opfibration-definition.tikz}}.
\end{equation*}
Such morphism $F$ is called an {\em opcartesian lifting} of $(a,f:x\rightarrow y)$. Note that we drop the prefix $p$ if the functor is clear from the context.
\end{definition}
In the context of (op)fibrations, we refer to the domain category $\cat Y$ as the {\em total} category, and to the codomain category $\cat X$ as the {\em base} category.

\begin{remark}[Preopfibration]\label{rem:preopfibration}
We obtain the notion of a {\em preopfibration} if we require $F$ in the above definition to be {\em weakly} $p$-opcartesian rather than $p$-opcartesian.
\end{remark}

By the defining universal property of opcartesian maps, opcartesian liftings are unique up to an isomorphism. An opfibration which comes with a choice of an opcartesian lifting for each opliftable pair is called {\em cloven}.

If $p:\cat Y\rightarrow\cat X$ is a cloven opfibration, then every morphism $f:x\rightarrow y$ in $\cat X$ induces a functor between the fibres $f^*:\cat Y_x\rightarrow\cat Y_y$ by mapping each $a\in\cat Y_x$ to the codomain of the chosen opcartesian lifting of $(a,f:x\rightarrow y)$, and each morphism $G\in\cat Y_x$ to the unique map above $\id_y$ induced by the universal property of the opcartesian lifting of the liftable pair $(\dom(G),f)$, where $\dom(G)$ is the domain of $G$. The functor $f^*$ is referred to as the {\em reindexing functor} induced by $f$. In fact, defining a reindexing functor for each morphism in the base category of an opfibration is equivalent to providing a choice of opcartesian liftings.
\begin{proposition}\label{prop:opfibration-to-pseudofunctor}
Any cloven opfibration $p:\cat Y\rightarrow\cat X$ induces a pseudofunctor $\cat X\rightarrow\Cat$ by sending each object to its fibre and each morphism $f$ to the functor $f^*$.
\end{proposition}
We say that an opfibration is {\em split} if the induced pseudofunctor is, in fact, a functor: i.e.~if for all $x\in\Ob(\cat X)$, the induced functor $\id_x^*$ is the identity functor on $\cat Y_x$ and for all composable maps we have $(f;g)^*=f^*;g^*$.

\begin{remark}
Note that the construction of the reindexing functor $f^*:\cat Y_x\rightarrow\cat Y_y$ already works for a preopfibration. However, the resulting assignment is not compositional (not even up to an isomorphism): the composition of two weakly opcartesian maps need not be weakly opcartesian. Thus, Proposition~\ref{prop:opfibration-to-pseudofunctor} fails for a preopfibration: its proof relies on the liftings being opcartesian (rather than merely weakly opcartesian). However, a preopfibration does induce a normal lax functor $\cat X\rightarrow\Prof$ into the 2-category of profunctors. We return to this in Section~\ref{subsec:collages} (see specifically Proposition~\ref{prop:displayed-factors-opfibration}).
\end{remark}

A fibration is the dual notion to opfibration, hence all the discussion here also holds for fibrations -- with the directions appropriately reversed.
\begin{definition}[Fibration]\label{def:fibration}
Given a functor $p:\cat Y\rightarrow\cat X$, a morphism $F:a\rightarrow b$ in $\cat Y$ is {\em $p$-cartesian} if it is $p$-opcartesian for $p:\cat Y^{op}\rightarrow\cat X^{op}$. The functor $p:\cat Y\rightarrow\cat X$ is a fibration if $p:\cat Y^{op}\rightarrow\cat X^{op}$ is an opfibration.
\end{definition}

The intuition behind (op)fibrations is that they capture indexing via disjoint unions of indexed sets. We demonstrate this with the following example.
\begin{example}[Family fibration]\label{ex:family-fibration}
Define the category $\Fam(\Set)$ of {\em families of sets} as follows:
\begin{itemize}
\item the objects are pairs $(I,f)$ of a set $I$ and a function $f:I\rightarrow\Set$ assigning to each element $i\in I$ a set $f(i)$,
\item a morphism $(I,f)\rightarrow (J,g)$ is a pair $\left(u,\left\{u_i\right\}_{i\in I}\right)$ of a function $u:I\rightarrow J$ and a family of functions $u_i:f(i)\rightarrow g(u(i))$,
\item the identity on $(I,f)$ is given by $(\id_I,\id_{fi})$,
\item the composition of $(u,u_i) : (I,f)\rightarrow (J,g)$ and $(v,v_j) : (J,g)\rightarrow (K,h)$ is given by $(u;v, u_i;v_{ui})$.
\end{itemize}
The projection $\Fam(\Set)\rightarrow\Set$ to the first component is a fibration:
\begin{itemize}
\item the cartesian maps $(u,\id_{fi}) : (I,f)\rightarrow (J,g)$ are given by commutative triangles
\begin{equation*}
\scalebox{1}{\input{journal-figures/families-fibration-cartesian.tikz}}
\end{equation*}
and the family of identity maps on each set $f(i)=g(u(i))$,
\item given a family of sets $(I,f)$ and a function $u:J\rightarrow I$, the cartesian lifting is given by
$$\left(u,\id_{(u;f)(j)}\right) : (J, u;f)\rightarrow (I,f).$$
\end{itemize}
\end{example}
Note that in Example~\ref{ex:family-fibration}, we did not use the structure of $\Set$ in the second component as the target of the indexing, apart from the fact that it is a category. This example, therefore, generalises to an arbitrary category: the category of {\em families over $\cat C$} is denoted by $\Fam(\cat C)$, and the projection to the first component $\Fam(\cat C)\rightarrow\Set$ is still a fibration.

Fibrations and opfibrations organise into categories. We will mostly focus on the fixed base case, however, the general case will also appear occasionally.
\begin{definition}[Morphism of (op)fibrations]
Let $p$ and $q$ in the diagram below be (op)fibrations. A {\em morphism} $(H,K):p\rightarrow q$ is a pair of functors such that the square
\begin{equation*}
\scalebox{1}{\input{journal-figures/morphism-opfibrations.tikz}}
\end{equation*}
commutes, and $H$ sends $p$-(op)cartesian maps to $q$-(op)cartesian maps.
\end{definition}
We denote the resulting categories of fibrations and opfibrations by $\Fib$ and $\OpFib$. For a fixed base category $\cat X$, there are also categories of (op)fibrations into $\cat X$, whose morphisms are commutative triangles which preserve the (op)cartesian maps: this is a special case of a morphism defined above when the functor $K$ between the bases is the identity. We denote the fixed base categories by $\Fib(\cat X)$ and $\OpFib(\cat X)$.

Given a morphism of (op)fibrations $(H,K):p\rightarrow q$, observe that $H$ sends $p$-(op)cartesian liftings to $q$-(op)cartesian liftings of the image of the (op)liftable pair under $K$. In the case of cloven (and hence split) opfibrations, we additionally require that the chosen liftings are preserved. The subcategories of split (op)fibrations are denoted by adding a subscript $\mathsf{sp}$ in all four cases defined above: the objects are split (op)fibrations, while morphisms are morphisms of (op)fibrations such that chosen liftings are mapped to chosen liftings.

Two important properties of (op)fibrations are closure under composition and pullbacks. We record this in the following propositions.
\begin{proposition}
If $p:\cat Y\rightarrow\cat X$ and $q:\cat X\rightarrow\cat Z$ are (op)fibrations, then $p;q : \cat Y\rightarrow\cat Z$ is also an (op)fibration.
\end{proposition}
\begin{proposition}\label{prop:opfibrations-pullbacks}
Let $p$ in the diagram below be an (op)fibration. If the square
\begin{equation*}
\scalebox{1}{\input{journal-figures/opfibrations-pullbacks.tikz}}
\end{equation*}
is a pullback, then $q$ is also an (op)fibration.
\end{proposition}
The above propositions imply that the categories $\Fib(\cat X)$ and $\OpFib(\cat X)$ are cartesian, with the cartesian products given by pullbacks. The terminal object is given by the identity functor $\id_{\cat X}:\cat X\rightarrow\cat X$. We denote the resulting cartesian categories by $(\Fib(\cat X),\boxtimes,\id_{\cat X})$ and $(\OpFib(X),\boxtimes,\id_{\cat X})$.

\subsection{Retrofunctors}\label{subsec:retrofunctors}

A split opfibration may be thought of as a way of providing chosen liftings for a functor in a way that is (1) functorial (it strictly preserves identities and composition), (2) compatible with the functor in the sense that the lifting is above the morphism that induces it, and (3) minimal in the sense made precise by the requirement of being opcartesian (Definition~\ref{def:opcartesian-map}). If we drop the third requirement (minimality), we obtain the notion of a {\em lens}, which provides {\em some} functorial assignment of liftings compatible with the functor. Further, dropping the second condition (compatibility), we obtain the notion of a {\em retrofunctor}, which is simply a functorial assignment of lifts. Note that this does not require any functor to be present in order to be defined -- a function on object suffices. We remark that retrofunctors are more commonly known in the literature as {\em cofunctors}. However, we agree with Di~Meglio~\cite{dimeglio-mscthesis} that the term cofunctor is misleading, as the concept is not dual to a functor in any meaningful way. We therefore stick to {\em retrofunctor}. We follow Clarke~\cite{clarke-thesis} in our definition of a retrofunctor.

\begin{definition}[Retrofunctor]\label{def:retrofunctor}
A {\em retrofunctor} $(p,\varphi) : \cat Y\rightarrow\cat X$ consists of a function $p:\Ob(\cat Y)\rightarrow\Ob(\cat X)$ and a function $\varphi$ from the opliftable pairs to $\Mor(\cat Y)$ satisfying the following properties:
\begin{itemize}
\item the opliftable pair $(a,f:x\rightarrow y)$ is mapped to $\varphi(a,f) : a\rightarrow a^{\varphi}$, where the codomain $a^{\varphi}$ is above $y$, i.e.~$p\left(a^{\varphi}\right)=y$,
\item $\varphi(a,\id_{pa})=\id_a$, and $\varphi(a,f;g)=\varphi(a,f);\varphi\left(a^{\varphi},g\right)$.
\end{itemize}
\end{definition}
We depict the action of a retrofunctor as follows:
\begin{equation*}
\scalebox{1}{\input{journal-figures/retrofunctor-definition.tikz}}.
\end{equation*}

\subsection{Indexed categories}\label{subsec:indexed-categories}

We define indexed and opindexed categories, and state the well-known equivalence between (op)fibrations and (op)indexed categories.
\begin{definition}[(Op)indexed category]
An {\em $\cat X$-opindexed category} is a pseudofunctor $I:\cat X\rightarrow\Cat$. Dually, an {\em $\cat X$-indexed category} is a pseudofunctor $I:\cat X^{op}\rightarrow\Cat$.
\end{definition}
We write the natural isomorphisms that witness the pseudofunctor structure as $c_{f,g}:I(f;g)\xrightarrow{\sim}I(f); I(g)$ and $u_x:I(\id_x)\xrightarrow{\sim}\id_{Ix}$.

We say that an (op)indexed category is {\em strict} when the pseudofunctor is an actual functor.

Given a small category $\cat X$, we denote the categories of $\cat X$-indexed and $\cat X$-opindexed categories by $\ICat(\cat X)$ and $\OpICat(\cat X)$, whose morphisms are given by pseudonatural transformations. The strict versions are denoted by adding the subscript $\mathsf{st}$, where the morphisms are ordinary natural transformations.

\begin{definition}[Grothendieck construction]
Given a pseudofunctor $I:\cat X\rightarrow\Cat$, the {\em Grothendieck construction} is the category $\Gr(\cat X)$ defined as follows:
\begin{itemize}
\item objects are pairs $(x,a)$, where $x\in\Ob(X)$ and $a\in\Ob(Iy)$,
\item morphisms $(f,F):(x,a)\rightarrow (y,b)$ are pairs of a morphism $f:x\rightarrow y$ in $\cat X$ and a morphism $F:(If)(a)\rightarrow b$ in $I(y)$,
\item the identity on $(x,a)$ is given by $(\id_x,(u_x)_a)$,
\item the composite of $(f,F):(x,a)\rightarrow (y,b)$ and $(g,G):(y,b)\rightarrow (z,c)$ is given by
$$\left(f;g, (c_{f,g})_a ; (Ig)(F); G\right).$$
\end{itemize}
\end{definition}

The following equivalence is well-known. We give a proof sketch, and suggest the reader who has not seen the details before to complete the proof, as it is instructive, and all of the subsequent developments will build on this equivalence.
\begin{theorem}\label{thm:opfibrations-opindexed-categories-equivalence}
Let $\cat X$ be a small category. There are equivalences of categories
\begin{align*}
\Fib(\cat X) &\simeq \ICat(\cat X) \\
\OpFib(\cat X) &\simeq \OpICat(\cat X).
\end{align*}
Moreover, the equivalences restrict to split (op)fibrations and strict (op)indexed categories:
\begin{align*}
\Fib_{\mathsf{sp}}(\cat X) &\simeq \ICat_{\mathsf{st}}(\cat X) \\
\OpFib_{\mathsf{sp}}(\cat X) &\simeq \OpICat_{\mathsf{st}}(\cat X).
\end{align*}
\end{theorem}
\begin{proof}[Proof sketch]
The left-to-right functor sends each (op)fibration to the pseudofunctor defined in Proposition~\ref{prop:opfibration-to-pseudofunctor}. The right-to-left functor sends each $\cat X$-(op)indexed category to the projection $\Gr(\cat X)\rightarrow\cat X$ to the first component from the Grothendieck construction.
\end{proof}
In fact, the above equivalences extend to 2-equivalences of 2-categories by adding the appropriate 2-cells to all the categories: we will, however, focus on the 1-categorical case here.

\subsection{(Op)fibrations with pseudomonoids and indexed monoidal categories}\label{subsec:opfibrations-indexed-monoidal}

Moeller and Vasilakopoulou~\cite{monoidal-gro} have extended the equivalence between (op)fibrations and (op)indexed categories in Theorem~\ref{thm:opfibrations-opindexed-categories-equivalence} to include monoidal structure. In fact, they have extended it in two ways: in the first case, each fibre is equipped with monoidal structure, which is preserved by the morphisms between the fibres; in the second case, there is a monoidal structure on the total and the base categories, which is preserved by the (op)fibration. Following Shulman~\cite{shulman2008}, we refer to the first (`fiberwise') case as {\em internal} monoidal structure, while to the second (`global') case as {\em external} monoidal structure. The two cases are, in general, distinct: we refer the reader to~\cite{monoidal-gro} for the details. Here we briefly sketch the situation of the first case, as our theory will build on a special case thereof.

On the side of (op)fibrations, the monoidal structure is captured by a pseudomonoid in the cartesian categories $(\Fib(\cat X),\boxtimes,\id_{\cat X})$ and $(\OpFib(X),\boxtimes,\id_{\cat X})$. A pseudomonoid in one of these categories is given by a triple $(\cat Y,\otimes,\one)$, where $\cat Y\rightarrow\cat X$ is an object (i.e.~an (op)fibration), while $\otimes:\cat Y\boxtimes\cat Y\rightarrow\cat Y$ and $\one :\cat X\rightarrow\cat Y$ are morphisms, i.e.~the diagrams
\begin{equation*}
\scalebox{1}{\input{journal-figures/gr-pseudomonoid.tikz}}
\end{equation*}
commute, and each functor preserves (op)cartesian maps, such that the usual unitality and associativity equations hold up to natural isomorphisms. A morphism of pseudomonoids is given by a morphism of (op)fibrations which commutes with the multiplication and the unit, again up to a natural isomorphism. We denote the resulting categories of pseudomonoids by $\PsMonFib(\cat X)$ and $\PsMonOpFib(\cat X)$. If the pseudomonoids are required to be actual monoids strictly preserved by the morphisms, the resulting categories of monoids are denoted by $\MonFib(\cat X)$ and $\MonOpFib(\cat X)$.

On the side of (op)indexed categories, the monoidal structure is captured by requiring the image of the pseudofunctor into $\Cat$ to be, in fact, contained in $\MonCat$.
\begin{definition}[(Op)indexed monoidal category]\label{def:opindexed-monoidal-category}
An {\em $\cat X$-opindexed monoidal category} is a pseudofunctor $I:\cat X\rightarrow\MonCat$. Dually, an {\em $\cat X$-indexed monoidal category} is a pseudofunctor $I:\cat X^{op}\rightarrow\MonCat$.
\end{definition}
We say that an (op)indexed monoidal category is {\em strict} if the pseudofunctor is an actual functor and its image is contained in $\MonCat_{\mathsf{st}}$, that is, each monoidal category and functor are strict.

Similarly to the plain indexed case, we denote the resulting categories by $\IMonCat(\cat X)$ and $\OpIMonCat(\cat X)$, and add the subscript $\mathsf{st}$ in the strict case.

\begin{warning}
The terms ``indexed monoidal category'' and ``monoidal indexed category'' are not used consistently in the literature. Some authors use it to refer to the external case when the indexing category $\cat X$ itself has a monoidal structure, and the pseudofunctor is required to preserve it in some sense. The terminology of Definition~\ref{def:opindexed-monoidal-category} is consistent with Hofstra \& De~Marchi~\cite{hofstra-demarchi2006} and Ponto \& Shulman~\cite{ponto-shulman12}.
\end{warning}

The following extends Theorem~\ref{thm:opfibrations-opindexed-categories-equivalence} to (op)fibrations with pseudomonoids and (op)indexed monoidal categories.
\begin{theorem}[Moeller and Vasilakopoulou~\cite{monoidal-gro}, Theorem 3.14]\label{thm:moeller-and-vasilakopoulou}
Let $\cat X$ be a small category. There are equivalences of categories
\begin{align*}
\PsMonFib(\cat X) &\simeq \IMonCat(\cat X) \\
\PsMonOpFib(\cat X) &\simeq \OpIMonCat(\cat X).
\end{align*}
\end{theorem}
As for Theorem~\ref{thm:opfibrations-opindexed-categories-equivalence}, this is, in fact, part of a 2-equivalence between 2-categories upon adding the appropriate 2-cells to all the categories.

While it would be interesting to study graphical languages for the general (non-strict) case as presented above, this would take us too far from the currently existing string diagrams for monoidal categories, almost all of which are only understood well in the strict case. Here, we therefore focus on the strict case stated below, and leave studying the non-strict case using the methods developed here for future work.
\begin{theorem}\label{thm:monoids-opfibrations-strict-imoncat-equivalence}
Let $\cat X$ be a small category. There are equivalences of categories
\begin{align*}
\MonFib_{\mathsf{sp}}(\cat X) &\simeq \IMonCat_{\mathsf{st}}(\cat X) \\
\MonOpFib_{\mathsf{sp}}(\cat X) &\simeq \OpIMonCat_{\mathsf{st}}(\cat X).
\end{align*}
\end{theorem}
\begin{proof}
Let $(\cat Y,\otimes,\one)$ be a monoid on a split opfibration $\cat Y\rightarrow\cat X$. For each $x\in\Ob(\cat X)$, this defines a functor
\begin{align*}
\otimes_x : \cat Y_x\times\cat Y_x &\rightarrow\cat Y_x \\
(a,b) &\mapsto a\otimes b,
\end{align*}
and the unit is given by $\one(x)\in\Ob(\cat Y_x)$. Associativity and unitality follow from those of $\otimes$ and $\one$.

Conversely, given a functor $\cat X\rightarrow\MonCat_{\mathsf{st}}$ (i.e.~a strict opindexed monoidal category), we define the monoid on the Grothendieck construction as follows:
\begin{equation*}
\scalebox{1}{\input{journal-figures/gr-monoid-induces-by-opindexed.tikz}},
\end{equation*}
where $(\otimes_x,I_x)$ denotes the monoidal structure in the category indexed by $x\in\Ob(\cat X)$. Associativity and unitality now follow from the corresponding properties of strict monoidal categories. Moreover, the triangles above commute by definition, and opcartesian maps (i.e.~pairs whose second component is the identity) are preserved by strictness of the monoidal structure.
\end{proof}

\subsection{Profunctors}\label{subsec:profunctors}

We follow Loregian~\cite{loregian} in our discussion of profunctors. We fix the convention that profunctors are contravariant in the first variable, and covariant in the second one. Profunctors are also known under the names of {\em distributors}, {\em relators}, {\em correspondences} and {\em bimodules}.

\begin{definition}[Profunctor]
Let $\cat A$ and $\cat B$ be small categories. A {\em profunctor} from $\cat A$ to $\cat B$ is a functor $\cat A^{op}\times\cat B\rightarrow\Set$.
\end{definition}
We denote a profunctor $P$ from $\cat A$ to $\cat B$ by $P:\cat A\srarrow\cat B$. Thus, a profunctor assigns to each objects $a\in\Ob(\cat A)$ and $b\in\Ob(\cat B)$ a set $P(a,b)$, which we think of as ``morphisms'' $a\rightsquigarrow b$ between objects in different categories\footnote{Such arrows are sometimes called {\em proarrows} or {\em heteromorphisms}.}. The action of $P$ on morphisms is contravariant in the first variable (morphisms of $\cat A$) and covariant in the second variable (morphisms of $\cat B$):
\begin{equation*}
\scalebox{1}{\input{journal-figures/profunctor-def.tikz}}.
\end{equation*}
The assignment is functorial: $P(\id_a,\id_b)$ is the identity function on the set $P(a,b)$, and $P(f';f,g;g') = P(f,g);P(f',g')$. These assumption ensure that the arrows $a\rightsquigarrow b$ behave essentially like morphisms within a category, which leads one to ask how to compose two such arrows. This leads to defining composition of two profunctors.

Given profunctors $P:\cat A\srarrow\cat B$ and $Q:\cat B\srarrow\cat C$, their composite profunctor $P;Q:\cat A\srarrow\cat C$ is defined on objects by the following coend formula
\begin{equation}\label{eqn:profunctor-composition}
P;Q(a,c)\coloneq\int^{b\in\cat B}P(a,b)\times Q(b,c)\coloneq\quot{\coprod\limits_{b\in\cat B}P(a,b)\times Q(b,c)}{\sim},
\end{equation}
where $\sim$ is the equivalence relation on the set $\coprod_{b\in\cat B}P(a,b)\times Q(b,c)$ generated by requiring that for all $f\in P(a,b)$, $g\in\cat B(b,b')$ and $h\in Q(b',c)$ one has
$$\left(f,Q(g,\id_c)(h)\right)\sim\left(P(\id_a,g)(f),h\right).$$
On morphisms, we define
\begin{equation}\label{eqn:profunctor-composition-morphisms}
\scalebox{1}{\input{journal-figures/profunctor-composition-def.tikz}}.
\end{equation}

Note that the equivalence relation defined above can be thought of as a kind of associativity condition, ensuring that the situation
\begin{equation*}
\scalebox{1}{\input{journal-figures/profunctor-composition.tikz}}
\end{equation*}
unambiguously defines an element in $P;Q(a,c)$. In the next section, we shall see a construction making this formal: the above arrows will become actual morphisms within the same category, and the equivalence relation will guarantee associativity. We do not discuss the general theory of coends here, and simply treat the expression on the right-hand side of~\eqref{eqn:profunctor-composition} as the definition of the coend on the left. For a detailed exposition, we refer the reader to Loregian~\cite{loregian}.

\begin{example}\label{ex:identity-profunctor}
Let $\cat A$ be a category. Then the assignment
\begin{align*}
\cat A^{op}\times\cat A &\rightarrow\Set \\
(a,b) &\mapsto\cat A(a,b) \\
\left(f : a'\rightarrow a,g : b\rightarrow b'\right) &\mapsto f;{-};g : A(a,b)\rightarrow A(a',b')\ ::\ h\mapsto f;h;g
\end{align*}
defines a profunctor $\cat A\srarrow\cat A$.
\end{example}

\begin{definition}\label{def:bicat-profunctors}
The {\em bicategory of profunctors} $\Prof$ has
\begin{itemize}
\item all small categories as the 0-cells,
\item all profunctors as the 1-cells,
\item given profunctors $P,Q:\cat A\srarrow\cat B$, the 2-cells $P\rightarrow Q$ are given by the natural transformations,
\item the composition $\Prof(\cat A,\cat B)\times\Prof(\cat B,\cat C)\rightarrow\Prof(\cat A,\cat C)$ is given by~\eqref{eqn:profunctor-composition} and~\eqref{eqn:profunctor-composition-morphisms},
\item the identity on a 0-cell $\cat A$ is given by the profunctor in Example~\ref{ex:identity-profunctor},
\item composition and identities for 2-cells are those for the natural transformations.
\end{itemize}
\end{definition}

There are two ways to embed $\Cat$ into $\Prof$: both embedding functors are identity on objects, one is contravariant on the 2-cells, while the other is contravariant on the 1-cells. Hence, let us denote by $\Cat^{co}$ the category with the same 0- and 1-cells as in $\Cat$, and whose 2-cells are those of $\Cat$ with the direction reversed. Similarly, let us denote by $\Cat^{op}$ the category with the same 0- and 2-cells as in $\Cat$, but whose 1-cells are reversed.

First, let us define
$$-^{\refine}:\Cat^{co}\rightarrow\Prof$$
by mapping the functor $F:\cat A\rightarrow\cat B$ to the following profunctor:
\begin{align*}
F^{\refine} : \cat A^{op}\times\cat B &\rightarrow\Set \\
(a,b) &\mapsto\cat B(Fa,b) \\
\left(f : a'\rightarrow a,g : b\rightarrow b'\right) &\mapsto Ff;{-};g : B(Fa,b)\rightarrow B(Fa',b')\ ::\ h\mapsto Ff;h;g,
\end{align*}
and by mapping a natural transformation $\eta:F\rightarrow G$ to the natural transformation $\eta^{\refine}:G^{\refine}\rightarrow F^{\refine}$ whose $(a,b)$-component is given by $\eta_a;{-}$.

Likewise, we define
$$-^{\coarsen}:\Cat^{op}\rightarrow\Prof$$
by mapping the functor $F:\cat A\rightarrow\cat B$ to the following profunctor:
\begin{align*}
F^{\coarsen} : \cat B^{op}\times\cat A &\rightarrow\Set \\
(b,a) &\mapsto\cat B(b,Fa) \\
\left(f : b'\rightarrow b,g : a\rightarrow a'\right) &\mapsto f;{-};Fg : B(b,Fa)\rightarrow B(b',Fa')\ ::\ h\mapsto f;h;Fg,
\end{align*}
and by mapping a natural transformation $\eta:F\rightarrow G$ to the natural transformation $\eta^{\coarsen}:F^{\coarsen}\rightarrow G^{\coarsen}$ whose $(b,a)$-component is given by ${-};\eta_a$.

\begin{proposition}\label{prop:prof-embedding-adjoints}
Let $F:\cat A\rightarrow\cat B$ be a functor (i.e.~a 1-cell in $\Cat$). Then the 1-cells $F^{\refine}$ and $F^{\coarsen}$ are adjoint with
$$F^{\refine}\dashv F^{\coarsen}$$
in the bicategory $\Prof$.
\end{proposition}
\begin{proof}
The unit 2-cell $\eta : \id_{\cat A}\rightarrow F^{\refine};F^{\coarsen}$ is defined as
\begin{align*}
\eta_{a,a'} : \cat A(a,a') &\rightarrow\int^{b\in\cat B}\cat B(Fa,b)\times\cat B(b,Fa') \\
f &\mapsto (Ff,\id_{Fa'}),
\end{align*}
while the counit 2-cell $\varepsilon : F^{\coarsen};F^{\refine}\rightarrow\id_{\cat B}$ is defined as
\begin{align*}
\varepsilon_{b,b'} : \int^{a\in\cat A}\cat B(b,Fa)\times\cat B(Fa,b') &\rightarrow\cat B(b,b') \\
(g,h) &\mapsto g;h.
\end{align*}
Naturality and the triangle equations now follow by observing that for the composite $F^{\refine};F^{\coarsen}$ we have $(f,g)\sim (\id_{Fa},f;g)$ for all $f\in\cat B(Fa,b)$ and $g\in\cat B(b,Fa')$, and similarly, for the composite $F^{\coarsen};F^{\refine}$ we have $(Ff,g)\sim (\id_{Fa},Ff;g)$ for all $f\in\cat A(a,a')$ and $g\in\cat B(Fa',b)$.
\end{proof}

\begin{remark}\label{rem:counit-surjective-twocells}
We observe that the counit 2-cell defined in the proof of Proposition~\ref{prop:prof-embedding-adjoints} has surjective components on the image of the functor $F$. More precisely, for all $a,a'\in\cat A$, the function
$$\varepsilon_{Fa,Fa'} : \int^{a''\in\cat A}\cat B(Fa,Fa'')\times\cat B(Fa'',Fa') \rightarrow\cat B(Fa,Fa')$$
is surjective: a section is given by mapping $h:Fa\rightarrow Fa'$ to $(\id_{Fa},h)$. The surjectivity of the components will allow us to introduce the sections for the counit 2-cells for deflational theories in Subsection~\ref{subsec:opfib-defl-theories}.
\end{remark}

\subsection{Collages}\label{subsec:collages}

The notion of a {\em collage} generalises the Grothendieck construction (see Section~\ref{subsec:indexed-categories}) from (op)fibrations to arbitrary functors. What one gets on the other side of the equivalence are {\em displayed categories}, and a functor is obtained from a displayed category by taking its collage.

The term {\em displayed category} was introduced by Ahrens and Lumsdaine~\cite{displayed-categories} for the purposes of studying fibrations in the type theoretic setting, where equivalence of their presentation with the notion used in Definition~\ref{def:displayed-category} is pointed out. The definition used here, as well as Theorem~\ref{thm:collage-functors} and Proposition~\ref{prop:displayed-pseudo-decompose} are discussed in the notes taken by Streicher based on B\' enabou's course~\cite{benabou}.

Let $\cat X$ and $\cat Y$ be bicategories. Recall that a {\em lax functor} $(D,\varphi):\cat X\rightarrow\cat Y$ consists of:
\begin{itemize}
\item a function $D:\Ob(\cat X)\rightarrow\Ob(\cat Y)$,
\item for all $x,y\in\cat X$, a functor $D:\cat X(x,y)\rightarrow\cat Y(Dx,Dy)$,
\item for all $x\in\cat X$, a 2-cell $\varphi^x:\id_{Dx}\rightarrow D(\id_x)$ in $\cat Y$,
\item for all composable morphisms $f:x\rightarrow y$ and $g:y\rightarrow z$ in $\cat X$, a 2-cell $\varphi^{f,g}:D(f);D(g)\rightarrow D(f;g)$ in $\cat Y$ which is natural in $f$ and $g$,
\end{itemize}
such that lax associativity and unitality diagrams commute (see e.g.~\cite{leinster98,johnson-yau21}). We refer to the natural collections of 2-cells $\varphi$ as the {\em laxators}. We say that a lax functor is {\em normal} if for each $x\in\Ob(\cat X)$, the laxator $\varphi^x$ is the identity, i.e.~$\id_{Dx}=D(\id_x)$.
\begin{definition}[Displayed category]\label{def:displayed-category}
Let $\cat X$ be a small 1-category. A {\em displayed category} is a normal lax functor $(D,\varphi):\cat X\rightarrow\Prof$.
\end{definition}

Every functor $p:\cat Y\rightarrow\cat X$ induces a displayed category $D_p:\cat X\rightarrow\Prof$ as follows:
\begin{itemize}
\item for every $x\in\Ob(\cat X)$, let $D_p(x)\coloneq\cat Y_x$,
\item for every morphism $f:x\rightarrow y$ in $\cat X$, define
\begin{equation*}
\scalebox{1}{\input{journal-figures/functor-to-displayed.tikz}},
\end{equation*}
\item given $x\xrightarrow{f}y\xrightarrow{g}z$, the laxator $\varphi_p^{f,g} : D_p(f);D_p(g)\rightarrow D_p(f;g)$ is defined componentwise for each $a\in\cat Y_x$ and $c\in\cat Y_z$ by
\begin{align*}
\left(\varphi_p^{f,g}\right)_{a,c} : \int^{b\in\cat Y_y} D_p(f)(a,b)\times D_p(g)(b,c) &\rightarrow D_p(f;g)(a,c) \\
(F,G) &\mapsto (F;G).
\end{align*}
\end{itemize}
Note that the lax functor so defined is indeed normal: $D_p(\id_x)=\cat Y_x(-,-)=\id_{D_p(x)}$.

\begin{definition}[Collage]\label{def:collage}
Given a displayed category $(D,\varphi):\cat X\rightarrow\Prof$, its {\em collage} is the category $\coprod D$ defined as follows:
\begin{itemize}
\item the objects are pairs $(x,a)$, where $x\in\cat X$ and $a\in D(x)$,
\item a morphism $(f,F) : (x,a)\rightarrow (y,b)$ consists of a morphism $f:x\rightarrow y$ in $\cat X$ and an element $F\in D(f)(a,b)$,
\item the identity on $(x,a)$ is given by $(\id_x,\id_a)$,
\item the composition of $(f,F) : (x,a)\rightarrow (y,b)$ and $(g,G) : (y,b)\rightarrow (z,c)$ is given by $\left(f;g,\varphi^{f,g}_{a,c}(F,G)\right)$.
\end{itemize}
\end{definition}
Note that there is a forgetful functor $\coprod D\rightarrow\cat X$ projecting the objects and the morphisms to their first component. The key observation by B\' enabou is that {\em any} functor into $\cat X$ arises as the collage of some displayed category. In order to state this result, let us define the category $\Disp(\cat X)$ as having all displayed categories $\cat X\rightarrow\Prof$ as objects, and the morphisms are all lax transformations whose components are in the image of the embedding $-^{\refine}:\Cat^{co}\rightarrow\Prof$.
\begin{theorem}[B\' enabou]\label{thm:collage-functors}
Let $\cat X$ be a small category. The slice category over $\cat X$ is equivalent to the category of displayed categories on $\cat X$:
$$\quot{\Cat}{\cat X}\simeq\Disp(\cat X).$$
\end{theorem}
\begin{proof}[Proof sketch]
Given a functor $p$ into $\cat X$, the corresponding displayed category is given by $D_p$. Conversely, given a displayed category $D$ on $\cat X$, the corresponding functor is given by the projection from the collage $\coprod D$ into $\cat X$
\end{proof}
We refer the reader to~\cite{benabou}, \cite{street-powerful-functors} and~\cite[Theorem 5.4.5]{loregian} for the detailed discussion and proofs.

Collage of a displayed category, together with Theorem~\ref{thm:collage-functors} allow to detect properties of displayed categories on $\cat X$ via properties of functors into $\cat X$. We shall need two such properties, namely, we characterise when $D:\cat X\rightarrow\Prof$ is a pseudofunctor (rather than merely a lax functor) and when it factors through $\Cat$. In the latter case, each morphism in $\cat X$ is, in fact, indexing a functor, so that we have an indexed category $\cat X\rightarrow\Cat$, showing that displayed categories generalise indexed categories, while the collage is indeed a generalised Grothendieck construction.

Definition~\ref{def:decomposition-lifting} characterises those functors which correspond to pseudofunctorial displayed categories. In order to state the definition, we need the notion of {\em path components} above a pair of composable maps, which we now proceed to describe.

Given a functor $\cat Y\rightarrow\cat X$, a pair of composable maps $f:x\rightarrow y$ and $g:y\rightarrow z$ in $\cat X$ and objects $a\in\cat Y_x$ and $c\in\cat Y_z$ in the fibres over $x$ and $z$, let us define the category $\cat Y(f,g)(a,c)$ as follows:
\begin{itemize}
\item the objects $(F,b,G)$ are composable pairs $F:a\rightarrow b$ and $G:b\rightarrow c$ such that $b\in\cat Y_y$, $F$ is above $f$ and $G$ is above $g$,
\item a morphism $H:(F,b,G)\rightarrow (F',b',G')$ is a map $H:b\rightarrow b'$ in the fibre $\cat Y_y$ such that the diagram below commutes:
\begin{equation*}
\scalebox{1}{\input{journal-figures/composable-pairs-morphism.tikz}}.
\end{equation*}
\end{itemize}
We say that two objects in $\cat Y(f,g)(a,c)$ are in the same {\em path component} if they are related by the equivalence relation generated by stipulating that $(F,b,G)\sim (F',b',G')$ if there is a morphism $(F,b,G)\rightarrow (F',b',G')$. In other words, two composable pairs are in the same path component if there is a zigzag of morphisms between them.

\begin{definition}[Factorisation lifting]\label{def:decomposition-lifting}
We say that a functor $\cat Y\rightarrow\cat X$ is a {\em factorisation lifting} if for any morphism $F:a\rightarrow c$ in $\cat Y$ above $f:x\rightarrow z$, whenever there are morphisms $g:x\rightarrow y$ and $g':y\rightarrow z$ such that $g;g'=F$, there are induced morphisms $G:a\rightarrow b$ and $G':b\rightarrow c$ above $G$ and $G'$, respectively, such that $G;G'=F$:
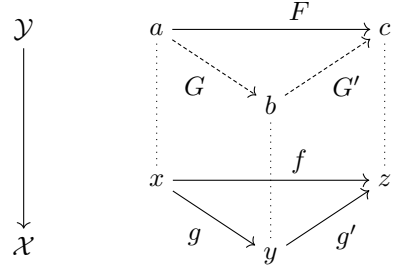
\begin{equation*}
\scalebox{1}{\input{journal-figures/decomposition-pair-induced.tikz}},
\end{equation*}
and, moreover, any such induced pairs of morphisms $(G,b,G')$ and $(H,b',H')$ are in the same path component of $\cat Y(g,g')(a,c)$.
\end{definition}
\begin{remark}
What we call a {\em factorisation lifting} is usually called a {\em Conduch\' e fibration}, a {\em Conduch\' e functor} or an {\em exponentiable functor}. However, we consider the term {\em factorisation lifting} more suggestive of the property such functor is required to satisfy. The term {\em exponentiable} should be reserved to the property of the last bullet point of Proposition~\ref{prop:displayed-pseudo-decompose}. Said proposition then establishes that a functor is a factorisation lifting if and only if it is exponentiable.
\end{remark}
By analogy with (op)fibrations, a factorisation lifting is {\em cloven} if it comes with chosen lifts for each factorisation. Further, a cloven factorisation lifting is {\em split} if the composition of any chosen lifts is the chosen lift of the composition, and the chosen lift of the trivial factorisation (either $f;\id_z$ or $\id_x;f$) is the corresponding trivial factorisation ($F;\id_c$ or $\id_a;F$).

The following proposition is due to Giraud~\cite{giraud64} and Conduch\'e~\cite{conduche72}. A detailed discussion can be found in Street~\cite{street-powerful-functors}.
\begin{proposition}\label{prop:displayed-pseudo-decompose}
For any functor $p:\cat Y\rightarrow\cat X$, the following are equivalent:
\begin{itemize}
\item $p$ is a factorisation lifting,
\item the displayed category $D_p:\cat X\rightarrow\Prof$ is a pseudofunctor,
\item the functor $p^*:\quot{\Cat}{\cat X}\rightarrow\quot{\Cat}{\cat Y}$ defined by taking pullbacks has a right adjoint.
\end{itemize}
\end{proposition}
We only prove the equivalence of first two bullet points, as this suffices for developing the subsequent theory. The third one is stated for the sake of completeness of the presentation: the reader is referred to Street~\cite{street-powerful-functors} for the proof.
\begin{proof}
Observe that the laxator of $D_p$
$$\left(\varphi_p^{f,g}\right)_{a,c} : \int^{b\in\cat Y_y} D_p(f)(a,b)\times D_p(g)(b,c) \rightarrow D_p(f;g)(a,c)$$
is an isomorphism if and only if for all $H:a\rightarrow c$ above $f;g$ there are $F:a\rightarrow b$ and $G:b\rightarrow c$ above $f$ and $g$ such that $F;G=H$, and such pair is unique up to the equivalence defining the coend, i.e.~precisely when $p$ is a factorisation lifting.
\end{proof}
In light of Proposition~\ref{prop:displayed-pseudo-decompose}, a factorisation lifting being cloven corresponds to a choice of an inverse for each component of the laxator of $D_p$. If such choice of inverses is, moreover, functorial in both $f$ and $g$, the factorisation lifting is split.

The following proposition and the ensuing corollary connect our discussion of displayed categories and collages to opfibrations.
\begin{proposition}\label{prop:displayed-factors-opfibration}
A functor $p:\cat Y\rightarrow\cat X$ is an preopfibration if and only if the displayed category $D_p:\cat X\rightarrow\Prof$ factors through the embedding $-^{\refine}:\Cat^{co}\rightarrow\Prof$.
\end{proposition}
\begin{proof}
Observe that we have a natural isomorphism
\begin{equation*}
\scalebox{1}{\input{journal-figures/preopfib-factors.tikz}}
\end{equation*}
for some functor $F:\cat Y_x\rightarrow\cat Y_y$ if and only if the opliftable pair $(a,f:x\rightarrow y)$ has a weak opcartesian lifting $f_a:a\rightarrow F(a)$ for each $a\in\cat Y_x$. To prove this, we argue as follows. If such liftings exist, then $F\coloneq f^*$ is the usual reindexing functor between the fibres, and the isomorphism is given by the universal property of weak opcartesian maps: any $F':a\rightarrow b$ in $D_p(f)(a,b)$ uniquely factorises as
\begin{equation*}
\scalebox{1}{\input{journal-figures/preopfib-factors-factorises.tikz}}.
\end{equation*}
Conversely, given such a functor and a natural isomorphism, define the lifting of the opliftable pair $(a,f:x\rightarrow y)$ as $\alpha^{-1}_{a,Fa}\left(\id_{Fa}\right)$.
\end{proof}
\begin{corollary}\label{cor:preopfib-opfib-factlift}
For a preopfibration $\cat Y\rightarrow\cat X$, the following are equivalent:
\begin{itemize}
\item it is an opfibration,
\item it is a factorisation lifting,
\item the composition of weakly opcartesian maps is weakly opcartesian.
\end{itemize}
\end{corollary}
Proposition~\ref{prop:displayed-factors-opfibration} and Corollary~\ref{cor:preopfib-opfib-factlift}, of course, dualise to (pre)fibrations. Namely, a functor $p:\cat Y\rightarrow\cat X$ is an prefibration if and only if the displayed category $D_p:\cat X\rightarrow\Prof$ factors through the embedding $-^{\coarsen}:\Cat^{op}\rightarrow\Prof$, and Corollary~\ref{cor:preopfib-opfib-factlift} holds verbatim upon removing `op' throughout.

\subsection{Monoidal theories and models}\label{subsec:models-monoidal-theories}

The terms for monoidal theories, upon quotienting by the structural identities~\ref{def:str-id}, give rise to {\em string diagrams}~\cite{joyal-street88,joyal-street91,selinger,piedeleu-zanasi}, which are sound and complete with respect to monoidal categories.

\begin{definition}[Monoidal signature]
A {\em monoidal signature} is a pair $(C,\Sigma)$ of a set $C$ and a function $\Sigma : C^*\times C^*\rightarrow\Set$.
\end{definition}
Given a monoidal signature $(C,\Sigma)$, the elements of $C$ are called {\em colours}, and for $a,b\in C^*$, the elements in $\Sigma(a,b)$ are the {\em monoidal generators} with arity $a$ and coarity $b$.

A {\em morphism} between two monoidal signatures $\left(f,f_{a,b}\right) : (C,\Sigma)\rightarrow (D,\Gamma)$ consists of a function $f:C\rightarrow D$, and for each pair $(a,b)\in C^*\times C^*$, a function $f_{a,b}:\Sigma(a,b)\rightarrow\Gamma\left(f^*a,f^*b\right)$. We denote the resulting category of monoidal signatures by $\MSgn$.

\begin{definition}[Terms of a monoidal signature]\label{def:terms-monoidal}
Given a monoidal signature $(C,\Sigma)$, the set of {\em sorts} is given by $C^*\times C^*$. The {\em terms} are generated by the recursive sorting procedure below:
\begin{center}
  \centering\small\noindent
  \begin{bprooftree}
    \AxiomC{$\sigma\in\Sigma(a,b)$}
    \RightLabel{\;\;}
    \UnaryInfC{$\scalebox{.8}{\begin{tikzpicture}
	\begin{pgfonlayer}{nodelayer}
		\node [style=none] (4) at (-0.75, 0) {};
		\node [style=none] (5) at (0.75, 0) {};
		\node [style=bdytextbox] (6) at (0, 0) {$\sigma$};
	\end{pgfonlayer}
	\begin{pgfonlayer}{edgelayer}
		\draw [style=edge] (4.center) to (6);
		\draw [style=edge] (6) to (5.center);
	\end{pgfonlayer}
\end{tikzpicture}
} : (a,b)$}
    \DisplayProof
    \AxiomC{$a\in C$}
    \RightLabel{\;\;}
    \UnaryInfC{$\scalebox{.8}{\input{journal-figures/iddiag-dashed.tikz}} : (a,a)$}
    \DisplayProof
    \AxiomC{\phantom{$a\in C$}}
    \RightLabel{\;\;}
    \UnaryInfC{$\scalebox{.8}{\input{journal-figures/emptydiag.tikz}} : (\varepsilon,\varepsilon)$}
  \end{bprooftree}
  \begin{bprooftree}
    \AxiomC{$\scalebox{.8}{\begin{tikzpicture}
	\begin{pgfonlayer}{nodelayer}
		\node [style=none] (0) at (-0.75, 0) {};
		\node [style=none] (1) at (0.75, 0) {};
		\node [style=bdytextbox] (2) at (0, 0) {$t$};
	\end{pgfonlayer}
	\begin{pgfonlayer}{edgelayer}
		\draw [style=edge] (0.center) to (2);
		\draw [style=edge] (2) to (1.center);
	\end{pgfonlayer}
\end{tikzpicture}
} : (a,b)$}
    \AxiomC{$\scalebox{.8}{\begin{tikzpicture}
	\begin{pgfonlayer}{nodelayer}
		\node [style=none] (0) at (-0.75, 0) {};
		\node [style=none] (1) at (0.75, 0) {};
		\node [style=bdytextbox] (2) at (0, 0) {$s$};
	\end{pgfonlayer}
	\begin{pgfonlayer}{edgelayer}
		\draw [style=edge] (0.center) to (2);
		\draw [style=edge] (2) to (1.center);
	\end{pgfonlayer}
\end{tikzpicture}
} : (b,c)$}
    \RightLabel{\;\;}
    \BinaryInfC{$\scalebox{.8}{\input{journal-figures/t-s-compose.tikz}} : (a,c)$}
    \DisplayProof
    \AxiomC{$\scalebox{.8}{\begin{tikzpicture}
	\begin{pgfonlayer}{nodelayer}
		\node [style=none] (0) at (-0.75, 0) {};
		\node [style=none] (1) at (0.75, 0) {};
		\node [style=bdytextbox] (2) at (0, 0) {$t$};
	\end{pgfonlayer}
	\begin{pgfonlayer}{edgelayer}
		\draw [style=edge] (0.center) to (2);
		\draw [style=edge] (2) to (1.center);
	\end{pgfonlayer}
\end{tikzpicture}
} : (a,b)$}
    \AxiomC{$\scalebox{.8}{\begin{tikzpicture}
	\begin{pgfonlayer}{nodelayer}
		\node [style=none] (0) at (-0.75, 0) {};
		\node [style=none] (1) at (0.75, 0) {};
		\node [style=bdytextbox] (2) at (0, 0) {$s$};
	\end{pgfonlayer}
	\begin{pgfonlayer}{edgelayer}
		\draw [style=edge] (0.center) to (2);
		\draw [style=edge] (2) to (1.center);
	\end{pgfonlayer}
\end{tikzpicture}
} : (c,d)$}
    \RightLabel{.}
    \BinaryInfC{$\scalebox{.8}{\input{journal-figures/t-s-tensor.tikz}} : (ac,bd)$}
  \end{bprooftree}
\end{center}
When using the linear notation, we denote the terms generated by the above rules by $\sigma$, $\id_a$, $\id_{\varepsilon}$, $(t;s)$ and $(t\otimes s)$, respectively. We denote the set of terms by $\Term_{C,\Sigma}$, and by $S:\Term_{C,\Sigma}\rightarrow C^*\times C^*$ the ``underlying sort function" mapping $t:(a,b)\mapsto (a,b)$.
\end{definition}

A pair of terms $(t,s)\in\Term_{C,\Sigma}\times\Term_{C,\Sigma}$ is {\em parallel} if $S(t)=S(s)$. Let us denote the set of parallel terms by $P_{C,\Sigma}$.

Given a morphism $f:(C,\Sigma)\rightarrow (D,\Gamma)$ of monoidal signatures, it immediately extends to a function on terms, also denoted by $f:\Term_{C,\Sigma}\rightarrow\Term_{D,\Gamma}$, by recursively defining:\label{p:extend-morphism-signatures-terms}
\begin{center}
\begin{tabular}{c c c}
$\sigma : (a,b) \mapsto f_{a,b}(\sigma) : (f^*a,f^*b)$, & $\id_a : (a,a) \mapsto \id_{fa} : (fa,fa)$, & $\id_{\varepsilon} : (\varepsilon,\varepsilon) \mapsto \id_{\varepsilon} : (\varepsilon,\varepsilon)$, \\
$(t;s) : (a,c) \mapsto \left(f(t);f(s)\right) : (f^*a,f^*c)$, & \multicolumn{2}{c}{$(t\otimes s) : (ac,bd) \mapsto \left(f(t)\otimes f(s)\right) : (f^*(ac),f^*(bd))$.}
\end{tabular}
\end{center}
Observe that $f:\Term_{C,\Sigma}\rightarrow\Term_{D,\Gamma}$ preserves parallel terms: if $(t,s)\in P_{C,\Sigma}$, then $(ft,fs)\in P_{D,\Gamma}$.

\begin{definition}[Monoidal theory]\label{def:monoidal-theory}
A {\em monoidal theory} $\mathcal T$ is a triple $(C,\Sigma,E)$, where $(C,\Sigma)$ is a monoidal signature, and $E\sse P_{C,\Sigma}$ is a set of parallel terms. We refer to $E$ as the {\em equations} of $\mathcal T$.
\end{definition}

A morphism of monoidal theories $f:(C,\Sigma,E)\rightarrow (D,\Gamma,F)$ is given by a morphism of monoidal signatures $f:(C,\Sigma)\rightarrow (D,\Gamma)$ such that for all $(s,t)\in E$, we have $(fs,ft)\in F$. We denote the category of monoidal theories by $\MTh$.

\begin{definition}[Structural identities]\label{def:str-id}
Given a monoidal signature $(C,\Sigma)$, the set of {\em structural identities} $S$ is given by the following equations, where $s$, $s_i$ and $t_i$ range over the terms of the appropriate type:
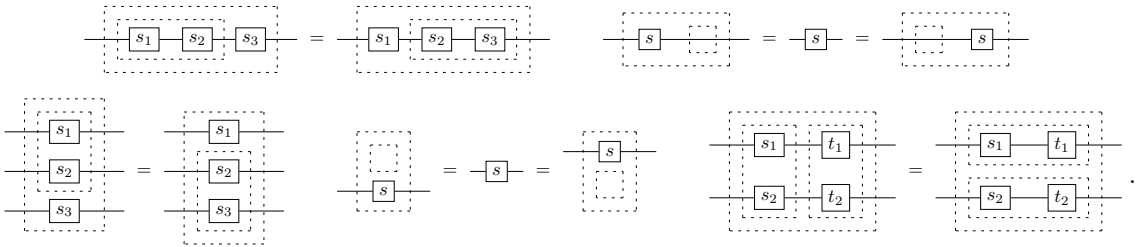
\begin{equation*}
\scalebox{.7}{\input{journal-figures/structural-identities.tikz}}.
\end{equation*}
\end{definition}

\begin{definition}[Term congruence]
Given a monoidal theory $(C,\Sigma,E)$, the {\em term congruence} $\Eeq$ is the smallest equivalence relation on $\Term_{C,\Sigma}$ generated by $E\cup S$, which is a congruence with respect to $;$ and $\otimes$, i.e.~if $t_1\Eeq t_2$ and $s_1\Eeq s_2$, then $t_1;s_1\Eeq t_2;s_2$
and $t_1\otimes s_1\Eeq t_2\otimes s_2$, whenever the composition is defined.
\end{definition}

Note that the structural equations justify dropping all the dashed boxes when drawing the terms up to a term congruence.

Examples~\ref{def:symm-mon-thy}-{def:thy-indexed-monoids} will play an important role when developing layered monoidal theories.
\begin{example}[Symmetric monoidal theory]\label{def:symm-mon-thy}
A {\em symmetric monoidal theory} is a monoidal theory $(C,\Sigma,\mathcal S)$ such that for all $a,b\in C$, the set $\Sigma(ab,ba)$ contains the special generator, denoted by $\scalebox{.5}{\input{journal-figures/symmetry.tikz}}$, called the {\em symmetry}, and the set $\mathcal S$ contains the following equations, where $s$ ranges over all terms of the appropriate type, and the symmetry on the non-generator wires is defined recursively:
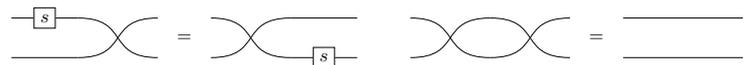
\begin{equation*}
\scalebox{.7}{\input{journal-figures/symmetry-eqns.tikz}}.
\end{equation*}
\end{example}
We remark that most of the concrete examples of monoidal theories we will see in this article will be symmetric.

\begin{example}[Theory of monoids]\label{def:thy-mon}
The {\em theory of monoids} is the monoidal theory $(\{\bullet\},\Sigma,\mathcal M)$ with generators \scalebox{.5}{\input{journal-figures/monoid.tikz}} and equations \scalebox{.5}{\input{journal-figures/monoid-equations.tikz}}.
\end{example}

\begin{example}[Theory of comonoids]\label{def:thy-comon}
The {\em theory of comonoids} is the monoidal theory $(\{\bullet\},\Sigma,\mathcal C)$ with generators \scalebox{.5}{\input{journal-figures/comonoid.tikz}} and equations \scalebox{.5}{\input{journal-figures/comonoid-equations.tikz}}.
\end{example}

\begin{example}[Theory with uniform comonoids]\label{def:thy-univ-comonoids}
We say that a symmetric monoidal theory $(C,\Sigma,\mathcal U)$ has {\em uniform comonoids} if every $a\in C$ has a comonoid structure (i.e.~the sets $\Sigma(a,aa)$ and $\Sigma(a,\varepsilon)$ contain the comonoid generators and $\mathcal U$ contains the comonoid equations from Definition~\ref{def:thy-comon}), and upon extending the comonoid structure to all $w\in C^*$ by the following recursion:
\begin{equation*}
\scalebox{.7}{\input{journal-figures/comon-extend.tikz}},
\end{equation*}
the set $\mathcal U$ contains the following equations, where $s$ ranges over all terms with appropriate type:
\begin{equation*}
\scalebox{.7}{\input{journal-figures/comon-nat-eqns.tikz}}.
\end{equation*}
\end{example}

\begin{example}[Theory with $1\mdash 1$-natural monoids]\label{def:thy-11-nat-monoids}
We say that a symmetric monoidal theory $(C,\Sigma,\mathcal N)$ has {\em $1\mdash 1$-natural monoids} if every $a\in C$ has a monoid structure (i.e.~the sets $\Sigma(aa,a)$ and $\Sigma(\varepsilon,a)$ contain the monoid generators and $\mathcal N$ contains the monoid equations from Definition~\ref{def:thy-mon}), and for every $\sigma\in\Sigma(a,b)$ with $a,b\in C$, we have the following equations:
\begin{equation*}
\scalebox{.7}{\input{journal-figures/locally-nat-monoids.tikz}}.
\end{equation*}
\end{example}

\begin{example}[Theory with indexed monoids]\label{def:thy-indexed-monoids}
We say that a symmetric monoidal theory has {\em indexed monoids} if it has both $1\mdash 1$-natural monoids (Definition~\ref{def:thy-11-nat-monoids}) and uniform comonoids (Definition~\ref{def:thy-univ-comonoids}).
\end{example}

A model interprets a monoidal signature (or theory) in a strict monoidal category.
\begin{definition}\label{def:model-monoidal-signature}
A {\em model} of a monoidal signature $(C,\Sigma)$ is a strict monoidal category $(\cat C,\otimes,I)$ together with a function $i:C\rightarrow\Ob(\cat C)$, and for each $a,b\in C^*$, a function $i_{a,b}:\Sigma(a,b)\rightarrow\cat C(i^*a,i^*b)$.
\end{definition}
We often denote a model of a monoidal signature $(C,\Sigma)$ simply by $(\cat C,i)$, where $\cat C$ is a strict monoidal category, and $i$ a family of functions as in the definition -- referred to as the {\em interpretation functions}.

\begin{definition}
The category of {\em monoidal models} $\MMod$ has as objects quadruples $(C,\Sigma,\cat C,i)$, where $(C,\Sigma)$ is a monoidal signature and $(\cat C,i)$ its model. A morphism
$$(f,F):(C,\Sigma,\cat C,i)\rightarrow (D,\Gamma,\cat D,j)$$
is given by a morphism of monoidal signatures $f:(C,\Sigma)\rightarrow (D,\Gamma)$ and a strict monoidal functor $F:\cat C\rightarrow\cat D$ such that the diagram on the left commutes (where $F_0$ is the action of $F$ on objects), and the diagram on the right commutes for all $a,b\in C^*$,
\begin{equation*}
\input{journal-figures/mor-mon-models.tikz}
\end{equation*}
where we denote by $F$ the appropriate restriction of the functor to the hom-set. Note that the restriction is indeed well-defined, as commutativity of the left diagram together with the fact that $F$ is strict monoidal imply that $i^* ; F_0 = f^* ; j^*$.
\end{definition}

Given a model $(\cat C,i)$ of a monoidal signature $(C,\Sigma)$, it extends to a function from terms to the morphisms of $\cat C$, also denoted by $i:\Term_{C,\Sigma}\rightarrow\Mor(\cat C)$, as follows:\label{p:extend-model-terms}
\begin{center}
\begin{tabular}{c c c}
$\sigma : (a,b) \mapsto i_{a,b}(\sigma) : (i^*a,i^*b)$, & $\id_a : (a,a) \mapsto \id_{ia} : (ia,ia)$, & $\id_{\varepsilon} : (\varepsilon,\varepsilon) \mapsto \id_I : (I,I)$, \\
$(t;s) : (a,c) \mapsto i(t);i(s) : (i^*a,i^*c)$, & \multicolumn{2}{c}{$(t\otimes s) : (ac,bd) \mapsto i(t)\otimes i(s) : (i^*(ac),i^*(bd))$,}
\end{tabular}
\end{center}
where $I$ is the unit of $\cat C$, while $;$ and $\otimes$ on the right-hand side are the composition and the monoidal product of $\cat C$.

\begin{definition}\label{def:model-monoidal-theory}
A {\em model} of a monoidal theory $(C,\Sigma,E)$ is a model $(\cat C,i)$ of the monoidal signature $(C,\Sigma)$ such that for all $(s,t)\in E$, we have $i(s)=i(t)$.
\end{definition}

We denote by $\MThMod$ the category of {\em models of monoidal theories}, whose objects are pairs of a monoidal theory and its model, and whose morphisms are pairs $(f,F)$ such that $f$ is a morphism of monoidal theories (hence, in particular, a morphism of monoidal signatures) and $(f,F)$ is a morphism in $\MMod$.

We note that every model of a monoidal signature $(C,\Sigma)$ can be viewed as a model of the monoidal theory $(C,\Sigma,\eset)$, so that $\MMod$ can be identified with the full subcategory of $\MThMod$ of theories with no equations.

We summarise the relationship between monoidal signatures, theories and their models in the following proposition.
\begin{proposition}\label{prop:monoidal-signatures-theories-models}
The vertical forgetful functors in the diagram below are fibrations, while the horizontal forgetful functors form a morphism of fibrations:
\begin{equation*}
\input{journal-figures/mthmod-pullback.tikz}.
\end{equation*}
Moreover, the objects in the fibre $\MMod(C,\Sigma)$ are precisely the models of the signature $(C,\Sigma)$, and the objects in the fibre $\MThMod(\mathcal T)$ are precisely the models of the theory $\mathcal T$.
\end{proposition}
\begin{proof}
The cartesian maps are those pairs whose monoidal functor part is (naturally isomorphic to) the identity functor.

Given an object $(C,\Sigma,E,\cat C,i)\in\MThMod_0$ and a morphism of monoidal theories $f:(D,\Gamma,F)\rightarrow (C,\Sigma,E)$, the cartesian lifting is given by $(f,\id_{\cat C}) : (D,\Gamma,F,\cat C,f;i)\rightarrow (C,\Sigma,E,\cat C,i)$, where the interpretation functions of the domain model are given by $f;i : D\rightarrow\cat C$ and
$$f_{a,b};i_{f^*a,f^*b} : \Gamma(a,b)\rightarrow\cat C(i^*f^*a,i^*f^*b).$$

For $\MMod\rightarrow\MSgn$, the situation is the same, except that all the theories are empty. It is then clear that the horizontal forgetful functors make the diagram commute, and that the top functor preserves cartesian liftings.
\end{proof}
Note that, while $\MTh\rightarrow\MSgn$ is also a fibration, the functor $\MThMod\rightarrow\MMod$ is {\em not} itself a fibration.

The usefulness of string diagrams lies in the fact that they are the {\em free models} of monoidal theories: an equation is derivable in the string diagrams of a monoidal theory if and only if it holds in all models of that theory.

\begin{definition}[Term model]\label{def:term-model}
Given a monoidal theory $\mathcal T = (C,\Sigma,E)$, the {\em term model} (or the {\em free model}) $F(\mathcal T)$ is the strict monoidal category defined by the following:
\begin{itemize}
\item the objects are $C^*$,
\item the morphisms $a\rightarrow b$ are the equivalence classes of terms with sort $(a,b)$ under the term congruence $\Eeq$,
\item the composition of $[s]$ and $[t]$ is given by $[s;t]$,
\item the monoidal unit is the empty word $\varepsilon\in C^*$,
\item the monoidal product is given by concatenation on objects and by $\otimes$ on morphisms.
\end{itemize}
We note that the operations $;$ and $\otimes$ above are well-defined since $\Eeq$ is a congruence. The interpretation function $i:C\rightarrow C^*$ sends the colour $a$ to the one-element word $a$, and $i_{a,b}:\Sigma(a,b)\rightarrow F(\mathcal T)(a,b)$ sends each generator to its equivalence class.
\end{definition}
For a monoidal signature $(C,\Sigma)$, we denote by $F(C,\Sigma)$ the free model on the theory $(C,\Sigma,\eset)$.

\begin{proposition}
The term model construction extends to functors $F:\MTh\rightarrow\MThMod$ and $F:\MSgn\rightarrow\MMod$, which are moreover the left adjoints to the respective forgetful functors.
\end{proposition}
\begin{proof}
The action of the functor on morphisms is given by the recursive extension of the morphism of monoidal signatures to terms defined on page~\pageref{p:extend-morphism-signatures-terms}.

The unit $\eta_{\mathcal T}:\mathcal T\rightarrow UF(\mathcal T) =\mathcal T$ is given by the identity $\id_{\mathcal T}$. The counit $\varepsilon_{\mathcal T,\mathcal M} : (\mathcal T, F(\mathcal T)) = FU(\mathcal T,\mathcal M)\rightarrow (\mathcal T,\mathcal M)$, where we denote $\mathcal M\coloneq (\mathcal C,i)$, is given by $(\id_{\mathcal T},i)$, where $i:F(\mathcal T)\rightarrow\mathcal C$ is given by $i^*$ on objects, and on morphisms by the recursive extension of $i$ to terms defined on page~\pageref{p:extend-model-terms}.
\end{proof}
\begin{corollary}\label{cor:monoidal-free-forgetful}
For every monoidal theory $\mathcal T$ and monoidal signature $(C,\Sigma)$, there are the following equivalences:
\begin{align*}
\MThMod(\mathcal T) &\simeq \quot{F(\mathcal T)}{\MonCat_{\mathsf{st}}}, \\
\MMod(C,\Sigma) &\simeq \quot{F(C,\Sigma)}{\MonCat_{\mathsf{st}}}.
\end{align*}
\end{corollary}

We conclude our discussion of models by pointing out that the notion of a model of a monoidal theory can be extended to non-strict monoidal categories upon interpreting the term congruence as isomorphism rather than equality. However, obtaining a free model via string diagrams is more involved in this case, as one needs to keep track of the isomorphisms capturing non-strict associativity and unitality. We refer the reader to Wilson, Ghica and Zanasi~\cite{non-strict-monoidal,non-strict-monoidal-journal} for the details.

\section{Layered theories}\label{sec:layered-theories}
Here we define the syntax of layered monoidal theories, including the classes of {\em opfibrational}, {\em fibrational} and {\em deflational} theories in Subsection~\ref{subsec:opfib-defl-theories}. While the technical details here are sufficient to follow the theoretical developments, we refer the reader to Part~I~\cite{lmt-part1} for a more detailed discussion and examples. We remark that the terms {\em layered theory} and {\em layered monoidal theory} are used interchangeably; ditto the terms {\em layered signature} and {\em layered monoidal signature}.

\begin{definition}[Layered signature]\label{def:layered-signature}
A {\em layered signature} is a tuple $\left(\Omega,\mathcal F,\left\{\M_{\omega}\right\}_{\omega\in\Omega}\right)$,
where $\Omega$ is a set, $\mathcal F:\Omega\times\Omega\rightarrow\Set$ is a function, and $\M_{\omega}$ is a monoidal signature for each $\omega\in\Omega$.
\end{definition}
Given a layered signature with monoidal signatures $\M_{\omega}$, we write $\M_{\omega}=(C_{\omega},\Sigma_{\omega})$. We often abbreviate a layered signature to $(\Omega,\mathcal F,\M_{\omega})$, where the index $\omega$ is implicitly assumed to range over $\Omega$. We refer to the elements of $\Omega$ as {\em layers} and to the elements of $\mathcal F(\omega,\tau)$ as {\em generators} with domain layer $\omega$ and codomain layer $\tau$.

A {\em morphism of layered signatures} $\left(F,F_{\tau,\omega},F^{\omega}\right):\left(\Omega,\mathcal F,\M_{\omega}\right)\rightarrow\left(\Psi,\mathcal G,\M_{\psi}\right)$ is given by: (1) a function $F:\Omega\rightarrow\Psi$, (2) for every pair $(\omega,\tau)\in\Omega\times\Omega$, a function $F_{\omega,\tau}:\mathcal F(\omega,\tau)\rightarrow\mathcal G(F(\omega), F(\tau))$, and (3) for every $\omega\in\Omega$, a morphism of monoidal signatures $F^{\omega}:\M_{\omega}\rightarrow\M_{F(\omega)}$.

We denote the resulting category of layered signatures by $\LSgn$. As for the category of monoidal signatures $\MSgn$, we denote a morphism in $\LSgn$ simply by $F:(\Omega,\mathcal F,\M_{\omega})\rightarrow (\Psi,\mathcal G,\M_{\psi})$ as follows: $F:\Omega\rightarrow\Psi$ (without subscripts or superscripts) denotes the function on layers, $F_{\omega,\tau}:\mathcal F(\omega,\tau)\rightarrow\mathcal G(F(\omega),F(\tau))$ the functions on the generators, and $F^{\omega}:\M_{\omega}\rightarrow\M_{F(\omega)}$ the morphisms of monoidal signatures.

Next, we define types, terms and 2-terms generated by a layered signature, summarised in the following table:
\begin{center}
\renewcommand{\arraystretch}{1.5}
\begin{tabular}{ c | c | c | c }
level & name & meaning & equations \\
\hhline{=|=|=|=}
$\Type_{\mathcal L}$ & types & objects / 0-cells & 0-equations (Definition~\ref{def:0equations}) \\
\hline
$\Term^1_{\mathcal L}$ & terms & morphisms / 1-cells & 1-equations (Definition~\ref{def:1equations}) \\
\hline
$\Term^2_{\mathcal L}$ & 2-terms & 2-morphisms / 2-cells & 2-equations (Definition~\ref{def:2equations})
\end{tabular}
\renewcommand{\arraystretch}{1}
\end{center}

\begin{definition}[Types of a layered signature]\label{def:types-layered-signature}
Given a layered signature $(\Omega,\mathcal F,\M_{\omega})$ with $\M_{\omega}=(C_{\omega},\Sigma_{\omega})$, the set of {\em types} $\Type_{\mathcal L}$ is recursively generated as follows:
\begin{center}
  \centering\small\noindent
  \begin{bprooftree}
    \AxiomC{\phantom{$\omega\in\Omega$}}
    \RightLabel{\;}
    \UnaryInfC{$\varepsilon : \varepsilon$}
    \DisplayProof
    \AxiomC{$\omega\in\Omega$}
    \RightLabel{\;}
    \UnaryInfC{$\varepsilon : \omega$}
    \DisplayProof
    \AxiomC{$\omega\in\Omega$}
    \AxiomC{$a\in C_{\omega}$}
    \RightLabel{\;}
    \BinaryInfC{$a : \omega$}
    \DisplayProof
    \AxiomC{$A : \omega$}
    \AxiomC{$f\in\mathcal F(\omega,\tau)$}
    \RightLabel{\;}
    \BinaryInfC{$f(A) : \tau$}
    \DisplayProof
    \AxiomC{$A : \omega$}
    \AxiomC{$B : \omega$}
    \RightLabel{\;}
    \BinaryInfC{$AB : \omega$}
    \DisplayProof
    \AxiomC{$T$}
    \AxiomC{$S$}
    \RightLabel{,}
    \BinaryInfC{$T,S$}
  \end{bprooftree}
\end{center}
such that the types formed with the last rule define a monoid with the unit $\varepsilon : \varepsilon$, and for each $\omega\in\Omega$, the types formed with the second to last rule define a monoid with the unit $\varepsilon : \omega$.
\end{definition}
Given a morphism of layered signatures $F:\mathcal L\rightarrow\mathcal K$ (with $\mathcal L = (\Omega,\mathcal F,\mathcal M_{\omega})$), it induces a function $F:\Type_{\mathcal L}\rightarrow\Type_{\mathcal K}$ recursively defined as follows:
\begin{align*}
\varepsilon : \varepsilon &\mapsto \varepsilon : \varepsilon &
\varepsilon : \omega &\mapsto \varepsilon : F(\omega) &
a : \omega &\mapsto F^{\omega}(a) : F(\omega) \\
f(A) : \tau &\mapsto F_{\omega,\tau}(f)(F(A)) : F(\tau) &
AB : \omega &\mapsto F(A)F(B) : F(\omega) &
T,S &\mapsto F(T),F(S).
\end{align*}

The types which are generated without applying the first and the last rules in Definition~\ref{def:types-layered-signature} are called {\em internal}. Thus, internal types are of the form $A:\omega$ for some $\omega\in\Omega$. We denote the subset of internal types by $\IntType_{\mathcal L}$. Let us define the binary relation $P^0_{\mathcal L}\sse\IntType_{\mathcal L}\times\IntType_{\mathcal L}$ on the internal types as containing those terms which are in the same layer: $(A:\omega,B:\tau)\in P^0_{\mathcal L}$ if and only if $\omega = \tau$. We note that this relation is preserved by any function $F:\Type_{\mathcal L}\rightarrow\Type_{\mathcal K}$ induced by a morphism of layered signatures: if $(T,S)\in P^0_{\mathcal L}$, then $(F(T),F(S))\in P^0_{\mathcal K}$.

\begin{definition}[0-equations]\label{def:0equations}
Given a layered signature $\mathcal L$, a set of {\em 0-equations} is a subset $E^0\sse P^0_{\mathcal L}$.
\end{definition}
Given a set of 0-equations $E^0$, we extend it to the {\em type congruence} $\Zeq$, i.e.~to the smallest equivalence relation on $\Type_{\mathcal L}$ satisfying $E^0\sse {\Zeq}$ as well as the recursive clauses:
\begin{center}
  \centering\small\noindent
  \begin{bprooftree}
    \AxiomC{$A:\omega\Zeq B:\omega$}
    \AxiomC{$f\in\mathcal F(\omega,\tau)$}
    \RightLabel{\;}
    \BinaryInfC{$f(A):\tau\Zeq f(B):\tau$}
    \DisplayProof
    \AxiomC{$A:\omega\Zeq C:\omega$}
    \AxiomC{$B:\omega\Zeq D:\omega$}
    \RightLabel{\;}
    \BinaryInfC{$AB:\omega\Zeq CD:\omega$}
    \DisplayProof
    \AxiomC{$T\Zeq S$}
    \AxiomC{$U\Zeq K$}
    \RightLabel{.}
    \BinaryInfC{$T,U\Zeq S,K$}
  \end{bprooftree}
\end{center}

We call the pairs of types $\Type_{\mathcal L}\times\Type_{\mathcal L}$ {\em sorts}, and denote a sort by $(T \mid S)$.
\begin{definition}[Basic terms]\label{def:terms-layered-signature}
Let $(\Omega,\mathcal F,\M_{\omega})$ be a layered signature. The {\em basic terms} are generated by the recursive sorting procedure in Figure~\ref{fig:layered-terms}, with the side condition that the rules~\ref{term:int-box} and \ref{term:int-tensor} only apply to {\em internal} terms, defined as follows:
\begin{itemize}
\item the terms generated by the rules~\ref{term:int-unit},~\ref{term:int-id},~\ref{term:int-gen},~\ref{term:int-box} or~\ref{term:int-tensor} are internal,
\item if the terms $x : (T \mid S)$ and $y : (S \mid U)$ are internal, then so is the term $x;y : (T \mid U)$ obtained by~\ref{term:comp}.
\end{itemize}
\end{definition}

\begin{figure}
  \centering\small\noindent
  \begin{prooftree}
    \AxiomC{$\omega\in\Omega$}
    \RightLabel{\scriptsize\customlabel{term:int-unit}{(int-unit)}\;\;}
    \UnaryInfC{$\scalebox{.8}{\input{journal-figures/emptydiag-sheet.tikz}} : (\varepsilon : \omega\mid\varepsilon : \omega)$}
    \DisplayProof
    \AxiomC{$A:\omega$}
    \AxiomC{$A\neq\varepsilon$}
    \RightLabel{\scriptsize\customlabel{term:int-id}{(int-id)}\;\;}
    \BinaryInfC{$\scalebox{.8}{\input{journal-figures/iddiag-sheet.tikz}} : (A:\omega \mid A:\omega)$}
    \DisplayProof
    \AxiomC{$\sigma\in\Sigma_{\omega}(a,b)$}
    \RightLabel{\scriptsize\customlabel{term:int-gen}{(int-gen)}}
    \UnaryInfC{$\scalebox{.8}{\input{journal-figures/internalsigmadiag.tikz}} : (a:\omega\mid b:\omega)$}
  \end{prooftree}
  \begin{bprooftree}
    \AxiomC{$\scalebox{.8}{\input{journal-figures/internalxdiag.tikz}} : (A:\omega\mid B:\omega)$}
    \AxiomC{$f\in\mathcal F(\omega,\tau)$}
    \RightLabel{\scriptsize\customlabel{term:int-box}{(int-box)}\;\;}
    \BinaryInfC{$\scalebox{.8}{\input{journal-figures/f-box.tikz}} : (f(A):\tau \mid f(B):\tau)$}
    \DisplayProof
    \AxiomC{$x : (T\mid S)$}
    \AxiomC{$y : (S\mid U)$}
    \RightLabel{\scriptsize\customlabel{term:comp}{(comp)}}
    \BinaryInfC{$x;y : (T\mid U)$}
  \end{bprooftree}
  \begin{prooftree}
    \AxiomC{$\scalebox{.8}{\input{journal-figures/internalxdiag.tikz}} : (A:\omega\mid B:\omega)$}
    \AxiomC{$\scalebox{.8}{\input{journal-figures/internalydiag.tikz}} : (C:\omega\mid D:\omega)$}
    \RightLabel{\scriptsize\customlabel{term:int-tensor}{(int-tensor)}}
    \BinaryInfC{$\scalebox{.8}{\input{journal-figures/internalxydiag.tikz}} : (AC:\omega \mid BD:\omega)$}
  \end{prooftree}
  \begin{bprooftree}
    \AxiomC{\phantom{$\omega\in\Omega$}}
    \RightLabel{\scriptsize\customlabel{term:ext-unit}{(ext-unit)}\;\;}
    \UnaryInfC{$\scalebox{.8}{\input{journal-figures/emptydiag.tikz}} : (\varepsilon : \varepsilon\mid\varepsilon : \varepsilon)$}
    \DisplayProof
    \AxiomC{$x : (T\mid S)$}
    \AxiomC{$y : (U \mid W)$}
    \RightLabel{\scriptsize\customlabel{term:ext-tensor}{(ext-tensor)}}
    \BinaryInfC{$x\otimes y : (T,U \mid S,W)$}
  \end{bprooftree}
  \caption{Rules for generating the basic terms of a layered signature\label{fig:layered-terms}}
\end{figure}

In the definition of a layered signature (Definition~\ref{def:layered-signature}), the generators between the layers are all of the type 1-1, i.e.~the arity and coarity both consist of a single layer (this is in contrast to the generators within the layers, whose arity and coarity are lists of colours). This restriction is imposed as there is no uniform way to draw an arbitrary generator between a list (product) of layers. Since we do want the layers to interact, we introduce the mixed (co)arities at the level of terms. To this end, we need the notion of a {\em recursive sorting procedure}, which imposes some restrictions on what kinds of terms can be introduced. Each layered theory comes with a fixed recursive sorting procedure, within which the terms of the theory are generated.

\begin{definition}[Recursive sorting procedure]\label{def:sorting-procedure}
A {\em recursive sorting procedure} is a set of recursive rules such that
\begin{itemize}
\item given a layered signature $\mathcal L$, the rules generate the set of expressions $\Term^1_{\mathcal L}$ called {\em terms},
\item each term has a unique sort: we denote a term with its sort by $t:(T\mid S)$,
\item the set contains the rules for the {\em basic terms} of Figure~\ref{fig:layered-terms},
\item any additional rules do not change the set of internal terms, as defined in Definition~\ref{def:terms-layered-signature},
\item any morphism of layered signatures $F:\mathcal L\rightarrow\mathcal K$ uniquely extends to a function $F:\Term^1_{\mathcal L}\rightarrow\Term^1_{\mathcal K}$ by translating the types and terms in the assumption of each recursive rule, and then applying the corresponding rule in $\Term^1_{\mathcal K}$, which is moreover consistent with the induced function on types: $t:(T\mid S)\mapsto F(t) : (F(T)\mid F(S))$.
\end{itemize}
\end{definition}
The additional rules we shall consider are given in Definition~\ref{def:symmetry-terms} as well as Figures~\ref{fig:opfibrational-terms} and~\ref{fig:fibrational-terms}.

We say that two terms are {\em parallel} when they have the same sort, up to 0-equations, if there are any. Formally, given a set of 0-equations $E^0$, we define the binary relation $P^1_{\mathcal L}\sse\Term^1_{\mathcal L}\times\Term^1_{\mathcal L}$ by $\left(t:(T_t\mid S_t),s:(T_s\mid S_s)\right)\in P^1_{\mathcal L}$ if and only if both $T_t\Zeq T_s$ and $S_t\Zeq S_s$, and call the pairs of terms in $P^1_{\mathcal L}$ {\em parallel} with respect to $E^0$. Note that if $E^0$ is empty (i.e.~there are no equations), then terms are parallel if and only if they have the same sort.

\begin{proposition}\label{prop:parallel-terms-preserved}
Let $F:(\mathcal L,E^0_{\mathcal L})\rightarrow (\mathcal K,E^0_{\mathcal K})$ be a morphism of layered signatures such that the induced function on types $F:\Type_{\mathcal L}\rightarrow\Type_{\mathcal K}$ preserves the chosen 0-equations. Then the induced function on terms $F:\Term^1_{\mathcal L}\rightarrow\Term^1_{\mathcal K}$ preserves parallel terms: if $(s,t)\in P^1_{\mathcal L}$, then $(F(s),F(t))\in P^1_{\mathcal K}$.
\end{proposition}
\begin{proof}
By the fact that terms with the same sort are mapped to the terms with the same sort, and by induction on the construction of the type congruence.
\end{proof}

\begin{definition}[1-equations]\label{def:1equations}
Given a layered signature $\mathcal L$ together with a recursive sorting procedure and a set of 0-equations $E^0$, a set of {\em 1-equations} with respect to $E^0$ is a subset $E^1\sse P^1_{\mathcal L}$.
\end{definition}
Given a set of 1-equations $E^1$, we extend it to the {\em term congruence} $\Oeq$, i.e.~the smallest equivalence relation on $P^1_{\mathcal L}$ containing $E^1$ that is preserved by the recursive rules in Figure~\ref{fig:layered-terms}.

\begin{definition}[Choice of 2-cells]\label{def:choice-2cells}
Given a layered signature $\mathcal L$ together with a recursive sorting procedure and a set of 0-equations $E^0$, a {\em choice of 2-cells} with respect to $E^0$ is a function $\eta : P^1_{\mathcal L}\rightarrow\Set$, assigning to each parallel pair of terms a set of {\em generating 2-cells}.
\end{definition}
Note that any layered signature in the sense of Definition~\ref{def:layered-signature} can be seen as having a choice of 2-cells by setting $\eta$ to be the constant function returning the empty set. A morphism between layered signatures with a choice of 2-cells $(F,F^2) : (\mathcal L, E^0_{\mathcal L}, \eta_{\mathcal L})\rightarrow (\mathcal K, E^0_{\mathcal K}, \eta_{\mathcal K})$ is given by a morphism of layered signatures $F:\mathcal L\rightarrow\mathcal K$ such that the induced function $F:\Type_{\mathcal L}\rightarrow\Type_{\mathcal K}$ preserves the 0-equations, and for each pair $(t,s)\in P^1_{\mathcal L}$, a function $F^2_{t,s}:\eta_{\mathcal L}(t,s)\rightarrow\eta_{\mathcal K}(F(t),F(s))$.

\begin{definition}[2-terms]\label{def:2terms}
Given a layered signature $\mathcal L$ with a choice of 2-cells $\eta$, a {\em 2-term} is an expression of the form $\alpha : (s,t)$, where $(s,t)\in P^1_{\mathcal L}$ is its {\em sort}, generated by the following recursive procedure:
\begin{center}
  \small
  \begin{bprooftree}
    \AxiomC{$a\in\eta(t,s)$}
    \RightLabel{\;\;}
    \UnaryInfC{$a : (t,s)$}
    \DisplayProof
    \AxiomC{$t\in\Term^1_{\mathcal L}$}
    \RightLabel{\;\;}
    \UnaryInfC{$\id_t : (t,t)$}
    \DisplayProof
    \AxiomC{$\alpha : (t,s)$}
    \AxiomC{$\beta : (s,k)$}
    \RightLabel{\;\;}
    \BinaryInfC{$\alpha;\beta : (t,k)$}
    \DisplayProof
    \AxiomC{$\alpha : (t_1,s_1)$}
    \AxiomC{$\beta : (t_2,s_2)$}
    \RightLabel{\;\;}
    \BinaryInfC{$\alpha\otimes\beta : (t_1\otimes t_2,s_1\otimes s_2)$}
  \end{bprooftree}
  \begin{prooftree}
    \AxiomC{$\alpha : (t_1,s_1)$}
    \AxiomC{$\beta : (t_2,s_2)$}
    \AxiomC{$t_1 : (T\mid S)$}
    \AxiomC{$t_2 : (S\mid K)$}
    \RightLabel{.}
    \QuaternaryInfC{$\alpha *\beta : (t_1;t_2, s_1;s_2)$}
  \end{prooftree}
  \end{center}
Akin to terms, we denote the set of 2-terms by $\Term^2_{\mathcal L}$.
\end{definition}
\begin{remark}
We have chosen not to extend the 2-terms to the recursively defined internal terms (i.e.~the ones generated by the rules~\ref{term:int-box} and~\ref{term:int-tensor}). There is no inherent reason for omitting these terms: upon adding them, one would just need to add the corresponding structural 2-equations (see Definition~\ref{def:str-2eqns}). However, in the models we will consider the 2-categorical structure only arises at the level of external terms, or does not need to be propagated between layers. Thus, for the sake of slightly reducing the complexity of the current presentation, we omit these terms. Moreover, in deflational theories, such terms are automatically induced by the external ones: see Corollary~\cite[Corollary~4.4]{lmt-part1}.
\end{remark}
Let $(F,F^2) : (\mathcal L, E^0_{\mathcal L}, \eta_{\mathcal L})\rightarrow (\mathcal K, E^0_{\mathcal K}, \eta_{\mathcal K})$ be a morphism between layered signatures with a choice of 2-cells. The functions between the generating 2-cells $F^2_{t,s}$ then yield a function $F^2:\Term^2_{\mathcal L}\rightarrow\Term^2_{\mathcal K}$ recursively defined as follows:
\begin{align*}
a : (t,s) &\mapsto F^2_{t,s}(a) : (F(t),F(s)) \\
\id_t : (t,t) &\mapsto \id_{F(t)} : (F(t),F(t)) \\
\alpha\otimes\beta : (t_1\otimes t_2,s_1\otimes s_2) &\mapsto F^2(\alpha)\otimes F^2(\beta) : (F(t_1)\otimes F(t_2),F(s_1)\otimes F(s_2)) \\
\alpha;\beta : (t,k) &\mapsto F^2(\alpha); F^2(\beta) : (F(t),F(k)) \\
\alpha *\beta : (t_1;t_2, s_1;s_2) &\mapsto F^2(\alpha) * F^2(\beta) : (F(t_1);F(t_2), F(s_1);F(s_2)).
\end{align*}

As for the terms, given a set of 1-equations $E^1$, a pair of 2-terms is in the {\em parallel 2-term relation} $P^2_{\mathcal L}$ with respect to $E^1$ if and only if both 2-terms have the same sort up to the 1-equations: $\left(\alpha:(t_{\alpha},s_{\alpha}), \beta:(t_{\beta},s_{\beta})\right)\in P^2_{\mathcal L}$ if and only if both $t_{\alpha}\Oeq t_{\beta}$ and $s_{\alpha}\Oeq s_{\beta}$. We draw a 2-term $\alpha : (t,s)$ either as \scalebox{.8}{\input{journal-figures/2term.tikz}}, or simply as an arrow $t\xrightarrow{\alpha} s$ (see e.g.~Figure~\ref{fig:structural-twocells-adjoints} on page~\pageref{fig:structural-twocells-adjoints}).

\begin{proposition}\label{prop:parallel-2terms-preserved}
Let $(F,F^2) : (\mathcal L, E^0_{\mathcal L}, \eta_{\mathcal L}, E^1_{\mathcal L})\rightarrow (\mathcal K, E^0_{\mathcal K}, \eta_{\mathcal K}, E^1_{\mathcal K})$
be a morphism between layered signatures with a choice of 2-cells such that the induced function on terms $F:\Term^1_{\mathcal L}\rightarrow\Term^1_{\mathcal K}$ preserves the specified 1-equations. Then the induced function on 2-terms $F^2:\Term^2_{\mathcal L}\rightarrow\Term^2_{\mathcal K}$ preserves parallel terms: if $(\alpha,\beta)\in P^2_{\mathcal L}$, then $(F^2(\alpha),F^2(\beta))\in P^2_{\mathcal K}$.
\end{proposition}
\begin{proof}
By the fact that 2-terms with the same sort are mapped to 2-terms with the same sort, and by induction on the construction of the term congruence.
\end{proof}

\begin{definition}[2-equations]\label{def:2equations}
Let $\mathcal L$ be layered signature with a set of 0-equations $E^0$. Given a set of 1-equations $E^1$ and a choice of 2-cells $\eta$ with respect to $E^0$, a set of {\em 2-equations} with respect to $E^1$ and $\eta$ is a subset $E^2\sse P^2_{\mathcal L}$.
\end{definition}
Given a set of 2-equations $E^2$, we extend it to the {\em 2-term congruence} $\Teq$, i.e.~the smallest equivalence relation on $P^2_{\mathcal L}$ containing $E^2$ that is preserved by the recursive rules in Definition~\ref{def:2terms}.

We now have all the ingredients to define a layered theory.

\begin{definition}[Layered theory]\label{def:layered-theory}
Let us fix a recursive sorting procedure (Definition~\ref{def:sorting-procedure}). A {\em layered theory} $(\mathcal L,E^0,E^1,\eta,E^2)$ consists of the following: a layered signature $\mathcal L$ (Definition~\ref{def:layered-signature}), a set of 0-equations $E^0$ (Definition~\ref{def:0equations}), a set of 1-equations $E^1$ with respect to $E^0$ (Definition~\ref{def:1equations}), a choice of 2-cells $\eta$ with respect to $E^0$ (Definition~\ref{def:choice-2cells}), a set of 2-equations $E^2$ with respect to $E^1$ and $\eta$ (Definition~\ref{def:2equations}).
\end{definition}
For a fixed recursive sorting procedure, a {\em morphism of layered theories}
$$(F,F^2) : (\mathcal L,E^0_{\mathcal L},E^1_{\mathcal L},\eta_{\mathcal L},E^2_{\mathcal L})\rightarrow (\mathcal K,E^0_{\mathcal K},E^1_{\mathcal K},\eta_{\mathcal K},E^2_{\mathcal K})$$
is given by a morphism of layered signatures with a choice of 2-cells $(F,F^2)$ (thus, in particular, the induced function $F:\Type_{\mathcal L}\rightarrow\Type_{\mathcal K}$ preserves the 0-equations), such that the induced functions $F:\Term^1_{\mathcal L}\rightarrow\Term^1_{\mathcal K}$ and $F^2:\Term^2_{\mathcal L}\rightarrow\Term^2_{\mathcal K}$ preserve the 1-equations and the 2-equations, respectively. We denote any category of layered theories by $\LTh$: note that each recursive sorting procedure results in a different category.
\begin{proposition}\label{prop:lth-lsgn-fibration}
Let us fix a recursive sorting procedure. The forgetful functor $\LTh\rightarrow\LSgn$ is a fibration. Moreover, the objects in the fibre $\LTh(\mathcal L)$ are precisely the layered theories with the signature $\mathcal L$.
\end{proposition}

Next, we define the structural equations that will be used to define the free models.

\begin{definition}[Structural 0-equations]\label{def:str-0eqns}
Given a layered signature, the {\em structural 0-equations} $S^0$ are given by:
\begin{align*}
(AB)C : \omega &= A(BC) : \omega, &
\varepsilon A : \omega &= A : \omega = A\varepsilon : \omega, &
f(AB):\tau &= f(A)f(B):\tau, &
f(\varepsilon) : \tau &= \varepsilon : \tau.
\end{align*}
\end{definition}

\begin{definition}[Structural 1-equations]\label{def:str-1eqns}
Given a layered signature with the structural 0-equations (Definition~\ref{def:str-0eqns}), the {\em structural 1-equations} $S^1$ are given by the following:
\begin{itemize}
\item the structural identities for monoidal theories (Definition~\ref{def:str-id}), where sequential composition is given by the rule~\ref{term:comp}, the identities are given by the terms~\ref{term:int-id}, parallel composition is given by the rule~\ref{term:ext-tensor}, and the monoidal unit is given by~\ref{term:ext-unit},
\item the structural identities for the monoidal product (the bottom half of Definition~\ref{def:str-id}) hold for the internal terms, where parallel composition is given by the rule~\ref{term:int-tensor}, the monoidal unit is given by~\ref{term:int-unit}, and the sequential composition is given by~\ref{term:comp}\footnote{We note that the structural identities for internal terms which only involve sequential composition and identities are already implied by requiring them for all terms.},
\item the following identity for the identity terms with sort $(AB:\omega\mid AB:\omega)$ generated by the rules~\ref{term:int-id} and~\ref{term:int-tensor}: $\scalebox{.8}{\input{journal-figures/int-id-equals-int-tensor.tikz}} : (AB:\omega\mid AB:\omega)$, where the term on the left-hand side is generated by the rule~\ref{term:int-id}, while the term on the right-hand side is obtained by applying~\ref{term:int-tensor} to the identity terms with sorts $(A:\omega\mid A:\omega)$ and $(B:\omega\mid B:\omega)$, which are, in turn, obtained from~\ref{term:int-id},
\item the identities in Figure~\ref{fig:structural-twocells-functors-int} involving the~\ref{term:int-box} terms, making each $f$ a monoidal functor.
\end{itemize}
\end{definition}

\begin{figure}
  \centering\small\noindent
  \scalebox{.8}{\input{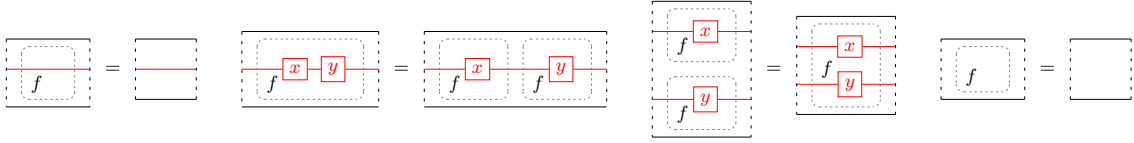}}
  \caption{1-equations defining monoidal functors inside the fibres\label{fig:structural-twocells-functors-int}}
\end{figure}

\begin{definition}[Structural 2-equations]\label{def:str-2eqns}
Given a layered signature and the structural 1-equations of Definition~\ref{def:str-1eqns}, the {\em structural 2-equations} $S^2$ are given in Figure~\ref{fig:str-2eqns} by the expressions on the left, as long as both sides are defined (on the right, we draw the situation in which each equation applies).
\end{definition}

\begin{figure}
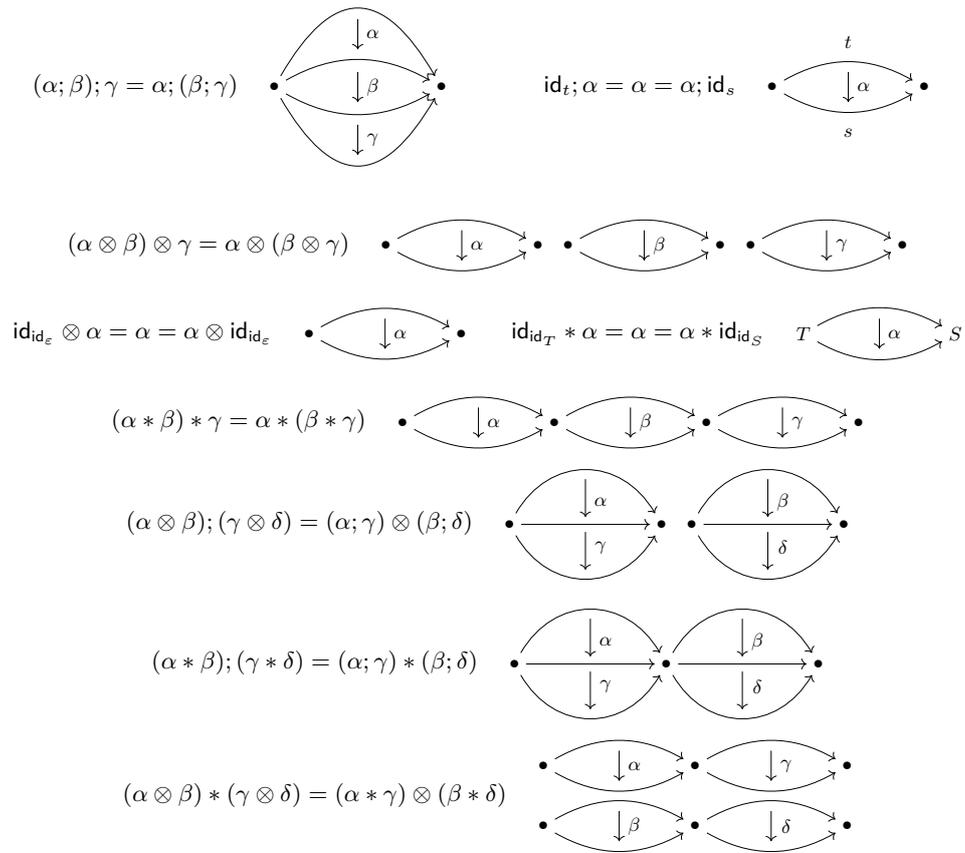

\centering\small
\renewcommand{\arraystretch}{3}
\begin{tabular}{c c}
$(\alpha;\beta);\gamma = \alpha;(\beta;\gamma)$\quad\scalebox{.8}{\input{journal-figures/2cells-compose-assoc.tikz}} & $\id_t;\alpha = \alpha = \alpha;\id_s$\quad\scalebox{.8}{\input{journal-figures/2cells-compose-id.tikz}} \\
\multicolumn{2}{c}{$(\alpha\otimes\beta)\otimes\gamma = \alpha\otimes(\beta\otimes\gamma)$\quad\scalebox{.8}{\input{journal-figures/2cells-tensor-assoc.tikz}}} \\
$\id_{\id_{\varepsilon}}\otimes\alpha = \alpha = \alpha\otimes\id_{\id_{\varepsilon}}$\quad\scalebox{.8}{\input{journal-figures/2cells-tensor-id.tikz}} & $\id_{\id_T} *\alpha = \alpha = \alpha *\id_{\id_S}$\quad\scalebox{.8}{\input{journal-figures/2cells-horizontal-id.tikz}} \\
\multicolumn{2}{c}{$(\alpha *\beta) *\gamma = \alpha *(\beta *\gamma)$\quad\scalebox{.8}{\input{journal-figures/2cells-horizontal-assoc.tikz}}} \\
\multicolumn{2}{c}{$(\alpha\otimes\beta);(\gamma\otimes\delta) = (\alpha;\gamma)\otimes (\beta;\delta)$\quad\scalebox{.8}{\input{journal-figures/2cells-tensor-compose.tikz}}} \\
\multicolumn{2}{c}{$(\alpha *\beta);(\gamma *\delta) = (\alpha;\gamma) * (\beta;\delta)$\quad\scalebox{.8}{\input{journal-figures/2cells-horizontal-compose.tikz}}} \\
\multicolumn{2}{c}{$(\alpha\otimes\beta) * (\gamma\otimes\delta) = (\alpha *\gamma)\otimes (\beta *\delta)$\quad\scalebox{.8}{\input{journal-figures/2cells-tensor-horizontal.tikz}}}
\end{tabular}
\renewcommand{\arraystretch}{1}
\caption{The structural 2-equations. Here $\id_T$ and $\id_S$ are the identity terms on the types $T$ and $S$, while $\id_{\varepsilon}$ is the identity term on $\varepsilon : \varepsilon$ obtained by the rule~\ref{term:ext-unit}.\label{fig:str-2eqns}}
\end{figure}

We say that a layered theory {\em has structural equations} when its sets of equations contain all the structural equations of Definitions~\ref{def:str-0eqns}, \ref{def:str-1eqns} and~\ref{def:str-2eqns}.

\noindent\begin{minipage}{0.4\textwidth}
\begin{definition}[Symmetry terms]\label{def:symmetry-terms}
Given a layered signature, the {\em symmetry terms} are generated by the rule on the right:
\end{definition}\end{minipage}
\begin{minipage}{0.6\textwidth}
\begin{center}
  \begin{prooftree}
    \AxiomC{$A:\omega$}
    \AxiomC{$B:\tau$}
    \RightLabel{\customlabel{term:swap}{(swap)}.\;\;}
    \BinaryInfC{$\scalebox{.8}{\input{journal-figures/symdiag-sheet1.tikz}} : (A:\omega, B:\tau \mid B:\tau, A:\omega)$}
  \end{prooftree}
\end{center}
\end{minipage}

\begin{definition}[Externally symmetric layered theory]\label{def:symmetry-1equations}
A layered theory is {\em externally symmetric} if it has the symmetry terms, and the structural 0-equations (Definition~\ref{def:str-0eqns}) and 1-equations (Definition~\ref{def:str-1eqns}) hold, and the 1-equations contain the equations for symmetric monoidal theories in Definition~\ref{def:symm-mon-thy} for the symmetry terms.
\end{definition}

\subsection{Opfibrational, fibrational and deflational theories}\label{subsec:opfib-defl-theories}

Opfibrational theories are defined for the recursive procedure specified in Definition~\ref{def:opfib-layered-theory} and have no generating 2-cells.

\begin{figure}
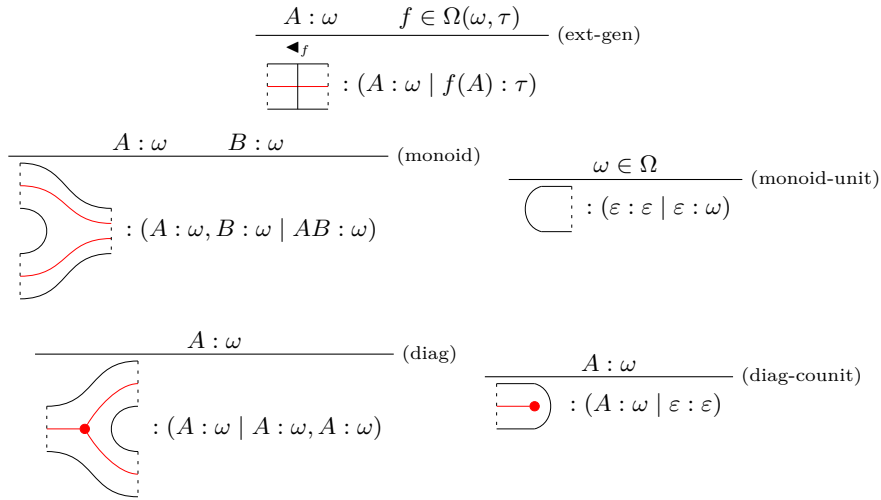

  \centering\small\noindent
  \begin{prooftree}
    \AxiomC{$A:\omega$}
    \AxiomC{$f\in\Omega(\omega,\tau)$}
    \RightLabel{\scriptsize\customlabel{term:ext-gen}{(ext-gen)}}
    \BinaryInfC{$\scalebox{.8}{\input{journal-figures/refine-sheet.tikz}} : (A:\omega \mid f(A):\tau)$}
  \end{prooftree}
  \begin{bprooftree}
    \AxiomC{$A:\omega$}
    \AxiomC{$B:\omega$}
    \RightLabel{\scriptsize\customlabel{term:monoid}{(monoid)}\;\;}
    \BinaryInfC{$\scalebox{.8}{\input{journal-figures/pants.tikz}} : (A:\omega, B:\omega \mid AB:\omega)$}
    \DisplayProof
    \AxiomC{$\omega\in\Omega$}
    \RightLabel{\scriptsize\customlabel{term:monoid-unit}{(monoid-unit)}}
    \UnaryInfC{$\scalebox{.8}{\input{journal-figures/cup.tikz}} : (\varepsilon : \varepsilon\mid\varepsilon : \omega)$}
  \end{bprooftree}
  \begin{prooftree}
    \AxiomC{$A:\omega$}
    \RightLabel{\scriptsize\customlabel{term:diag}{(diag)}\;\;}
    \UnaryInfC{$\scalebox{.8}{\input{journal-figures/copants-copy.tikz}} : (A:\omega \mid A:\omega, A:\omega)$}
    \DisplayProof
    \AxiomC{$A:\omega$}
    \RightLabel{\scriptsize\customlabel{term:diag-counit}{(diag-counit)}}
    \UnaryInfC{$\scalebox{.8}{\input{journal-figures/a-cap.tikz}} : (A:\omega \mid \varepsilon : \varepsilon)$}
  \end{prooftree}
  \caption{Rules for generating the opfibrational terms\label{fig:opfibrational-terms}}
\end{figure}

\begin{definition}[Opfibrational layered theory]\label{def:opfib-layered-theory}
We say that a layered theory is {\em opfibrational} if it has no generating 2-cells, and its recursive sorting procedure consists of the rules in Figure~\ref{fig:opfibrational-terms} and the symmetry terms of Definition~\ref{def:symmetry-terms} (in addition to the basic terms in Figure~\ref{fig:layered-terms}).
\end{definition}

The intuition behind the opfibrational terms is that we take an ``external'' view of the monoidal categories and monoidal functors between them. In more detail:
\begin{itemize}
\item the terms~\ref{term:monoid} and~\ref{term:monoid-unit} capture the monoidal product and unit in the layer $\omega$: e.g.~sliding the internal terms through the term~\ref{term:monoid} defines the monoidal product on morphisms (see Figure~\ref{fig:structural-twocells-monoidal}),
\item the terms~\ref{term:ext-gen} capture the monoidal functors $f:\omega\rightarrow\tau$: any internal term appearing on the left-hand side of the boundary is in the domain $\omega$, and can be pushed through the boundary into the codomain $\tau$ (see Figure~\ref{fig:structural-twocells-functors-ext}),
\item the terms~\ref{term:diag-counit} and~\ref{term:diag} capture the cartesian monoidal structure of the category of monoidal categories; note that they do not correspond to any deleting or copying {\em inside} a layer, rather, it might be instructive to think of them as {\em branching} or {\em nondeterminism}: \ref{term:diag} represents branching into two potential histories of a system, while \ref{term:diag-counit} corresponds to discarding one of the branches or histories.
\end{itemize}

\begin{remark}
A more accurate name for the theories of Definition~\ref{def:opfib-layered-theory} would be a {\em strict layered theory for opfibrations with indexed monoids}. Since this is somewhat cumbersome, and we shall not consider any other layered theories featuring an opfibration, we stick to the term {\em opfibrational}.
\end{remark}

We denote the resulting category of opfibrational layered theories by $\OpFTh$. In this case the morphisms are simply equation preserving morphisms of signatures rather than pairs of morphisms, as there are no generating 2-cells.

\begin{definition}[Structural opfibrational equations]\label{def:str-opfib-eqns}
By the {\em structural opfibrational equations} we mean the structural equations, the external symmetry equations, as well as the following sets of 1-equations, denoted by $S_{\mathsf{opf}}^1$: (1) the equations of uniform comonoids (Definition~\ref{def:thy-univ-comonoids}) for the terms~\ref{term:diag} and~\ref{term:diag-counit}, (2) the defining identities of monoidal categories in Figure~\ref{fig:structural-twocells-monoidal}, and (3) the defining identities of monoidal functors in Figure~\ref{fig:structural-twocells-functors-ext}.
\end{definition}

\begin{figure}
  \centering\small\noindent
  \scalebox{.8}{\input{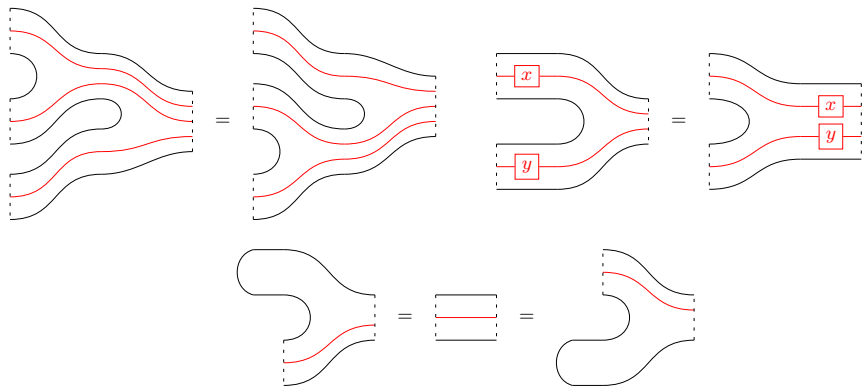}}
  \caption{1-equations defining monoidal categories\label{fig:structural-twocells-monoidal}}
\end{figure}

\begin{figure}
  \centering\small\noindent
  \scalebox{.8}{\input{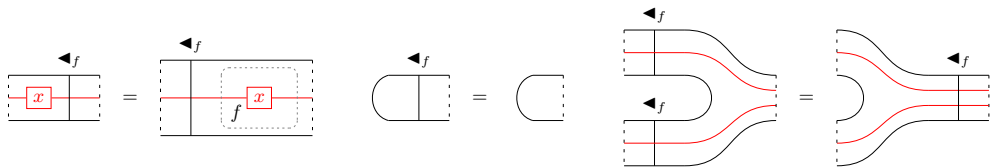}}
  \caption{1-equations defining monoidal functors\label{fig:structural-twocells-functors-ext}}
\end{figure}

Dualising the construction of opfibrational theories, we obtain {\em fibrational theories}, where the composition of internal and external terms go in opposite directions. Formally, we obtain opfibrational theories by replacing the opfibrational terms (Figure~\ref{fig:opfibrational-terms}) by the {\em fibrational terms} (Figure~\ref{fig:fibrational-terms}) in Definition~\ref{def:opfib-layered-theory}. Note that each fibrational term is a horizontal reflection of an opfibrational term, and that in the rule~\ref{term:ext-gen-op} we use of an unfilled triangle for emphasis -- the lack of filling plays no formal role.

\begin{figure}
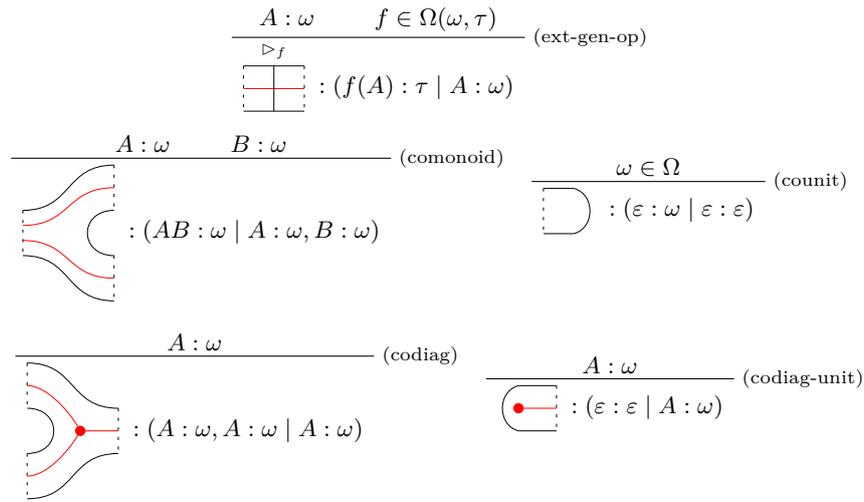

  \centering\small\noindent
  \begin{prooftree}
    \AxiomC{$A:\omega$}
    \AxiomC{$f\in\Omega(\omega,\tau)$}
    \RightLabel{\scriptsize\customlabel{term:ext-gen-op}{(ext-gen-op)}}
    \BinaryInfC{$\scalebox{.8}{\input{journal-figures/coarsen-sheet.tikz}} : (f(A):\tau \mid A:\omega)$}
  \end{prooftree}
  \begin{bprooftree}
    \AxiomC{$A:\omega$}
    \AxiomC{$B:\omega$}
    \RightLabel{\scriptsize\customlabel{term:comonoid}{(comonoid)}\;\;}
    \BinaryInfC{$\scalebox{.8}{\input{journal-figures/copants.tikz}} : (AB:\omega \mid A:\omega, B:\omega)$}
    \DisplayProof
    \AxiomC{$\omega\in\Omega$}
    \RightLabel{\scriptsize\customlabel{term:counit}{(counit)}\;\;}
    \UnaryInfC{$\scalebox{.8}{\input{journal-figures/cap.tikz}} : (\varepsilon : \omega \mid \varepsilon : \varepsilon)$}
  \end{bprooftree}
  \begin{prooftree}
    \AxiomC{$A:\omega$}
    \RightLabel{\scriptsize\customlabel{term:codiag}{(codiag)}\;\;}
    \UnaryInfC{$\scalebox{.8}{\input{journal-figures/pants-copy.tikz}} : (A:\omega, A:\omega \mid A:\omega)$}
    \DisplayProof
    \AxiomC{$A:\omega$}
    \RightLabel{\scriptsize\customlabel{term:codiag-unit}{(codiag-unit)}}
    \UnaryInfC{$\scalebox{.8}{\input{journal-figures/a-cup.tikz}} : (\varepsilon : \varepsilon \mid A:\omega)$}
  \end{prooftree}
  \caption{Rules for generating the fibrational terms\label{fig:fibrational-terms}}
\end{figure}

The structural fibrational equations are likewise obtained by dualising the structural opfibrational equations (Definition~\ref{def:str-opfib-eqns}). For the sake of completeness, and since they will be used for defining the structural deflational equations, we state the definition explicitly.
\begin{definition}[Structural fibrational equations]\label{def:str-fib-eqns}
By the {\em structural fibrational equations} we mean the structural equations, the external symmetry equations, as well as the following sets of 1-equations, where ``dual'' and ``dualising'' means reflecting the terms horizontally: (1) the equations of uniform monoids (the dual of Definition~\ref{def:thy-univ-comonoids}) for the terms~\ref{term:codiag} and~\ref{term:codiag-unit}, (2) the defining identities of monoidal categories obtained by dualising Figure~\ref{fig:structural-twocells-monoidal}, and (3) the defining identities of monoidal functors obtained by dualising Figure~\ref{fig:structural-twocells-functors-ext}.
\end{definition}
With these modifications, any properties of fibrational theories can be obtained by dualising the corresponding properties of opfibrational theories: all the diagrams are simply reversed.

We now arrive to our last -- and most important -- class of theories. Deflational theories can be thought of as glueing an opfibrational and a fibrational theory together along the internal terms.

\begin{definition}[Deflational layered theory]\label{def:deflational-theory}
A layered theory is {\em deflational} if its recursive sorting procedure consists of the rules for the symmetry terms of Definition~\ref{def:symmetry-terms}, as well as of the rules for opfibrational and fibrational terms in Figures~\ref{fig:opfibrational-terms} and~\ref{fig:fibrational-terms} (in addition to the basic terms in Figure~\ref{fig:layered-terms}).
\end{definition}
We denote the category of deflational layered theories by $\DeflTh$. In a deflational theory, every opfibrational term $x:(T\mid S)$ generated by a rule in Figure~\ref{fig:opfibrational-terms} has a corresponding fibrational term $\bar x:(S\mid T)$ generated by the corresponding rule in Figure~\ref{fig:fibrational-terms}. We use this to define the {\em structural 2-cells}, expressing $\bar x$ as the right adjoint to $x$.
\begin{definition}[Structural 2-cells]\label{def:str-twocells}
Given a deflational layered theory with layered signature $\mathcal L$, the {\em structural 2-cells} are given by the choice of 2-cells $\eta_{\mathsf{str}}$ defined as follows: for every opfibrational term $x:(T\mid S)$ generated by a rule in Figure~\ref{fig:opfibrational-terms} one has $\eta_{\mathsf{str}}\left(\id_{T},x;\bar x\right) \coloneq\{\eta_x\}$, $\eta_{\mathsf{str}}\left(\bar x;x,\id_{S}\right) \coloneq\{\varepsilon_x\}$, and $\eta_{\mathsf{str}}$ returns the empty set otherwise.
\end{definition}
We display the 10 structural 2-cells explicitly in Figure~\ref{fig:structural-twocells-adjoints}: the left-hand side contains the unit 2-cells ($\eta_x$), while the right-hand side contains the counit 2-cells ($\varepsilon_x$). We consistently omit the labels of the structural 2-cells, as the domain and the codomain determine them uniquely.

\begin{figure}
  \centering\small\noindent
  \scalebox{.75}{\input{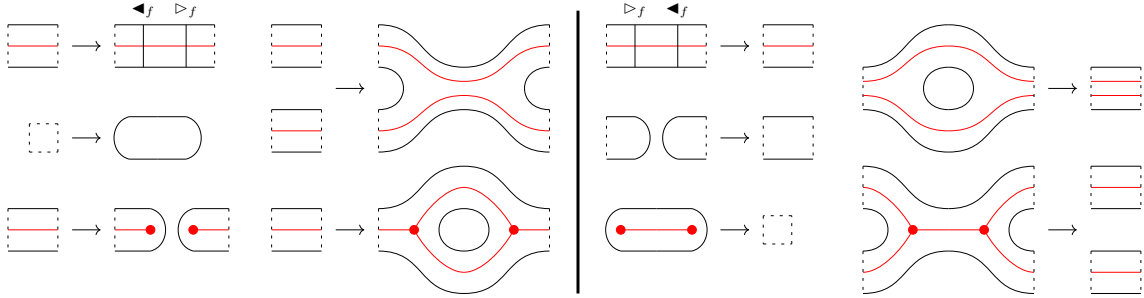}}
  \caption{Structural 2-cells for deflational theories\label{fig:structural-twocells-adjoints}}
\end{figure}

\begin{definition}[Structural deflational equations]\label{def:str-defl-eqns}
By the {\em structural deflational equations} we mean: the structural opfibrational equations (Definition~\ref{def:str-opfib-eqns}), the structural fibrational equations (Definition~\ref{def:str-fib-eqns}) and the following 2-equations for each opfibrational term $x:(T\mid S)$ and all terms $y:(T\mid T)$ and $z:(S\mid S)$ that are either internal or a product (obtained by the~\ref{term:ext-tensor}) of internal terms:
\begin{align*}
\left(\eta_x*\id_x\right);\left(\id_x*\varepsilon_x\right) &=\id_x &
\left(\id_{\bar x}*\eta_x\right);\left(\varepsilon_x*\id_{\bar x}\right) &=\id_{\bar x} &
\id_y*\eta_x &= \eta_x*\id_y &
\id_z*\varepsilon_x &= \varepsilon_x*\id_z.
\end{align*}
We denote the structural deflational {1-} and 2-equations by $S_{\mathsf{defl}}^1$ and $S_{\mathsf{defl}}^2$.
\end{definition}
Note that the structural deflational equations contain, in particular, the structural {0-}, {1-} and 2-equations of Definitions~\ref{def:str-0eqns}, \ref{def:str-1eqns} and~\ref{def:str-2eqns}. The additional 2-equations that are required to hold are the usual zigzag equations defining an adjunction (see Definition~\ref{def:zigzag-category} for their graphical depiction), as well as the equations stating the compatibility of the structural 2-cells with the ``sliding through'' 1-equations for internal terms in Figures~\ref{fig:structural-twocells-monoidal} and~\ref{fig:structural-twocells-functors-ext}.

\section{Indexed monoids}\label{sec:indexed-monoidal}
In this section, we study categories with indexed monoids (Definition~\ref{def:indexed-monoids}). We ultimately extend this notion to opfibrations with indexed monoids (Definition~\ref{def:opfib-indexed-mon}), a special case of which will be used as a semantics for opfibrational theories in Section~\ref{subsec:opfib-models}.

A symmetric monoidal category is denoted by $(\cat C,\otimes,I,\alpha,\lambda,\rho,\sigma)$, where $I$ is the monoidal unit, while the Greek letters denote the natural isomorphisms: $\alpha$ is the associator, $\lambda$ and $\rho$ are, respectively, the left and right unitors, while $\sigma$ is the symmetry. For brevity, we shall leave the natural isomorphisms implicit, and simply write $(\cat C,\otimes,I)$ for a monoidal category. We take the liberty to use string diagrams even when reasoning about non-strict monoidal categories: we trust that the reader will be able to infer the missing natural isomorphisms where necessary.

\subsection{Categories with indexed monoids}

Here we define categories with indexed monoids, and show that free models of monoidal theories with indexed monoids give the left adjoint to the forgetful functor from the categories with indexed monoids to categories.

The following definition should be compared to Definition~\ref{def:thy-univ-comonoids}, where we defined the monoidal theory with uniform comonoids.
\begin{definition}[Uniform comonoids]
A symmetric monoidal category $(\cat C,\otimes,I)$ has {\em uniform comonoids} if the following obtain:
\begin{itemize}
\item for every object $x\in\Ob(\cat C)$, there are morphisms $d_x: x\rightarrow x\otimes x$ and $e_x:x\rightarrow I$,
\item for every object $x$, the morphism $d_x$ is a comonoid with the counit $e_x$ (up to the natural isomorphisms),
\item $d_I=\rho_I^{-1}$ and $e_I=\id_1$,
\item $d$ and $e$ are natural, i.e.~for every morphism $f:x\rightarrow y$, we have $d_x;(f\otimes f)=f;d_y$ and $f;e_y = e_x$,
\item for all objects $x,y\in\Ob(\cat C)$, we have (up to associators) $(d_x\otimes d_y) ; (\id_x\otimes\sigma_{xy}\otimes\id_y)= d_{x\otimes y}$.
\end{itemize}
\end{definition}

Note that the naturality conditions imply that the uniform comonoids are unique, in the sense that if both $(d,e)$ and $(d',e')$ are uniform comonoids on the same symmetric monoidal category, then $(d,e)=(d',e')$. The following result is known as {\em Fox's theorem}. We refer the reader to Melli\` es~\cite[Section~6.4]{mellies-cat-semantics-linear-logic} for a more detailed discussion, including more refined versions of the equivalence.
\begin{proposition}\label{prop:fox}
A symmetric monoidal category $(\cat C,\otimes,I)$ is cartesian monoidal (i.e.~$\otimes$ is the cartesian product and $I$ is the terminal object) if and only if it has uniform comonoids.
\end{proposition}
\begin{proof}
Suppose $(\cat C,\otimes,I)$ has uniform comonoids $(d,e)$. Given maps $f$ and $g$ as below
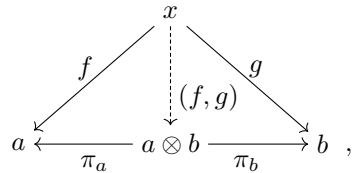
\begin{equation*}
\scalebox{1}{\input{journal-figures/product-diagram.tikz}},
\end{equation*}
the map $(f,g) : x\rightarrow a\otimes b$ is given by $d_x;(f\otimes g)$, while the projections are given by $\pi_a\coloneq(\id_a\otimes e_b);\rho_a$ and $\pi_b\coloneq(e_a\otimes\id_b);\lambda_b$. It is then immediate that composing $(f,g)$ with the projections recovers the original maps $f$ and $g$. For uniqueness, suppose there is some map $h:x\rightarrow a\otimes b$ such that
\begin{equation*}
\scalebox{.7}{\input{journal-figures/fox-proof-uniqueness-equations.tikz}}.
\end{equation*}
We then have that $h=(f,g)$ as follows:
\begin{equation*}
\scalebox{.7}{\input{journal-figures/fox-proof-uniqueness.tikz}}.
\end{equation*}
Finally, observe that $I$ is the terminal object by naturality of $e$.

Conversely, if $(\cat C,\otimes,I)$ is cartesian monoidal, then $d_x$ is the diagonal and $e_x$ is the unique map into the terminal object.
\end{proof}

\begin{definition}[Consistent monoids]
{\em Consistent monoids} in a symmetric monoidal category $(\cat C,\otimes,I)$ consist of the following data:
\begin{itemize}
\item for every object $x\in\Ob(\cat C)$, there are morphisms $m_x: x\otimes x\rightarrow x$ and $u_x:I\rightarrow x$,
\item for every object $x$, the morphism $m_x$ is a monoid with the unit $u_x$ (up to the natural isomorphisms),
\item $m_I=\rho_I$ and $u_I=\id_I$,
\item $u$ is natural, i.e.~for every morphism $f:x\rightarrow y$, we have $u_x ; f = u_y$,
\item for all objects $x,y\in\Ob(\cat C)$, we have (up to associators) $(\id_x\otimes\sigma_{yx}\otimes\id_y) ; (m_x\otimes m_y) = m_{x\otimes y}$.
\end{itemize}
\end{definition}
Similarly to uniform comonoids, the units of consistent monoids are unique: in fact, naturality and $u_I=\id_I$ imply that $I$ is the initial object. Contrariwise, since naturality is not asked of monoid multiplication, they need not be unique, and therefore correspond to a choice of structure rather than being a property of a category. This motivates the following definition.
\begin{definition}[Indexed monoids]\label{def:indexed-monoids}
We call a choice of consistent monoids in a symmetric monoidal category with uniform comonoids {\em indexed monoids}. We denote a {\em category with indexed monoids} by $(\cat C,\otimes,I,m)$.
\end{definition}
Using Fox's theorem (Proposition~\ref{prop:fox}), we restate the definition of a category with indexed monoids as follows.
\begin{proposition}
A symmetric monoidal category $(\cat C,\otimes,I)$ has indexed monoids if
\begin{itemize}
\item it is cartesian monoidal,
\item $I$ is also an initial object (thus a zero object),
\item for every object $x\in\Ob(\cat C)$, there is a monoid $m_x: x\otimes x\rightarrow x$ with unit $u_x:I\rightarrow x$,
\item $m_I=\rho_I$,
\item for all objects $x,y\in\Ob(\cat C)$, we have (up to associators) $(\id_x\otimes\sigma_{yx}\otimes\id_y) ; (m_x\otimes m_y) = m_{x\otimes y}$.
\end{itemize}
\end{proposition}
\begin{example}
Let $(\Vect_k,\oplus,\{0\})$ be the category of vector spaces over a fixed field $k$, with the direct sum as the monoidal product. Then, for every vector space $V$, addition of vectors $+:V\oplus V\rightarrow V$ gives an indexed monoid structure. This example generalises to any semiadditive category.
\end{example}
\begin{example}
Consider the forgetful functor $U:\Mon\rightarrow\Set_*$ from the category of monoids to the category of pointed sets. Define the category $U\mdash\Mon$ to have the same objects as $\Mon$, and the hom-sets as $U\mdash\Mon(M,N)\coloneqq\Set_*(UM,UN)$, i.e.~the maps are required to preserve the monoidal unit, but not the binary operation. The category $U\mdash\Mon$ has a cartesian monoidal structure given by the product of monoids, the zero object is the monoid with one element, and for every monoid $M=(X,\cdot,1)$, we define the morphism $m_M:M\times M\rightarrow M$ by the multiplication of the monoid: $(x,y)\mapsto x\cdot y$. This gives $U\mdash\Mon$ an indexed monoid structure.
\end{example}
Let us denote by $\IMon$ the category whose objects are symmetric monoidal categories with indexed monoids, and whose morphisms are strong monoidal functors $(F,\mu):(\cat C,\otimes,I,m)\rightarrow (\cat D,\otimes,I,m)$ such that for all $x\in\Ob(\cat C)$, we have $\mu_{xx};F(m_x)=m_{Fx}$.
\begin{definition}
Let $\cat I=(\cat C,\otimes,I,m)$ be a category with indexed monoids. A morphism $f:x\rightarrow y$ in $\cat C$ is a {\em monoid homomorphism} if $m_x;f = (f\otimes f) ; m_y$. We denote the set of monoid homomorphisms by $\hom(\cat I)$.
\end{definition}
\begin{proposition}\label{prop:imon-subcat-preserved}
The monoid homomorphisms in a category with indexed monoids form a wide monoidal subcategory. Moreover, the monoid homomorphisms are preserved by the morphisms in $\IMon$.
\end{proposition}
\begin{proof}
The identity morphisms are evidently monoid homomorphisms. Closure under composition and monoidal products follows from the defining equation of monoid homomorphisms and the consistency requirement. Preservation of a monoid homomorphism $f:x\rightarrow y$ under a functor $(F,\mu)$ in $\IMon$ amounts to the equation $m_{Fx};Ff = (Ff\otimes Ff);m_{Fy}$, which we obtain as the commutativity of the following diagram:
\begin{equation*}
\scalebox{1}{\input{journal-figures/monoid-homomorphism-preserved.tikz}},
\end{equation*}
where the bottom square commutes since $f$ is a monoid homomorphism, and the top square commutes by naturality of the strongator $\mu$.
\end{proof}

We define the functor $\hom:\IMon\rightarrow\Cat$ by mapping each $\cat I$ to $\hom(\cat I)$, and each functor to its restriction to monoid homomorphisms. We conclude this section by showing that the monoid homomorphism functor so defined has a left adjoint, obtained by freely adding indexed monoids.

\begin{definition}[Theory with indexed monoids generated by a category]\label{def:indexed-by-category}
Let $\cat X$ be a small category. The {\em theory with indexed monoids generated by $\cat X$} is the theory with indexed monoids (Example~\ref{def:thy-indexed-monoids}) $\mathsf{im}(\cat X)=(\Ob(\cat X),\Sigma,\mathcal I)$ whose set of colours is given by the objects of $\cat X$, and for each $a,b\in\Ob(\cat X)$ we have $\Sigma(a,b)\coloneqq\cat X(a,b)$, together with the following additional equations for all $a\in\Ob(\cat X)$ and all $f\in\cat X(a,b)$, $g\in\cat X(b,c)$ and $h\in\cat X(a,c)$ such that $f;g=h$:
\begin{equation*}
\scalebox{.7}{\input{journal-figures/id-eqn.tikz}}.
\end{equation*}
\end{definition}

\begin{definition}[Free category with indexed monoids]\label{def:free-indexed-cat}
Let $\cat X$ be a small category. The {\em free category with indexed monoids generated by $\cat X$} is the term model (Definition~\ref{def:term-model}) $\Fim(\cat X)$ of the theory with indexed monoids generated by $\cat X$ (Definition~\ref{def:indexed-by-category}).
\end{definition}

\begin{proposition}
In the situation of Definition~\ref{def:free-indexed-cat}, the interpretation functions $i:\Ob(\cat X)\rightarrow\Ob(\cat X)^*$ and $i_{a,b}:\cat X(a,b)\rightarrow \Fim(\cat X)(a,b)$ define a functor $i:\cat X\rightarrow \hom\left(\Fim(\cat X)\right)$. Moreover, the functor $i$ satisfies the following universal property: for every category with indexed monoids $\cat I$ and a functor $G:\cat X\rightarrow\hom\left(\cat I\right)$ in $\Cat$, there is a unique (up to a monoidal natural isomorphism) morphism $\bar G : \Fim(\cat X)\rightarrow\cat I$ in $\IMon$ such that the diagram
\begin{equation*}
\scalebox{1}{\input{journal-figures/indexed-universal.tikz}}
\end{equation*}
commutes; thus exhibiting $\Fim$ as the left adjoint to $\hom:\IMon\rightarrow\Cat$.
\end{proposition}
\begin{proof}
Functoriality of $i$ follows from the equations in Definition~\ref{def:indexed-by-category}. The fact that $i_{a,b}(f)=f$ is a monoid homomorphism is part of the definition of a theory with $1\mdash 1$-natural monoids (Definition~\ref{def:thy-11-nat-monoids}).

When defining a morphism $\bar G : \Fim(\cat X)\rightarrow\cat I$ in $\IMon$, the only freedoms we have is choosing where to map the colours (objects of $\cat X$) and the generators that are not part of the indexed monoid structure (the morphisms in $\hom\left(\Fim(\cat X)\right)$), both of which are uniquely (up to the structural isomorphisms of $\cat I$) determined by commutativity of the diagram.
\end{proof}

\subsection{Opfibrations with indexed monoids}\label{subsec:opfibrations-indexed-monoids}

Here we extend the notion of indexed monoids from categories to opfibrations. The central notion is that of an {\em opfibration with indexed monoids}, or {\em im-opfibration}, which is shown to be very closely related to monoids in the category of split opfibrations $\MonOpFib_{\mathsf{sp}}(\cat X)$ (Theorem~\ref{thm:monopfib-imopfib}). We choose to work with opfibrations rather than fibrations, however, all the developments dualise to fibrations, in which the appropriate notion is {\em indexed comonoids} rather than indexed monoids.

\begin{definition}[Im-opfibration]\label{def:opfib-indexed-mon}
We say that an opfibration $\cat Y\rightarrow\cat X$ {\em has indexed monoids} if the base category $\cat X$ has indexed monoids, $\cat Y$ is cartesian monoidal, the functor preserves and reflects the cartesian products, and the cartesian product of opcartesian maps in $\cat Y$ is opcartesian. For the sake of brevity, we refer to an opfibration with indexed monoids as an {\em im-opfibration}.
\end{definition}
Note that in the above definition, $\cat X$ having indexed monoids implies that $\cat X$ is cartesian, so that it makes sense to ask the functor to preserve and reflect the cartesian products.

An im-opfibration is {\em split} if it is a split opfibration and the chosen cartesian products are preserved and reflected strictly.

\begin{definition}[Morphism of im-opfibrations]\label{def:morphism-im-opfib}
A {\em morphism of im-opfibrations} is a morphism of opfibrations $(H,K)$ such that the functor between the base categories $K$ is a morphism in $\IMon$, and the functor between the total categories $H$ preserves the cartesian monoidal structure.
\end{definition}
A morphism of split im-opfibrations is a morphism of split opfibrations such that both functors are strict monoidal. Let us denote the category of im-opfibrations by $\imOpFib$, and the fixed base case by $\imOpFib(\cat X)$. As for (op)fibrations, the subcategory of split im-opfibrations is denoted by adding the subscript $\mathsf{sp}$ in each case.

Given an im-opfibration $\cat Y\rightarrow\cat X$, define its {\em restriction} as the functor $\cat Y^*\rightarrow\hom(\cat X)$, where $\hom(\cat X)$ is the subcategory of monoid homomorphisms, and $\cat Y^*$ is the wide subcategory of $\cat Y^*$ defined by $F\in\cat Y^*$ if and only if the image of $F$ is in $\hom(\cat X)$, while the action of the functor is given by that of the original functor.

\begin{proposition}
For any im-opfibration $\cat Y\rightarrow\cat X$, its restriction $\cat Y^*\rightarrow\hom(\cat X)$ is an opfibration. Moreover, the restriction extends to a functor $\hom:\imOpFib\rightarrow\OpFib$.
\end{proposition}
\begin{proof}
The opcartesian maps are inherited by restriction: $F\in\cat Y^*$ is opcartesian if and only if $F\in\cat Y$ is opcartesian. The opcartesian liftings are likewise given by opcartesian liftings in $\cat Y$. The morphisms restrict to morphisms of opfibrations by Proposition~\ref{prop:imon-subcat-preserved}.
\end{proof}

\begin{theorem}\label{thm:monopfib-imopfib}
Let $\cat X$ be a small category. The category of split opfibrations with monoids with base $\cat X$ is equivalent to the category of split im-opfibrations with base $\Fim(\cat X)$:
$$\MonOpFib_{\mathsf{sp}}(\cat X)\simeq\imOpFib_{\mathsf{sp}}(\Fim(\cat X)).$$
\end{theorem}
\begin{proof}
Given a split opfibration $\cat Y\rightarrow\cat X$ with a monoid
\begin{equation*}
\scalebox{1}{\input{journal-figures/gr-pseudomonoid.tikz}},
\end{equation*}
define the split im-opfibration $\cat Y_{\otimes}\rightarrow\Fim(\cat X)$ as follows. The total category $\cat Y_{\otimes}$ is the term model of the following monoidal theory:
\begin{itemize}
\item the theory contains the monoidal theory with uniform comonoids generated by $\cat Y$,
\item for all $x\in\Ob(\cat X)$ and $a,b\in\cat Y_x$, we have the following additional generators:
\begin{equation*}
\scalebox{1}{\input{journal-figures/im-extension-generators.tikz}},
\end{equation*}
\item for all $a,b,c\in\Ob(\cat Y)$ in the same fibre, $(\sigma,\tau)\in\cat Y\boxtimes\cat Y$, and $f:x\rightarrow w$ in $\cat X$, the theory contains the following equations:
\begin{equation*}
\scalebox{1}{\input{journal-figures/im-extension-equations.tikz}}.
\end{equation*}
\end{itemize}
The action of the functor $\cat Y_{\otimes}\rightarrow \Fim(\cat X)$ is defined by extending $\cat Y\rightarrow\cat X$ to cartesian products, and by mapping the newly added generators to the freely added monoid multiplication and unit in the base. The opcartesian morphisms in $\cat Y_{\otimes}$ are generated as follows:
\begin{itemize}
\item if $f\in\cat Y$ is opcartesian, then so is $f\in\cat Y_{\otimes}$,
\item the newly added generators are opcartesian,
\item if $f$ and $g$ are opcartesian, then so are the composition $f;g$ (whenever defined) and the cartesian product $f\times g$.
\end{itemize}
It is then straightforward to construct the opcartesian liftings. Since the cartesian structure is preserved and reflected by construction, we indeed obtain an im-opfibration.

Conversely, given a split im-opfibration $\cat Y\rightarrow\Fim(\cat X)$, we proceed to define a monoid
\begin{equation*}
\scalebox{1}{\input{journal-figures/im-opfibration-to-monoid.tikz}}
\end{equation*}
on the split opfibration $\cat Y^*\rightarrow\cat X$, where $\cat Y^*$ is the wide subcategory of $\cat Y$ obtained by restriction, i.e.~$F\in\cat Y^*$ if and only if the image of $F$ is in $\cat X$. First, observe that, since the cartesian products are strictly preserved and reflected, for all morphisms $f,g\in\cat X$, we have $\cat Y_f\times\cat Y_g\simeq\cat Y_{f\times g}$, and the isomorphism is given (from left to right) by the action of the product functor on $\mathcal Y$. Further, each $f:x\rightarrow y$ in $\cat X$ is a monoid homomorphism by construction of $\Fim(\cat X)$, which induces a function $\cat Y_{f\times f}\rightarrow\cat Y_f$ by the universal property of the opcartesian liftings of the monoid multiplication on $x$. Hence, for every morphism $f\in\cat X$, we define the mapping $\otimes_f:\cat Y_f\times\cat Y_f\rightarrow\cat Y_f$ by composing the isomorphism $\cat Y_f\times\cat Y_f\simeq \cat Y_{f\times f}$ with the function $\cat Y_{f\times f}\rightarrow\cat Y_f$. We now define the functor $\otimes_{\cat Y}:\cat Y^*\boxtimes\cat Y^*\rightarrow\cat Y^*$ as the disjoint union of these mappings: $\otimes_{\cat Y}\coloneq\coprod_{f\in\cat X}\otimes_f$. Next, once again by strict preservation and reflection of cartesian products, we have that $\cat Y_{I}=\{I\}$, so that the reindexing functor $\one_x:\cat Y_{I}\rightarrow\cat Y_x$ induced by the monoidal unit on $x$ constitutes a choice of an object in $\cat Y_x$, which we denote by $\one(x)$. Moreover, the reindexing functor $f^*:\cat Y_x\rightarrow\cat Y_w$ induced by a morphism $f:x\rightarrow w$ in $\cat X$ preserves these chosen object in the sense that $f^*(\one(x))=\one(w)$. We thus define the unit functor $\one_{\cat Y}:\cat X\rightarrow\cat Y^*$ by $x\mapsto\one(x)$ on objects and by the opcartesian liftings on morphisms. Associativity and unitality are now induced by associativity and unitality of the monoids in $\Fim(\cat X)$, so that $(\cat Y^*,\otimes_{\cat Y},\one_{\cat Y})$ indeed defines a monoid in $(\OpFib_{\mathsf{sp}}(\cat X),\boxtimes,\id_{\cat X})$.

Applying the monoid construction to the im-opfibration obtained from a monoid recovers the original monoid: $\otimes_{\cat Y_{\otimes}}=\otimes$ and $\one_{\cat Y_{\otimes}}=\one$, since the restriction of $\cat Y_{\otimes}$ is $\cat Y$. In the other direction, we recover the original im-opfibration up to an isomorphism: $(\cat Y^*)_{\otimes}\simeq\cat Y$, as the objects and morphisms in $\cat Y$ that are not in $\cat Y^*$ are first removed and then added with potentially new labels.
\end{proof}

We observe that Theorem~\ref{thm:monopfib-imopfib} can be thought of as a kind of internalisation of the monoid on the opfibration. On the left-hand side of the equivalence, the monoid is an externally imposed structure, whereas on the right-hand side, the monoids are absorbed into the internal structure of the base and the total categories. In particular, given an im-opfibration $\cat Y\rightarrow\cat X$ and an object $x\in\Ob(\cat X)$, we denote by $\otimes_x : \cat Y_x\times\cat Y_x\rightarrow\cat Y_x$ the monoidal functor on the fibre $\cat Y_x$ obtained from $\otimes_f$ by setting $f=\id_x$ in the above proof. We refer to this monoidal structure as {\em fibrewise}. It is precisely this internalisation that will make opfibrations with indexed monoids convenient to work with when defining the semantics of layered theories in Section~\ref{subsec:opfib-models}.

Combining Theorem~\ref{thm:monopfib-imopfib} with Theorem~\ref{thm:monoids-opfibrations-strict-imoncat-equivalence}, we also have the following:
\begin{corollary}\label{cor:opfibrations-indexed-monoids-opindexed-monoidal}
Let $\cat X$ be a small category. The category of split im-opfibrations with base $\Fim(\cat X)$ is equivalent to the category of strict $\cat X$-opindexed monoidal categories:
$$\imOpFib_{\mathsf{sp}}\left(\Fim(\cat X)\right)\simeq\OpIMonCat_{\mathsf{st}}(\cat X).$$
\end{corollary}

The reader may wonder what happens in the general case, when the base of an im-opfibration is not freely generated. In such a case, each monoid in the base still induces some monoidal structure on the fibres, but now we are able to detect (or impose) properties of monoidal categories at the level of the base. For example, requiring the commutativity equation \scalebox{.7}{\input{journal-figures/monoid-commutative.tikz}} to strictly hold in the split case would enforce the equation $a\otimes b=b\otimes a$ in the corresponding fibre. In the general (non-split) case, this defines a natural isomorphism, i.e. a symmetric monoidal category. Thus, in order to properly exploit the additional equations in the base, one needs to move beyond the split (strict) setting. It is unclear whether an analogue of Theorem~\ref{thm:monopfib-imopfib} holds in this case, as it is not obvious what the equations between the (pseudo)monoids on opfibrations should be. We leave exploring these developments to future work.

\section{Deflations}\label{sec:deflations}
We now wish to impose additional structure on a functor $\cat Y\rightarrow\cat X$ so that it can be used as a model for layered theories with bidirectional functor boundaries. Proposition~\ref{prop:displayed-pseudo-decompose} suggests it should be a factorisation lifting, while Proposition~\ref{prop:displayed-factors-opfibration} suggests it should somehow restrict to an opfibration, which would give the forward functor boundaries. We ensure that it also restricts to a fibration, giving the reverse functor boundary. We do this by freely adding right adjoints to the base category $\cat X$ (Definition~\ref{def:zigzag-category}), and by requiring that the adjoints lift to the total category (Definition~\ref{def:local-retrofunctor}).

Let $\cat X$ be a category. The collection of {\em zigzags} on $\cat X$ is recursively defined as follows:
\begin{itemize}
\item for every morphism $f:x\rightarrow y$ in $\cat X$, the expressions $f:x\rightarrow y$ and $\overline f:y\rightarrow x$ are zigzags,
\item if $\varphi:x\rightarrow y$ and $\gamma:y\rightarrow z$ are zigzags, then so is $\varphi ;\gamma : x\rightarrow z$.
\end{itemize}
In other words, a zigzag is a (bracketed) list of morphisms with possibly alternating directions; for example, $x\rightarrow y\leftarrow z\leftarrow w\rightarrow s$ is a zigzag $x\rightarrow s$.

\begin{definition}[Zigzag 2-category]\label{def:zigzag-category}
Given a small 1-category $\cat X$, the {\em zigzag 2-category} $\Zg(\cat X)$ of $\cat X$ is the strict 2-category whose 0-cells are the objects of $\cat X$ and the 1-cells are given by the zigzags, subject to the following equations:
\begin{align*}
\overline{\id}_x &= \id_x & f;g &= fg &
(\varphi;\gamma);\psi &= \varphi;(\gamma;\psi) & \overline f;\overline g &= \overline{gf},
\end{align*}
where $;$ is used for composition of zigzags and concatenation is used for composition in $\cat X$. Given a morphism $f:x\rightarrow y$ in $\cat X$, we have the following generating 2-cells:
\begin{align*}
\eta_f &: \id_x\rightarrow f;\overline f &
\varepsilon_f &: \overline f;f\rightarrow\id_y,
\end{align*}
subject to the zigzag equations defining an adjunction:
\begin{equation*}
\scalebox{.9}{\input{journal-figures/zigzag-equations-twocells.tikz}},
\end{equation*}
as well as the following coherence equations:
\begin{align*}
\eta_{\id_x} &= \id_{\id_x}  & \eta_{f;h} &= \eta_f;(\id_f*\eta_h*\id_{\overline f}) &
\varepsilon_{\id_x} &= \id_{\id_x} & \varepsilon_{f;h} &= (\id_{\overline h}*\varepsilon_f*\id_h);\varepsilon_h.
\end{align*}
\end{definition}
For any category, there are the following identity-on-objects inclusions into its zigzag 2-category:
\begin{alignat*}{2}
\cat X &\xhookrightarrow{\ \iota^*\ }\Zg(\cat X) {} &&\xhookleftarrow{\ \iota^{\circ}\ }\cat X^{op} \\
f &\longmapsto f\quad\ \overline f {} &&\longmapsfrom f.
\end{alignat*}

\begin{definition}[Local retrofunctor]\label{def:local-retrofunctor}
Let $\cat X$ and $\cat Y$ be strict 2-categories. A {\em local retrofunctor} $(p,\varphi):\cat Y\rightarrow\cat X$ consists of a 1-functor $p:\cat Y\rightarrow\cat X$, and for all pairs of objects $a,b\in\Ob(\cat Y)$, a retrofunctor $(p,\varphi_{a,b}) : \cat Y(a,b)\rightarrow\cat X(pa,pb)$ whose function part is given by the action of $p$ on morphisms, satisfying the following coherence equation for all $F:a\rightarrow b$, $F':b\rightarrow c$, $\alpha:pF\rightarrow g$, $\alpha':pF'\rightarrow g'$:
\begin{equation*}
\varphi_{a,b}(F,\alpha)*\varphi_{b,c}(F',\alpha') = \varphi_{a,c}(F;F',\alpha *\alpha').
\end{equation*}
\end{definition}

We will now focus on local retrofunctors whose codomain is a zigzag 2-category, whose functor part is a factorisation lifting, together with a certain condition on the liftings of counits. This situation will play such a prominent role in the subsequent development that we give it a special name: {\em deflation}. We shall further introduce monoidal deflations, which will be used to give semantics to layered theories in Section~\ref{subsec:deflational-models}.

Let $(p,\varphi):\cat Y\rightarrow\Zg(\cat X)$ be a local retrofunctor such that $p:\cat Y\rightarrow\Zg(\cat X)$ is a factorisation lifting. A {\em lifting} of the opliftable pair $(a,f:x\rightarrow y)$ such that $f\in\cat X$ is any morphism $F:a\rightarrow b$ obtained by decomposing the codomain of the 2-cell $\varphi_{a,a}\left(\id_a,\eta_f\right)$ into $F:a\rightarrow b$ above $f:x\rightarrow y$ and $\overline F:b\rightarrow a$ above $\overline f:y\rightarrow x$:
\begin{equation*}
\scalebox{1}{\input{journal-figures/decomposition-opfibration-lifting.tikz}}.
\end{equation*}
Dually, we say that $\overline F:b\rightarrow a$ is a {\em lifting} of the liftable pair $(\overline f:y\rightarrow x,a)$. In the context of deflations (which we about to define), the morphism of an opliftable pair will always be assumed to lie in $\iota^*\cat X$, and similarly, the morphism of a liftable pair will always be assumed to lie in $\iota^{\circ}\cat X$. While {\em a priori} the same pair can have many liftings, as we shall see imminently (Lemma~\ref{lma:deflation-unique-liftings}), in our situation of interest the liftings are essentially unique.
\begin{definition}[Deflation]\label{def:deflation}
Let $\cat X$ be a 1-category and $\cat Y$ be a strict 2-category. A {\em deflation} from $\cat Y$ to $\cat X$ is a local retrofunctor $(p,\varphi):\cat Y\rightarrow\Zg(\cat X)$ such that $p:\cat Y\rightarrow\Zg(\cat X)$ is a factorisation lifting, and for every lifting $F:a\rightarrow b$ of an opliftable pair $(a,f:x\rightarrow y)$, the codomain of the 2-cell $\varphi_{b,b}\left(\overline F;F,\varepsilon_f\right)$ is $\id_b$:
\begin{equation*}
\scalebox{1}{\input{journal-figures/deflation-def-counit-iso.tikz}}.
\end{equation*}
\end{definition}
We say that a deflation is {\em cloven} or {\em split} if the underlying factorisation lifting is.

\begin{lemma}\label{lma:deflation-unique-liftings}
Let $(p,\varphi):\cat Y\rightarrow\Zg(\cat X)$ be a deflation, and let $F:a\rightarrow b$ be a lifting of an opliftable pair $(a,f:x\rightarrow y)$. Then for any morphism $F':a\rightarrow b'$ above $f$, there is a unique morphism $\alpha:b\rightarrow b'$ in the fibre $\cat Y_y$ such that $F;\alpha =F'$:
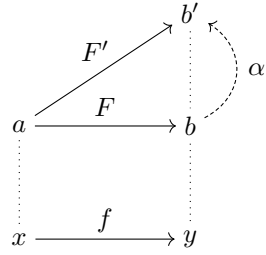
\begin{equation*}
\scalebox{1}{\input{journal-figures/deflation-unique-liftings.tikz}}.
\end{equation*}
\end{lemma}
\begin{proof}
Define $\alpha:b\rightarrow b'$ as the codomain of the 2-cell $\varphi_{b,b'}\left(\overline F;F',\varepsilon_f\right)$:
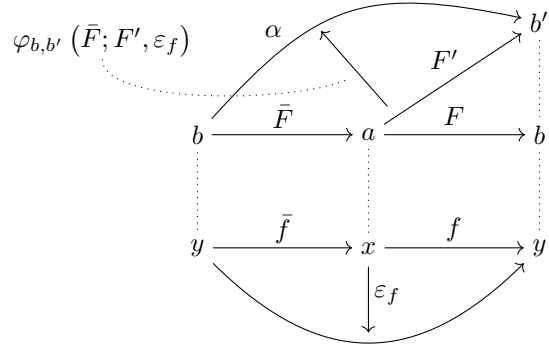
\begin{equation*}
\scalebox{1}{\input{journal-figures/unique-lifting-mediator.tikz}}.
\end{equation*}
We claim that this 1-cell has the desired property. To see this, consider the following diagram of 2-cells in $\cat Y$:
\begin{equation*}
\scalebox{1}{\input{journal-figures/unique-lifting-mediator-commutes.tikz}},
\end{equation*}
which simplifies to:
\begin{align*}
&\phantom{=} \left(\varphi_{a,a}\left(\id_a,\eta_f\right)*\id_{F'}\right);\left(\id_F*\varphi_{b,b'}\left(\overline F;F',\varepsilon_f\right)\right) \\
&= \left(\varphi_{a,a}\left(\id_a,\eta_f\right)*\varphi_{a,b'}\left(F',\id_f\right)\right);\left(\varphi_{a,b}\left(F,\id_f\right)*\varphi_{b,b'}\left(\overline F;F',\varepsilon_f\right)\right) \\
&= \varphi_{a,b'}\left(F',\eta_f*\id_f\right);\varphi_{a,b'}\left(F;\overline F;F',\id_f*\varepsilon_f\right) \\
&= \varphi_{a,b'}\left(F',\left(\eta_f*\id_f\right);\left(\id_f*\varepsilon_f\right)\right) \\
&= \varphi_{a,b'}\left(F',\id_f\right) \\
&= \id_{F'},
\end{align*}
which, in particular, implies that $F;\alpha=F'$, as required.

For uniqueness, suppose that $\alpha':b\rightarrow b'$ is some 1-cell such that $F;\alpha'=F'$. We then compute
\begin{align*}
\varphi_{b,b'}\left(\overline F;F',\varepsilon_f\right) &= \varphi_{b,b'}\left(\overline F;F;\alpha',\varepsilon_f*\id_{\id_y}\right) \\
&= \varphi_{b,b}\left(\overline F;F,\varepsilon_f\right)*\varphi_{b,b'}\left(\alpha',\id_{\id_y}\right) \\
&= \varphi_{b,b}\left(\overline F;F,\varepsilon_f\right)*\id_{\alpha'}.
\end{align*}
Since the codomain of $\varphi_{b,b}\left(\overline F;F,\varepsilon_f\right)$ is $\id_b$, this, in particular, implies that $\alpha'=\alpha$.
\end{proof}

\begin{corollary}\label{cor:decomposition-biretrofunctor-opfibration}
Let $(p,\varphi):\cat Y\rightarrow\Zg(\cat X)$ be a (split) deflation. Consider the restrictions of $p$
\begin{equation*}
\scalebox{1}{\input{journal-figures/decomposition-biretrofunctor-opfibration.tikz}}
\end{equation*}
to the wide subcategories $\cat Y^*$ and $\cat Y^{\circ}$ of $\cat Y$ defined by $F\in\cat Y^*$ if and only if $pF\in\iota^*\cat X$, and $F\in\cat Y^{\circ}$ if and only if $pF\in\iota^{\circ}\cat X^{op}$. Then $p^*$ is a (split) opfibration, and $p^{\circ}$ is a (split) fibration.
\end{corollary}
\begin{proof}
Lemma~\ref{lma:deflation-unique-liftings} establishes that $p^*$ is a preopfibration. It is also a factorisation lifting by restricting the liftings of factorisations under $p$. Thus, by Corollary~\ref{cor:preopfib-opfib-factlift}, $p^*$ is an opfibration. The argument establishing that $p^{\circ}$ is a fibration is dual.
\end{proof}

\begin{definition}[Morphism of deflations]\label{def:morphism-deflations}
Let $(p,\varphi):\cat Y\rightarrow\Zg(\cat X)$ and $(q,\varphi):\cat Y'\rightarrow\Zg(\cat X')$ be deflations. A {\em morphism} $(H,K):p\rightarrow q$ is given by a 2-functor $H:\cat Y\rightarrow\cat Y'$ and a 1-functor $K:\cat X\rightarrow\cat X'$ such that the diagram
\begin{equation*}
\scalebox{1}{\input{journal-figures/morphism-deflations.tikz}},
\end{equation*}
where $K$ is freely extended to the zigzag 2-categories, commutes, and for every $F:a\rightarrow b$ in $\cat Y$ above $f:x\rightarrow y$ in $\Zg(\cat X)$ and every 2-cell $\alpha:f\rightarrow g$ we have
$$H\left(\varphi_{a,b}(F,\alpha)\right)=\varphi_{Ha,Hb}(HF,K\alpha).$$
\end{definition}
A morphism of split deflations is additionally required to strictly preserve the chosen liftings of factorisations. We denote the category of deflations by $\Defl$, and the fixed base case by $\Defl(\cat X)$. As before, we add the subscript $\mathsf{sp}$ for the subcategories of split deflations.

\begin{definition}[Minimal deflation]\label{def:minimal-deflation}
We say that a deflation $(p,\varphi):\cat Y\rightarrow\Zg(\cat X)$ is {\em minimal} if the only non-identity 2-cells of $\cat Y$ are the ones in the image of the retrofunctors $(p,\varphi_{a,b})$.
\end{definition}
An equivalent way to define a minimal deflation would be to require $p$ to be a 2-functor, and each retrofunctor $(p,\varphi_{a,b}):\cat Y(a,b)\rightarrow\cat X(pa,pb)$ to be a {\em lens}, i.e.~the lifted 2-cells should also be compatible with the action of $p$ on 2-cells, not just the 1-cells. Minimal deflations are equivalent to (op)indexed categories (Theorem~\ref{thm:minimal-deflation-indexed-category}).

\begin{remark}\label{rem:minimal-deflation}
The reason we wish to consider non-minimal deflations are many scenarios in which there are transformations between morphisms which are non-trivial, in the sense that they do not arise as liftings of the 2-cells in the base. For example, in~\cite[5.3]{lmt-part1} we consider 2-cells that are semantics-preserving graph rewrites. Moreover, the syntax developed in Section~\ref{sec:layered-theories} is sound and complete with respect to all deflations, not just the minimal ones. The general deflations, therefore, seem like a more natural semantic domain for interpreting layered theories.
\end{remark}
Let us denote by $\MinDefl(\cat X)$ the full subcategory of $\Defl(\cat X)$ on minimal deflations.

\begin{theorem}\label{thm:minimal-deflation-indexed-category}
Let $\cat X$ be a small category. There is an equivalence of categories between minimal deflations into $\cat X$ and opindexed categories on $\cat X$:
$$\MinDefl(\cat X)\simeq\OpICat(\cat X),$$
which restricts to split minimal deflations and strict functors $\cat X\rightarrow\Cat$.
\end{theorem}
\begin{proof}
By Corollary~\ref{cor:decomposition-biretrofunctor-opfibration}, each deflation restricts to an opfibration into $\cat X$. By Proposition~\ref{prop:displayed-factors-opfibration}, the corresponding displayed category $\cat X\rightarrow\Prof$ factors as $\cat X\rightarrow\Cat\xrightarrow{\refine}\Prof$, thus producing the desired indexed category. It is indeed a pseudofunctor by Proposition~\ref{prop:displayed-pseudo-decompose}, since the original deflation is a factorisation lifting.

Conversely, any opindexed category $D:\cat X\rightarrow\Cat$ extends to a normal pseudofunctor $\hat D:\Zg(\cat X)\rightarrow\Prof$ given by the action of the original functor $D$ on the 0-cells, and by the following recursive definition on the 1-cells:
\begin{align*}
f &\mapsto (Df)^{\refine} &
\overline f &\mapsto (Df)^{\coarsen} &
\varphi;\gamma &\mapsto\hat D(\varphi);\hat D(\gamma),
\end{align*}
where $;$ on the right-hand side denotes the composition of profunctors. The generating 2-cells $\eta_f$ and $\varepsilon_f$ are, respectively, sent to the unit $\id_{Dx}\rightarrow (Df)^{\refine};(Df)^{\coarsen}$ and the counit $(Df)^{\coarsen};(Df)^{\refine}\rightarrow\id_{Dy}$ constructed in Proposition~\ref{prop:prof-embedding-adjoints}. The collage of $\hat D$ then gives the sought-after deflation $\coprod\hat D\rightarrow\Zg(\cat X)$, with the 2-cells between two parallel morphisms $(f,F)$ and $(g,G)$ of type $(x,a)\rightarrow (y,b)$ given by a 2-cell $\alpha:f\rightarrow g$ in $\Zg(\cat X)$ such that $\hat D(\alpha)_{a,b}(F)=G$.

The equivalence on the 1-categorical part then follows by restricting the equivalence of Theorem~\ref{thm:collage-functors} to deflations and indexed categories. Moreover, the 2-cells of the collage precisely recover the 2-cells of a minimal deflation, hence concluding the proof.
\end{proof}

\subsection{Monoidal deflations}

Next, we wish to incorporate monoidal structure into deflations. As for opfibrations, we do this via the means of indexed monoids. Recall that $\Fim(\cat X)$ denotes the free category with indexed monoids generated by a small category $\cat X$ (Definition~\ref{def:free-indexed-cat}). We denote the cartesian monoidal structure of $\Fim(\cat X)$ by $(\times,1)$. First, we observe that $\Zg(\Fim(\cat X))$ can be given a monoidal structure as follows.
\begin{proposition}\label{prop:zg-monoidal-structure}
Let $\cat X$ be a small category with a cartesian monoidal structure $(\times,1)$. The zigzag 2-category $\Zg(\cat X)$ becomes a strict monoidal 2-category upon imposing the following additional equation for all $f:x\rightarrow y$ and $g:z\rightarrow w$ in $\cat X$:
$$(f\times\id_w);\overline{(\id_y\times g)} = \overline{(g\times\id_x)};(\id_z\times f).$$
\end{proposition}
\begin{proof}
We define the monoidal functor $\times:\Zg(\cat X)\times\Zg(\cat X)\rightarrow\Zg(\cat X)$ on objects by the action of the product $\times$ on $\cat X$, while for all $f:x\rightarrow y$ and $g:z\rightarrow w$ in $\cat X$, we define the product on the generating zigzags as follows, where on the right-hand side $\times$ denotes the cartesian product of $\cat X$:
\begin{align*}
(f,g) &\mapsto f\times g, &
(\overline f,\overline g) &\mapsto \overline{f\times g}, &
(f,\overline g) &\mapsto (f\times\id_w);\overline{(\id_y\times g)}, &
(\overline f,g) &\mapsto \overline{(f\times\id_z)};(\id_x\times g).
\end{align*}
Finally, on the generating 2-cells we stipulate as follows, where, as before, $f:x\rightarrow y$ and $g:z\rightarrow w$ are in $\cat X$ and $\times$ on the right-hand side denotes the cartesian product of $\cat X$:
\begin{align*}
\eta_f\times\eta_g &\coloneq\eta_{f\times g}, &
\varepsilon_f\times\varepsilon_g &\coloneq\varepsilon_{f\times g}, &
\eta_f\times\varepsilon_g &\coloneq \varepsilon_{\id_x\times g};\eta_{f\times\id_w}, &
\varepsilon_f\times\eta_g &\coloneq \varepsilon_{f\times\id_z};\eta_{\id_y\times g}.
\end{align*}
The monoidal unit is given by that of $\cat X$, i.e.~by $1$.
\end{proof}
Whenever we write $\Zg(\Fim(\cat X))$, we assume that it has the monoidal structure as described in Proposition~\ref{prop:zg-monoidal-structure}, which we also denote by $(\times,1)$. Note that, despite notation, this monoidal structure is {\em not} cartesian monoidal.

\begin{definition}[Monoidal deflation]\label{def:monoidal-deflation}
A {\em monoidal deflation} is a deflation $(p,\varphi):\cat Y\rightarrow\Zg(\Fim(\cat X))$ such that $\cat Y$ is a strict monoidal 2-category with the monoidal structure $(\otimes,I)$ such that $p$ strictly preserves and reflects the monoidal structure:
\begin{itemize}
\item for all $a\in\Ob(\cat Y)$, we have $p(a)=1$ if and only if $a=I$,
\item for all $F,G\in\cat Y$, we have $p(F\otimes G)=pF\times pG$,
\item for all $H\in\cat Y$ and $f,g\in\Zg(\Fim(\cat X))$, if $pH=f\times g$, then there are $F,G\in\cat Y$ with $pF=f$, $pG=g$ and $F\otimes G=H$,
\end{itemize}
and the following additional equation holds for all 1-cells $F:a\rightarrow b$ and $G:c\rightarrow d$ in $\cat Y$ that, respectively, are above the domains of the 2-cells $\alpha$ and $\beta$ in $\Zg(\Fim(\cat X))$:
$$\varphi_{a\otimes c,b\otimes d}(F\otimes G,\alpha\times\beta)=\varphi_{a,b}(F,\alpha)\otimes\varphi_{c,d}(G,\beta).$$
\end{definition}
Note that, unlike in the definition of im-opfibrations (Definition~\ref{def:opfib-indexed-mon}), we do not require that the monoidal product on the total category of a monoidal deflation sends liftings to liftings in any sense. This, however, follows from other requirements, as we record in the following proposition.
\begin{proposition}\label{prop:monoidal-deflation-liftings-preserved}
Let $(p,\varphi):\cat Y\rightarrow\Zg(\Fim(\cat X))$ be a monoidal deflation. Let $F:a\rightarrow b$ and $G:c\rightarrow d$, respectively, be liftings of opliftable pairs $(a,f:x\rightarrow y)$ and $(c,g:z\rightarrow w)$. Then $F\otimes G$ is a lifting of $(a\otimes c,f\times g)$.
\end{proposition}
\begin{proof}
The lifting of $(a\otimes c,f\times g)$ is defined by factorising the codomain of the following 2-cell:
\begin{align*}
\varphi_{a\otimes c,a\otimes c}(\id_{a\otimes c},\eta_{f\times g}) &= \varphi_{a\otimes c,a\otimes c}(\id_{a}\otimes\id_{c},\eta_{f}\times\eta_{g}) \\
&= \varphi_{a,a}(\id_{a},\eta_{f})\otimes\varphi_{c,c}(\id_{c},\eta_{g}),
\end{align*}
which is therefore equal to $(F;\overline F)\otimes (G;\overline G) = (F\otimes G);(\overline F;\overline G)$. Thus, $F\otimes G$ is indeed a lifting.
\end{proof}
Since liftings in deflations are unique up to an isomorphism, Proposition~\ref{prop:monoidal-deflation-liftings-preserved} implies that in monoidal deflations liftings of products are given by products of liftings. We say that a monoidal deflation is {\em split} if the underlying deflation is split, and the chosen liftings of products are the products of chosen liftings.

A morphism of monoidal deflations is given by a morphism of deflations $(H,K)$ such that $H$ is additionally a strict monoidal functor. We denote the various categories of monoidal deflations by $\MonDefl$.

The following lemma establishes a tight link between monoidal deflations and opfibrations with indexed monoids (Definition~\ref{def:opfib-indexed-mon}).
\begin{lemma}\label{lma:monoidal-deflation-opfibration-indexed-monoids}
Let $(p,\varphi):\cat Y\rightarrow\Zg(\Fim(\cat X))$ be a split deflation. If $p$ is monoidal, then the opfibration $p^*$ obtained by restriction
\begin{equation*}
\scalebox{1}{\input{journal-figures/monoidal-deflation-restriction-indexed-monoids.tikz}}
\end{equation*}
has indexed monoids. Moreover, if $p$ is minimal, then the converse also holds: if the restriction $p^*$ has indexed monoids, then the deflation $p$ is monoidal.
\end{lemma}
\begin{proof}
First, suppose that the (not necessarily minimal) deflation $p$ is monoidal. Given objects $a,b\in\Ob(\cat Y)$ with $pa=x$, define the counit $C_a:a\rightarrow 1$ as the opcartesian lifting of the counit $(a,x\rightarrow 1)$, and the maps
\begin{equation*}
\scalebox{1}{\input{journal-figures/monoidal-deflation-restriction-indexed-monoids-lifted-maps.tikz}}
\end{equation*}
as the opcartesian liftings of the diagonal and the symmetry, respectively. Compositionality of the opcartesian liftings, together with preservation and reflection of the monoidal structure, then implies that $a'=a''=\hat a$ and $b=\hat b$, so that we may define the above maps as the diagonal and the symmetry. The equations for symmetries and uniform comonoids then follow, once again by compositionality of the opcartesian liftings. Thus, $\cat Y^*$ is cartesian monoidal by Proposition~\ref{prop:fox}, and the cartesian monoidal structure is preserved and reflected by assumption, establishing that $p^*$ has indexed monoids.

For the (partial) converse, suppose that $p$ is minimal and that $p^*$ has indexed monoids. Let us denote by $\times^*:\cat Y^*\times\cat Y^*\rightarrow\cat Y^*$ the cartesian product functor on $\cat Y^*$. We define the monoidal functor $\otimes:\cat Y\times\cat Y\rightarrow\cat Y$ on objects by the action of $\times^*$. For defining $\otimes$ on morphisms, observe that Lemma~\ref{lma:deflation-unique-liftings} (and its dual) implies that the morphisms in $\cat Y$ are generated by morphisms of the following form:
$$\alpha;F;\beta;\overline G,$$
where $F$ is a lifting of some $f\in\iota^*(\cat X)$ and $\overline G$ is a lifting of some $\overline g\in\iota^{\circ}(\cat X^{op})$, while $\alpha$ and $\beta$ are in the fibres of the domain and codomain of $f$. Absorbing $\alpha$ and $\beta$ into $F$, we conclude that the morphisms of $\cat Y$ are generated by morphisms of the form $F;\overline G$, where $F\in\cat Y^*$ and $G$ is, as before, a lifting of some $\overline g\in\iota^{\circ}(\cat X^{op})$. Hence, it suffices to define the monoidal product on these two classes of morphisms. Hence let $F,F'\in\cat Y^*$ and let $\overline G,\overline{G'}$ lifting of liftable pairs, with $F:a\rightarrow b$ and $\overline G:c\rightarrow d$. We define
\begin{align*}
F\otimes F' &\coloneq F\times^* F', &
F\otimes\overline G &\coloneq (F\times^*\id_c);\overline{(\id_b\times^* G)}, \\
\overline G\otimes\overline{G'} &\coloneq \overline{G\times^* G'}, &
\overline G\otimes F &\coloneq \overline{(G\times^*\id_a)};(\id_d\times^* F).
\end{align*}
The monoidal unit is given by that of $\cat Y^*$, and this monoidal structure is preserved and reflected by $p$. It remains to define $\otimes$ on the 2-cells. Since the deflation $p$ is minimal, it suffices to define $\otimes$ on the liftings of the 2-cells in $\Zg(\Fim(\cat X))$, which we do as follows, where $\alpha$ and $\beta$ are any 2-cells in $\Zg(\Fim(\cat X))$, while $F:a\rightarrow b$ and $G:c\rightarrow d$ are above their domains:
$$\varphi_{a,b}(F,\alpha)\otimes\varphi_{c,d}(G,\beta) \coloneq \varphi_{a\times^*c,b\times^*d}(F\otimes G,\alpha\times\beta).$$
\end{proof}
Note that Lemma~\ref{lma:monoidal-deflation-opfibration-indexed-monoids} implies that every split monoidal deflation has a fibrewise monoidal structure induced by the fibrewise monoidal structure of the split opfibration with indexed monoids obtained by restriction. Moreover, this fibrewise monoidal structure coincides with the fibrewise monoidal structure induced by the split fibration with indexed comonoids obtained by restriction (Corollary~\ref{cor:decomposition-biretrofunctor-opfibration}). As for im-opfibrations, we denote the fibrewise monoidal structure in the fibre of $x$ by $\otimes_x$.

The construction in the proof of Lemma~\ref{lma:monoidal-deflation-opfibration-indexed-monoids} establishes the following theorem, linking monoidal deflations and opfibrations with indexed monoids. By results of Section~\ref{subsec:opfibrations-indexed-monoids} (specifically~Theorem~\ref{thm:monopfib-imopfib} and Corollary~\ref{cor:opfibrations-indexed-monoids-opindexed-monoidal}), minimal monoidal deflations are equivalent to opfibrations with a monoid and strict opindexed monoidal categories (i.e.~functors) $\cat X\rightarrow\MonCat_{\mathsf{st}}$.
\begin{theorem}\label{thm:monoidal-deflations-indexed-monoidal}
Let $\cat X$ be a small category. The categories of minimal split monoidal deflations into $\cat X$ and split im-opfibrations into $\Fim(\cat X)$ are equivalent:
$$\MinMonDefl_{\mathsf{sp}}(\cat X)\simeq\imOpFib_{\mathsf{sp}}(\Fim(\cat X)).$$
\end{theorem}
\begin{proof}
Given a minimal split monoidal deflation $(p,\varphi):\cat Y\rightarrow\cat \Zg(\Fim(X))$, the corresponding split im-opfibration is given by the restriction $p^*:\cat Y^*\rightarrow\Fim(\cat X)$.

Conversely, given a split im-opfibration $p:\cat Y\rightarrow\Fim(\cat X)$, we extend it to a split deflation $(\hat p,\varphi):\hat{\cat Y}\rightarrow\cat \Zg(\Fim(X))$ as follows. The category $\hat{\cat Y}$ has the same objects as $\cat Y$, while the morphisms are defined as:
\begin{itemize}
\item each morphism of $\cat Y$ is a morphism of $\hat{\cat Y}$, so that $\cat Y$ is a subcategory of $\hat{\cat Y}$,
\item for each opcartesian $F\in\cat Y$, we add the morphism $\overline F$ in the opposite direction and let $p\overline F\coloneq\overline{pF}$,
\item we impose the following equations on the morphisms: $\overline{\id_a} = \id_a$ and $\overline F;\overline G = \overline{G;F}$.
\end{itemize}
The 2-cells are generated by adding exactly those 2-cells required by the local retrofunctor, i.e.~for every 2-cell $\alpha$ in $\Fim(\cat X)$ and every $F:a\rightarrow b$ in $\hat{\cat Y}$ above the domain of $\alpha$ we add a generating 2-cell $\varphi_{a,b}(F,\alpha)$, subject to the equations of a deflation. Now, $(\hat p,\varphi)$ is a minimal split deflation, whose restriction is the split im-opfibration $p$ we started with. Lemma~\ref{lma:monoidal-deflation-opfibration-indexed-monoids} guarantees that $(\hat p,\varphi)$ is monoidal.
\end{proof}

In light of Theorem~\ref{thm:monoidal-deflations-indexed-monoidal}, one might ask why go through all the developments of the current section just to obtain something we already defined in the previous section? Why not directly work with im-opfibrations and be content? The answer is twofold. First, presenting an (op)indexed monoidal category (equivalently, an im-opfibration) as a monoidal deflation explicitly introduces the functor boundaries in both directions, displaying the categories and the functors involved in a highly modular way. This makes deflations ideal for a purely syntactic treatment, which is the subject of the next section. Second, dropping the requirement of minimality, we do obtain a genuine generalisation of (op)indexed monoidal categories: the axiomatisation of the next section is sound and complete with respect to all monoidal deflations, not merely the minimal ones.

\section{Semantics of layered theories}\label{sec:semantics}
In this final section, we define the models of opfibrational, fibrational and deflation theories in full generality. In each case, we construct a free-forgetful adjunction with the category of models: opfibrations with indexed monoids, fibrations with indexed monoids and monoidal deflations.

\subsection{(Op)fibrational models}\label{subsec:opfib-models}

Here we consider models of layered theories consisting of a split im-opfibration $\cat Y\rightarrow\Fim(\cat X)$ (see section~\ref{subsec:opfibrations-indexed-monoids}) together with an interpretation of all types, terms and 2-terms of the theory.

\begin{definition}[Opfibrational model of a layered signature]\label{def:model-layered-signature}
An {\em opfibrational model} of a layered signature $(\Omega,\mathcal F,\M_{\omega})$ is a split im-opfibration $p:\cat Y\rightarrow\Fim(\cat X)$, together with the following data:
\begin{itemize}
\item a function $i:\Omega\rightarrow\Ob(\cat X)$,
\item for every $\omega,\tau\in\Omega$, a function $i_{\omega,\tau}:\mathcal F(\omega,\tau)\rightarrow\cat X(i\omega,i\tau)$,
\item for every $\omega\in\Omega$, a function $i^{\omega}:C_{\omega}\rightarrow\Ob(\cat Y_{i\omega})$,
\item for every $\omega\in\Omega$ and all $a,b\in C_{\omega}^*$, a function $i^{\omega}_{a,b}:\Sigma_{\omega}(a,b)\rightarrow\cat Y_{i\omega}\left(\left(i^{\omega}\right)^*a,\left(i^{\omega}\right)^*b\right)$,
\end{itemize}
where $(i^{\omega})^*:C_{\omega}^*\rightarrow\Ob(\cat Y_{i\omega})$ is defined by extending $i^{\omega}$ to the fibrewise monoidal structure of $p$.
\end{definition}
Note that the last two bullet points imply that for each $\omega\in\Omega$, the fibre $\cat Y_{i\omega}$ is a model of the monoidal signature $\M_{\omega}$ in the sense of Definition~\ref{def:model-monoidal-signature}. We denote a model of $(\Omega,\mathcal F,\M_{\omega})$ by $(p,i)$.

Given an opfibrational model $(p:\cat Y\rightarrow\Fim(\cat X),i)$ of $\mathcal L=(\Omega,\mathcal F,\M_{\omega})$, it induces a function $i:\Type_{\mathcal L}\rightarrow\Ob(\cat Y)$ recursively defined as follows:
\begin{align*}
\varepsilon : \varepsilon &\mapsto 1, &
\varepsilon : \omega &\mapsto I_{i\omega} &
a : \omega &\mapsto i^{\omega}(a) \\
f(A) : \tau &\mapsto i_{\omega,\tau}(f)^*\left(i(A:\omega)\right) &
AB : \omega &\mapsto i(A:\omega)\otimes_{i\omega} i(B:\omega), &
T,S &\mapsto i(T)\times i(S),
\end{align*}
where $i_{\omega,\tau}(f)^*:\cat Y_{i\omega}\rightarrow\cat Y_{i\tau}$ is the reindexing functor induced by $i_{\omega,\tau}(f):i\omega\rightarrow i\tau$.

Likewise, the same opfibrational model further induces the function $i:\Term^1_{\mathcal L}\rightarrow\cat Y$ form the basic (Figure~\ref{fig:layered-terms}), symmetry (Definition~\ref{def:symmetry-terms}) and opfibrational (Figure~\ref{fig:opfibrational-terms}) terms as follows:
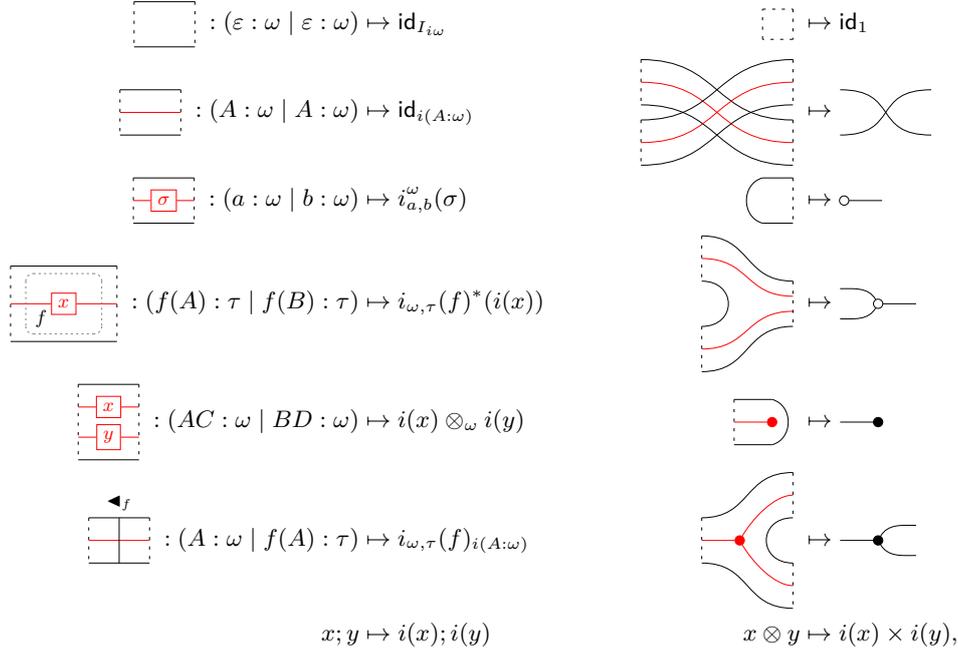
\begingroup
\small
\begin{align*}
\scalebox{.8}{\input{journal-figures/emptydiag-sheet.tikz}} : (\varepsilon :\omega\mid\varepsilon :\omega) &\mapsto \id_{I_{i\omega}} & \scalebox{.8}{\input{journal-figures/emptydiag.tikz}} &\mapsto \id_1 \\
\scalebox{.8}{\input{journal-figures/iddiag-sheet.tikz}} : (A:\omega\mid A:\omega) &\mapsto \id_{i(A:\omega)} & \scalebox{.8}{\input{journal-figures/symdiag-sheet1.tikz}} &\mapsto \scalebox{.8}{\input{journal-figures/symmetry.tikz}} \\
\scalebox{.8}{\input{journal-figures/internalsigmadiag.tikz}} : (a:\omega\mid b:\omega) &\mapsto i^{\omega}_{a,b}(\sigma) & \scalebox{.8}{\input{journal-figures/cup.tikz}} &\mapsto \scalebox{.8}{\begin{tikzpicture}
	\begin{pgfonlayer}{nodelayer}
		\node [style=wnode] (0) at (-0.5, 0) {};
		\node [style=none] (1) at (0.75, 0) {};
	\end{pgfonlayer}
	\begin{pgfonlayer}{edgelayer}
		\draw [style=edge] (0) to (1.center);
	\end{pgfonlayer}
\end{tikzpicture}
} \\
\scalebox{.8}{\input{journal-figures/f-box.tikz}} : (f(A):\tau\mid f(B):\tau) &\mapsto i_{\omega,\tau}(f)^*(i(x)) & \scalebox{.8}{\input{journal-figures/pants.tikz}} &\mapsto \scalebox{.8}{\input{journal-figures/monoid-multiplication.tikz}} \\
\scalebox{.8}{\input{journal-figures/internalxydiag.tikz}} : (AC:\omega\mid BD:\omega) &\mapsto i(x)\otimes_{\omega} i(y) & \scalebox{.8}{\input{journal-figures/a-cap.tikz}} &\mapsto \scalebox{.8}{\begin{tikzpicture}
	\begin{pgfonlayer}{nodelayer}
		\node [style=node] (0) at (0.75, 0) {};
		\node [style=none] (1) at (-0.5, 0) {};
	\end{pgfonlayer}
	\begin{pgfonlayer}{edgelayer}
		\draw [style=edge] (0) to (1.center);
	\end{pgfonlayer}
\end{tikzpicture}
} \\
\scalebox{.8}{\input{journal-figures/refine-sheet.tikz}} : (A:\omega\mid f(A):\tau) &\mapsto i_{\omega,\tau}(f)_{i(A:\omega)} & \scalebox{.8}{\input{journal-figures/copants-copy.tikz}} &\mapsto \scalebox{.8}{\input{journal-figures/comonoid-comultiplication.tikz}} \\
x;y &\mapsto i(x);i(y) & x\otimes y &\mapsto i(x)\times i(y), \\
\end{align*}
\endgroup
where $i_{\omega,\tau}(f)^*:\cat Y_{i\omega}\rightarrow\cat Y_{i\tau}$ is the reindexing functor induced by $i_{\omega,\tau}(f):i\omega\rightarrow i\tau$, and
$$i_{\omega,\tau}(f)_{i(A:\omega)}:i(A:\omega)\rightarrow i(f(A):\tau)$$
is the opcartesian lifting of the opliftable pair $\left(i(A:\omega), i_{\omega,\tau}(f):i\omega\rightarrow i\tau\right)$.

\begin{definition}\label{def:model-opfib-thy}
A {\em model of an opfibrational layered theory} with signature $\mathcal L$ is an opfibrational model $(p,i)$ of $\mathcal L$ such that the induced functions $i:\Type_{\mathcal L}\rightarrow\Ob(\cat Y)$ and $i:\Term^1_{\mathcal L}\rightarrow\cat Y$ preserve the 0- and the 1-equations, respectively.
\end{definition}

\begin{proposition}\label{prop:opfib-equations-preserved}
Any opfibrational model of a layered signature preserves the structural opfibrational equations (Definition~\ref{def:str-opfib-eqns}).
\end{proposition}
\begin{proof}
The structural 0-equations (Definition~\ref{def:str-0eqns}) are preserved by associativity and unitality of the monoids, and by the fact that the reindexing functors preserve the monoidal multiplication and units.

The structural 1-equations (Definition~\ref{def:str-1eqns}) are preserved since $\cat Y$ is (cartesian) monoidal, each fibre $\cat Y_{i\omega}$ is a monoidal category, and $\otimes_x:\cat Y_x\times\cat Y_x\rightarrow\cat Y_x$ is a functor.

The structural 1-equations in Figure~\ref{fig:structural-twocells-functors-int} are preserved since the reindexing functors between the fibres are strictly functorial and monoidal.

The symmetry equations are preserved since $\cat Y$ is cartesian monoidal, hence in particular symmetric monoidal. Likewise, the equations of uniform comonoids are preserved since $\cat Y$ is cartesian monoidal.

The equations in Figure~\ref{fig:structural-twocells-monoidal} are preserved by lifting the corresponding equations from the base category.

Preservation of the equations in Figure~\ref{fig:structural-twocells-functors-ext} is somewhat less immediate. For the first equation, assuming the sorts $x:(A:\omega\mid B:\omega)$ and $f\in\mathcal F(\omega,\tau)$ we have to show that
$$i(x);i_{\omega,\tau}(f)_{i(B:\omega)} = i_{\omega,\tau}(f)_{i(A:\omega)};i_{\omega,\tau}(f)^*(i(x)),$$
which holds by the definition of $i_{\omega,\tau}(f)^*(i(x))$: it is the unique map induced by the lifting property of the opcartesian map $i_{\omega,\tau}(f)_{i(A:\omega)}$ that makes the square
\begin{equation*}
\scalebox{1}{\input{journal-figures/interpretation-lifting-x.tikz}}
\end{equation*}
commute. The remaining two equations in Figure~\ref{fig:structural-twocells-monoidal} are preserved by Theorem~\ref{thm:monopfib-imopfib}.
\end{proof}

\begin{definition}\label{def:category-models-opfibrational-theories}
The category of {\em models of opfibrational theories} $\OpFThMod$ has as objects tuples $(\mathcal T,p,i)$ of an opfibrational theory $\mathcal T$ and its model $p:\cat Y\rightarrow\Fim(\cat X)$ (with the corresponding family of interpretation functions $i$). A morphism
$$(F,P,Q) : (\mathcal T,p:\cat Y\rightarrow\cat X,i)\rightarrow (\mathcal K,q:\cat Y'\rightarrow\cat X',j)$$
is given by a morphism of opfibrational theories $F:\mathcal T\rightarrow\mathcal K$ together with a morphism of im-opfibrations $(P,Q):p\rightarrow q$ such that the diagrams below commute, where we assume that the layered signature of $\mathcal T$ is $(\Omega,\mathcal F,\M_{\omega})$ and the layered signature of $\mathcal K$ is $(\Psi,\mathcal G,\M_{\psi})$:
\begin{equation*}
\scalebox{.9}{\input{journal-figures/mor-lay-mod.tikz}}.
\end{equation*}
Note that the bottom diagrams say that for every $\omega\in\Omega$, the restriction of $P:\cat Y\rightarrow\cat Y'$ to the fibre $P_{i\omega}:\mathcal Y_{i\omega}\rightarrow\mathcal Y'_{jF(\omega)}$ makes $(F^{\omega},P_{i\omega})$ into a morphism in $\MMod$.
\end{definition}

We denote the full subcategory of $\OpFThMod$ defined by the theories whose sets of equations are empty by $\OpFMod$, and call it the category of {\em opfibrational models of layered signatures}, since models of empty theories are in one-to-one correspondence with models of signatures. We summarise the relationship between opfibrational theories, models and layered signatures in the following proposition (cf.~Proposition~\ref{prop:monoidal-signatures-theories-models}).

\begin{proposition}\label{prop:opfibrational-signatures-theories-models}
The vertical forgetful functors in the diagram below are fibrations, while the horizontal forgetful functors form a morphism of fibrations:
\begin{equation*}
\scalebox{1}{\input{journal-figures/opfibrational-signatures-theories-models.tikz}}.
\end{equation*}
Moreover, the objects in the fibre $\OpFMod(\mathcal L)$ are precisely the opfibrational models of the layered signature $\mathcal L$, and the objects in the fibre $\OpFThMod(\mathcal T)$ are precisely the models of the opfibrational theory $\mathcal T$.
\end{proposition}
Similarly to Proposition~\ref{prop:monoidal-signatures-theories-models}, we note that $\OpFTh\rightarrow\LSgn$ is also a fibration, while the functor $\OpFThMod\rightarrow\OpFMod$ fails to be a fibration.

We conclude the discussion of opfibrational models by constructing free models generated by opfibrational theories, i.e.~by constructing the left adjoint to the vertical forgetful functors of Proposition~\ref{prop:opfibrational-signatures-theories-models}. We first define the total and the base categories generated by a theory and construct the im-opfibration between them.
\begin{definition}[Total category generated by an opfibrational theory]
Let $\mathcal T=(\mathcal L,E^0,E^1)$ be an opfibrational theory. The {\em total category} $T(\mathcal T)$ generated by $\mathcal T$ has the equivalence classes of types $\Type_{\mathcal L}$ under the type congruence generated by $E^0\cup S^0$ as objects, and the morphisms are the equivalence classes of terms $\Term^1_{\mathcal L}$ under the term congruence generated by $E^1\cup S_{\mathsf{opf}}^1$.
\end{definition}

\begin{definition}[Base category generated by a layered signature]
Let $\mathcal L=(\Omega,\mathcal F,\M_{\omega})$ be a layered signature. The {\em base category} $B(\mathcal L)$ generated by $\mathcal L$ is the free category with indexed monoids generated by $(\Omega,\mathcal F)$ viewed as a monoidal signature.
\end{definition}

Given an opfibrational theory $\mathcal T=(\mathcal L,E^0,E^1)$ with $\mathcal L=(\Omega,\mathcal F,\M_{\omega})$, define the functor $p_{\mathcal T}:T(\mathcal T)\rightarrow B(\mathcal L)$ on internal types by $A:\omega\mapsto\omega$, on all types by extension to products, and by action on the generating terms as follows, where $x:(A:\omega\mid B:\omega)$ is any internal term:
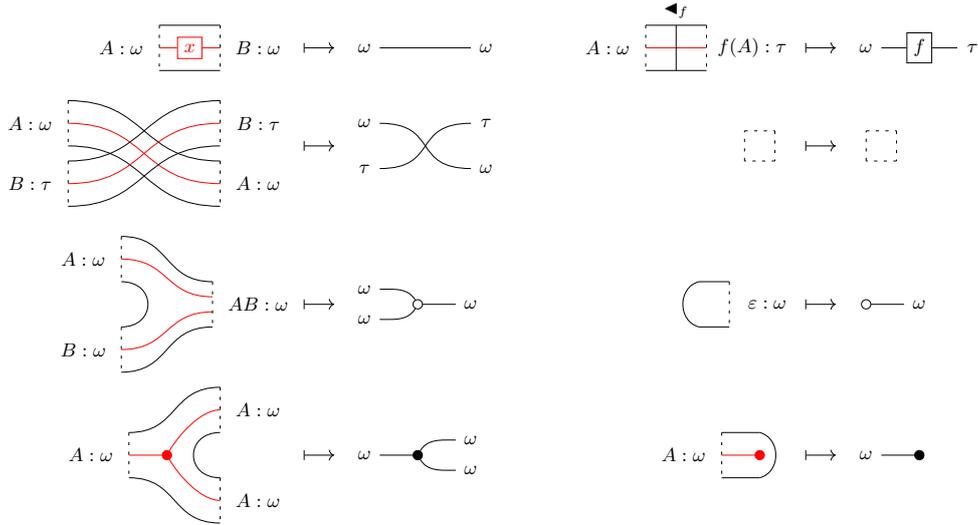
\begin{equation*}\label{page:total-to-base-opf}
\small\scalebox{.8}{\input{journal-figures/total-to-base.tikz}}.
\end{equation*}

\begin{proposition}\label{prop:syntactic-opfibrational-model}
For any opfibrational theory $\mathcal T$ with the layered signature $\mathcal L$, the functor $p_{\mathcal T}:T(\mathcal T)\rightarrow B(\mathcal L)$ is a split im-opfibration.
\end{proposition}
\begin{proof}
We observe that the opcartesian morphisms in $T(\mathcal T)$ are given by those terms that are equal to a term with no non-trivial internal structure. More precisely, a morphism is opcartesian if and only if the equivalence class contains a term whose construction does not involve the rule~\ref{term:int-gen}. We now observe that if we set the internal term $x$ to be the identity term in the definition of $p_{\mathcal T}:T(\mathcal T)\rightarrow B(\mathcal L)$, then for any type $S$ and a generating morphism $f:p_{\mathcal T}(S)\rightarrow\omega$, there is a unique term $\hat f:S\rightarrow K$ in $T(\mathcal T)$ with $M_{\mathcal T}\left(\hat f\right)=f$, which is moreover opcartesian. Extending this to all morphisms gives the split opfibration structure of $p_{\mathcal T}$, which is, moreover, consistent with the (cartesian) monoidal structure of $T(\mathcal T)$. Finally, the opfibration $p_{\mathcal T}$ has indexed monoids by construction.
\end{proof}
\begin{corollary}
Any opfibrational theory $\mathcal T=(\mathcal L,E^0,E^1)$ gives rise to the opfibrational model $\left(p_{\mathcal T}:T(\mathcal T)\rightarrow B(\mathcal L),i\right)$, whose interpretation functions $i$ are given by:
\begin{align*}
i(\omega) &=\omega, &
i_{\omega,\tau}(f) &=f, &
i^{\omega}(a) &= a:\omega, &
i^{\omega}_{A,B}(\sigma) &= \sigma : (A:\omega\mid B:\omega).
\end{align*}
\end{corollary}

We have finally arrived to the main theorem about opfibrational theories.
\begin{theorem}\label{thm:opfibrational-adjoint}
The construction of opfibrational models from opfibrational theories extends to a functor $\OpFTh\rightarrow\OpFThMod$, which is moreover the left adjoint to the forgetful functor.
\end{theorem}
\begin{proof}
Any morphism of opfibrational theories $F:\mathcal T\rightarrow\mathcal K$ recursively extends to types and terms, as shown in Section~\ref{sec:layered-theories}, giving the functor between the total categories $T(\mathcal T)\rightarrow T(\mathcal K)$. The induced functor between the base categories $B(\mathcal T)\rightarrow B(\mathcal K)$ is given by extending the action of $F$ on the layers and the generators to a morphism of indexed monoids. This gives a morphism of im-opfibrations.

By Proposition~\ref{prop:opfib-equations-preserved}, the term model of $\mathcal T$ is the initial object of the fibre $\OpFThMod(\mathcal T)$. Moreover, any morphism out of the term model of $\mathcal T$ into any model of $\mathcal K$ factors through the term model of $\mathcal K$, thus exhibiting the desired adjunction.
\end{proof}
\begin{corollary}\label{cor:opfibrational-completeness}
For any opfibrational theory $\mathcal T$ and layered signature $\mathcal L$, we have the following equivalences:
\begin{align*}
\OpFThMod(\mathcal T) &\simeq \quot{F(\mathcal T)}{\MonOpFib_{\mathsf{sp}}}, \\
\OpFMod(\mathcal L) &\simeq \quot{F(\mathcal L)}{\MonOpFib_{\mathsf{sp}}}.
\end{align*}
\end{corollary}

The class of models characterised by fibrational theories are {\em fibrations with indexed comonoids} $\cat Y\rightarrow\Fim(\cat X)^{op}$ obtained by appropriately dualising Definition~\ref{def:opfib-indexed-mon}. The free fibrational models are obtained by imposing the structural fibrational equations (Definition~\ref{def:str-fib-eqns}) on the fibrational terms in Figure~\ref{fig:fibrational-terms}.

\subsection{Deflational models}\label{subsec:deflational-models}

As anticipated, the deflational models are given by split monoidal deflations (Definition~\ref{def:monoidal-deflation}).

\begin{definition}[Deflational model]\label{def:deflational-model}
A {\em deflational model} of a layered signature $(\Omega,\mathcal F,\mathcal M_{\omega})$ is a split monoidal deflation $(p,\varphi):\cat Y\rightarrow\Zg(\Fim(\cat X))$ together with the following data:
\begin{itemize}
\item a function $i:\Omega\rightarrow\Ob(\cat X)$,
\item for every $\omega,\tau\in\Omega$, a function $i_{\omega,\tau}:\mathcal F(\omega,\tau)\rightarrow\cat X(i\omega,i\tau)$,
\item for every $\omega\in\Omega$, a function $i^{\omega}:C_{\omega}\rightarrow\Ob(\cat Y_{i\omega})$,
\item for every $\omega\in\Omega$ and all $a,b\in C_{\omega}^*$, a function $i^{\omega}_{a,b}:\Sigma_{\omega}(a,b)\rightarrow\cat Y_{i\omega}\left(\left(i^{\omega}\right)^*a,\left(i^{\omega}\right)^*b\right)$,
where $(i^{\omega})^*:C_{\omega}^*\rightarrow\Ob(\cat Y_{i\omega})$ is defined by extending $i^{\omega}$ to the fibrewise monoidal structure of $p$.
\end{itemize}
\end{definition}
Note that a deflational model on $(p,\varphi):\cat Y\rightarrow\Zg(\Fim(\cat X))$ is at the same time an opfibrational model on the restriction $p^*:\cat Y^*\rightarrow\Fim(\cat X)$ {\em and} a fibrational model on the restriction $p^{\circ}:\cat Y^{\circ}\rightarrow\Fim(\cat X)^{op}$. We denote a model of $(\Omega,\mathcal F,\mathcal M_{\omega})$ by $(p,\varphi,i)$.

Since any deflational model $(p,\varphi,i)$ of $\mathcal L$ restricts to an opfibrational model, it induces the function $i:\Type_{\mathcal L}\rightarrow\Ob(\cat Y)$, as $\cat Y^*$ and $\cat Y$ have the same objects. Moreover, this function is equal to the function induced by the restricted fibrational model.

The opfibrational restriction $p^*:\cat Y^*\rightarrow\Fim(\cat X)$ is likewise used to obtain the function from the basic (Figure~\ref{fig:layered-terms}), symmetry (Definition~\ref{def:symmetry-terms}) and opfibrational (Figure~\ref{fig:opfibrational-terms}) terms into $\cat Y^*$, as described in Section~\ref{subsec:opfib-models}. In order to obtain the function $i:\Term_{\mathcal L}^1\rightarrow\cat Y$, it remains to interpret the fibrational terms (Figure~\ref{fig:fibrational-terms}), as well as to extend the interpretation to composite terms. Hence, let each fibrational term be interpreted as the dual of the corresponding opfibrational term, while the composition and product are interpreted as composition and the (global) monoidal product in $\cat Y$. Explicitly, we stipulate:
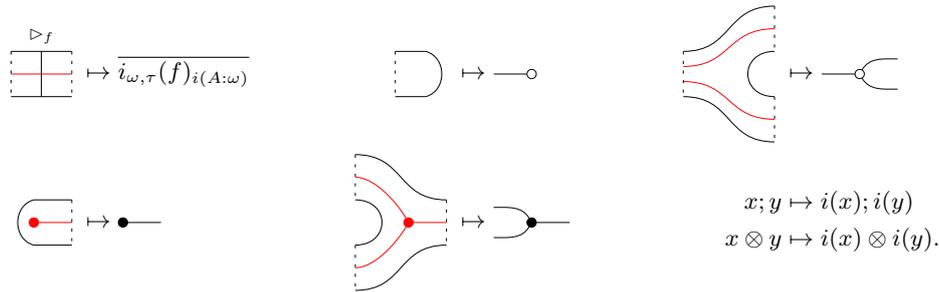
\begingroup
\small
\begin{align*}
\scalebox{.8}{\input{journal-figures/coarsen-sheet.tikz}} &\mapsto \overline{i_{\omega,\tau}(f)_{i(A:\omega)}} &
\scalebox{.8}{\input{journal-figures/cap.tikz}} &\mapsto \scalebox{.8}{\begin{tikzpicture}
	\begin{pgfonlayer}{nodelayer}
		\node [style=wnode] (0) at (0.75, 0) {};
		\node [style=none] (1) at (-0.5, 0) {};
	\end{pgfonlayer}
	\begin{pgfonlayer}{edgelayer}
		\draw [style=edge] (0) to (1.center);
	\end{pgfonlayer}
\end{tikzpicture}
} &
\scalebox{.8}{\input{journal-figures/copants.tikz}} &\mapsto \scalebox{.8}{\input{journal-figures/monoid-multiplication-op.tikz}} \\
\scalebox{.8}{\input{journal-figures/a-cup.tikz}} &\mapsto \scalebox{.8}{\begin{tikzpicture}
	\begin{pgfonlayer}{nodelayer}
		\node [style=node] (0) at (-0.75, 0) {};
		\node [style=none] (1) at (0.5, 0) {};
	\end{pgfonlayer}
	\begin{pgfonlayer}{edgelayer}
		\draw [style=edge] (0) to (1.center);
	\end{pgfonlayer}
\end{tikzpicture}
} &
\scalebox{.8}{\input{journal-figures/pants-copy.tikz}} &\mapsto \scalebox{.8}{\input{journal-figures/comonoid-comultiplication-op.tikz}} &
\begin{split}
x;y &\mapsto i(x);i(y) \\
x\otimes y &\mapsto i(x)\otimes i(y).
\end{split}
\end{align*}
\endgroup

\begin{definition}\label{def:deflational-theory-model}
A {\em model of a deflational layered theory} $(\mathcal L,E^0,E^1,\eta,E^2)$ consists of: (1) a deflational model $(p,\varphi,i)$ of $\mathcal L$ (Definition~\ref{def:deflational-model}) such that the induced functions $i:\Type_{\mathcal L}\rightarrow\Ob(\cat Y)$ and $i:\Term_{\mathcal L}^1\rightarrow\cat Y_1$ preserve $E^0$ and $E^1$, and (2) for each parallel pair of terms $(x,y)\in P^1_{\mathcal L}$ with sort $(T\mid S)$ a function
$$i_{x,y}:\eta(x,y)\sqcup\eta_{\mathsf{str}}(x,y)\rightarrow\cat Y_2\left(i(x),i(y)\right)$$
from the generating 2-cells into the 2-cells $i(x)\rightarrow i(y)$, such that the structural 2-cells are mapped to the following 2-cells returned by the retrofunctor:
\begin{align*}
&\eta_x\mapsto\varphi_{i(T),i(T)}\left(\id_{i(T)},\eta_{i(x)}\right) & &\varepsilon_x\mapsto\varphi_{i(S),i(S)}\left(i(\bar x);i(x),\varepsilon_{i(x)}\right),
\end{align*}
and the function $i:\Term^2_{\mathcal L}\rightarrow\cat Y_2$ on all 2-terms (Definition~\ref{def:2terms}) recursively obtained from $\eta\sqcup\eta_{\mathsf{str}}$ preserves $E^2$.
\end{definition}

\begin{proposition}\label{prop:deflational-model-preserves-structural}
Any deflational model of a deflational layered theory preserves the structural deflational equations (Definition~\ref{def:str-defl-eqns}).
\end{proposition}
\begin{proof}
The structural opfibrational and fibrational equations are preserved by Proposition~\ref{prop:opfib-equations-preserved} (and its dual). The 2-equations are preserved since the local retrofunctor preserves both vertical and composition, and therefore equations between 2-cells.
\end{proof}

\begin{definition}\label{def:category-deflational-models}
The category of {\em models of deflational theories} $\DeflThMod$ has as objects tuples $(\mathcal T,p,\varphi,i)$ of a deflational theory $\mathcal T$ and its model $(p,\varphi):\cat Y\rightarrow\Zg(\Fim(\cat X))$ with the family of interpretation functions $i$. A morphism
$$(F,F^2,P,Q) : (\mathcal T,p,\varphi,i)\rightarrow (\mathcal K,q,\rho,j),$$
is given by a morphism of deflational theories $(F,F^2):\mathcal T\rightarrow\mathcal K$ together with a morphism of monoidal deflations $(P,Q):(p,\varphi)\rightarrow (q,\rho)$ such that the diagrams of Definition~\ref{def:category-models-opfibrational-theories} as well as the diagrams below commute, where $(p,\varphi):\cat Y\rightarrow\Zg(\Fim(\cat X))$ and $(q,\rho):\cat Y'\rightarrow\Zg(\Fim(\cat X'))$, and we write $\eta_{\mathcal T}$ and $\eta_{\mathcal K}$ for the choice of 2-cells in $\mathcal T$ and $\mathcal K$, and $(\Omega,\mathcal F,\mathcal M_{\omega})$ for the layered signature of $\mathcal T$ and $(\Psi,\mathcal G,\mathcal M_{\psi})$ for the layered signature of $\mathcal K$:
\begin{equation*}
\scalebox{1}{\input{journal-figures/category-deflational-models-diagram.tikz}},
\end{equation*}
where $F:\Term^1_{\mathcal L}\rightarrow\Term^1_{\mathcal K}$, $i:\Term^1\rightarrow\cat Y_1$ and $j:\Term^1_{\mathcal K}\rightarrow\cat Y'_1$ are the induced map, and $Q^2$ denotes the action of $Q$ on 2-cells.
\end{definition}

As for monoidal and opfibrational theories, we denote the full subcategory of $\DeflThMod$ of {\em deflational models of layered signatures} defined by the theories whose sets of equations and 2-cells are empty by $\DeflMod$. As in Propositions~\ref{prop:monoidal-signatures-theories-models} and~\ref{prop:opfibrational-signatures-theories-models}, below we summarise the relationship between theories, models and signatures.

\begin{proposition}\label{prop:deflational-signatures-theories-models}
The vertical forgetful functors in the diagram below are fibrations, while the horizontal forgetful functors form a morphism of fibrations:
\begin{equation*}
\scalebox{1}{\input{journal-figures/deflational-signatures-theories-models.tikz}}.
\end{equation*}
Moreover, the objects in the fibre $\DeflMod(\mathcal L)$ are precisely the deflational models of the layered signature $\mathcal L$, and the objects in the fibre $\DeflThMod(\mathcal T)$ are precisely the models of the deflational theory $\mathcal T$.
\end{proposition}
As in Propositions~\ref{prop:monoidal-signatures-theories-models} and~\ref{prop:opfibrational-signatures-theories-models}, we note that $\DeflTh\rightarrow\LSgn$ is also a fibration, while the functor $\DeflThMod\rightarrow\DeflMod$ fails to be a fibration.

We now turn to the construction of free deflational models. As one might anticipate from the preceding development, the free models are given by quotienting the types, terms and 2-terms by the structural equations, and any free deflational model restricts to a free opfibrational model and a free fibrational model.
\begin{definition}[Total category generated by a deflational theory]
Let $\mathcal T=(\mathcal L,E^0,E^1,\eta,E^2)$ be a deflational theory. The {\em deflational total category} $T_{\mathsf{defl}}(\mathcal T)$ generated by $\mathcal T$ has the equivalence classes of types $\Type_{\mathcal L}$ under the type congruence generated by $E^0\cup S^0$ as the 0-cells, the equivalence classes of terms $\Term^1_{\mathcal L}$ under the term congruence $E^1\cup S_{\mathsf{defl}}^1$ as the 1-cells, and the equivalence classes of terms $\Term^2_{\mathcal L}$ generated by $\eta\sqcup\eta_{\mathsf{str}}$ under the term congruence $E^2\cup S_{\mathsf{defl}}^2$ as the 2-cells.
\end{definition}

\begin{definition}[Zigzag base category generated by a layered signature]
Let $\mathcal L=(\Omega,\mathcal F,\M_{\omega})$ be a layered signature. The {\em zigzag base category} $\Zg(B(\mathcal L))$ generated by $\mathcal L$ is the zigzag 2-category of the free category with indexed monoids generated by $(\Omega,\mathcal F)$ viewed as a monoidal signature.
\end{definition}

Given a deflational theory $\mathcal T=(\mathcal L,E^0,E^1,\eta,E^2)$ with $\mathcal L=(\Omega,\mathcal F,\M_{\omega})$, define the functor $p_{\mathcal T}:T_{\mathsf{defl}}(\mathcal T)\rightarrow \Zg(B(\mathcal L))$ on objects by letting $A:\omega\mapsto\omega$ on internal types, and on all types by extension to monoidal products. On morphisms, $p$ is defined by mapping the basic, symmetry and opfibrational terms exactly as in the opfibrational case (see page~\pageref{page:total-to-base-opf}), while each fibrational term is mapped to the dual of the image of the corresponding fibrational term. Explicitly, we define:
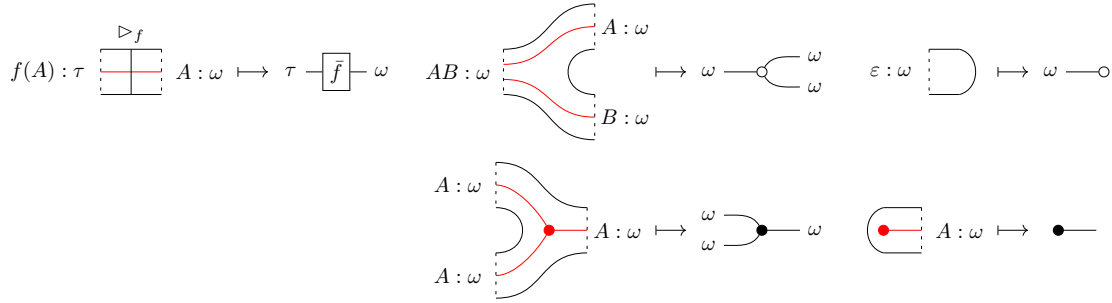
\begin{equation*}
\scalebox{.8}{\input{journal-figures/total-to-base-fib.tikz}}.
\end{equation*}
The definition of the retrofunctors
$$\left(p_{\mathcal T}, \varphi_{T,S}\right) : T_{\mathsf{defl}}(\mathcal T)(T\mid S)\rightarrow \Zg(B(\mathcal L))(pT,pS)$$
is more interesting. First, for every $f\in\mathcal F(\omega,\tau)$ define the functions
\begin{align*}
\varphi_{A:\omega,B:\omega}(-,\eta_f) : T_{\mathsf{defl}}(\mathcal T)_{\omega}(A:\omega\mid B:\omega)_0 &\rightarrow T_{\mathsf{defl}}(\mathcal T)(A:\omega\mid B:\omega)_1 \\
\varphi_{C:\tau,D:\tau}(-,\varepsilon_f) : T_{\mathsf{defl}}(\mathcal T)_{\bar f;f}(C:\tau\mid D:\tau)_0 &\rightarrow T_{\mathsf{defl}}(\mathcal T)(C:\tau\mid D:\tau)_1
\end{align*}
via the following assignments:
\begin{equation*}
\scalebox{.8}{\input{journal-figures/local-retrofunctor-free-model.tikz}},
\end{equation*}
where $x:(A:\omega\mid B:\omega)$ is any term in the fibre above $\omega$, while $y:(C:\tau\mid f(A):\tau)$ and $z:(f(A):\tau\mid D)$ are any terms in the fibre above $\tau$: note that the terms in the fibres are precisely the internal terms. Further, note that the structural equations guarantee that it does not matter which side of the term $x$ the unit is applied at: the resulting 2-cells are equal. When defining the retrofunctor on the counit, note that indeed any term above $\bar f;f$ decomposes in the displayed way, by pushing any internal terms between $\coarsen_f$ and $\refine_f$ outside: the structural equations guarantee that it does not matter which side the terms are pushed to. The functions $\varphi_{T,S}$ are defined analogously on any opfibrational-fibrational term pair. Since the only non-trivial 2-cells in $\Zg(B(\mathcal L))$ are the units and the counits, it remains to extend the functions to composite terms, which is done by the following recursion:
\begin{align*}
\varphi_{(T,T'),(S,S')}\left(x\otimes y,\eta_{f\times g}\right) &\coloneq\varphi_{T,S}\left(x,\eta_f\right)\otimes\varphi_{T',S'}\left(y,\eta_g\right) \\
\varphi_{(T,T'),(S,S')}\left(x\otimes y,\varepsilon_{f\times g}\right) &\coloneq\varphi_{T,S}\left(x,\varepsilon_f\right)\otimes\varphi_{T',S'}\left(y,\varepsilon_g\right) \\
\varphi_{T,S}(x,\alpha;\alpha') &\coloneq \varphi_{T,S}(x,\alpha);\varphi_{T,S}(y,\alpha') \\
\varphi_{T,S}(x;y,\alpha*\alpha') &\coloneq \varphi_{T,U}(x,\alpha)*\varphi_{U,S}(y,\alpha').
\end{align*}
\begin{remark}
Note that it would have been sufficient to only define $\varphi_{A:\omega,A:\omega}\left(\id_A,\eta_f\right)$ on all the internal identity terms and, likewise, $\varphi_{f(A):\tau,f(A):\tau}\left(\coarsen_f;\refine_f,\varepsilon_f\right)$ on only the opfibrational-fibrational composites -- the above recursion then takes care of all the other terms. However, we find it more intuitive to define the whole function for a fixed (co)unit at once. This also makes it evident that the construction is well-defined with respect to the structural identities.
\end{remark}

\begin{proposition}
For any deflational theory $\mathcal T$ with the layered signature $\mathcal L$, the local retrofunctor
$$\left(p_{\mathcal T},\varphi\right):T_{\mathsf{defl}}(\mathcal T)\rightarrow \Zg(B(\mathcal L))$$ is a split monoidal deflation.
\end{proposition}
\begin{proof}
The functor $p$ is a decomposition lifting by construction: given a decomposition in the base category, its chosen lifting is given by the corresponding sequence of (op)fibrational terms (i.e.~by ``inflating'' the diagrams), interspersed with the internal terms present in the original morphism. Any two such liftings are equal, since the internal terms slide inside the diagrams. Since the codomain of $\varphi_{S,S}\left(\bar F;F,\varepsilon_f\right)$, where $F:T\rightarrow S$ is a lifting of $f$ (i.e.~some composite of opfibrational terms), is $\id_S$ by definition, we indeed have a split deflation.

To show that the deflation is monoidal, we use Lemma~\ref{lma:monoidal-deflation-opfibration-indexed-monoids}: it is equivalent to show that the opfibrational restriction has indexed monoids. We observe that the opfibrational restriction gives the syntactic opfibrational model $p_{\mathcal T}:T(\mathcal T)\rightarrow B(\mathcal L)$, which, by Proposition~\ref{prop:syntactic-opfibrational-model} has indexed monoids.
\end{proof}

\begin{corollary}
Any deflational theory $\mathcal T=(\mathcal L,E^0,E^1,\eta,E^2)$ gives rise to the deflational model $\left(p_{\mathcal T},\varphi,i\right)$ whose interpretation functions $i$ are given by:
\begin{align*}
i(\omega) &=\omega, &
i_{\omega,\tau}(f) &=f, &
i^{\omega}(a) &= a:\omega, &
i^{\omega}_{A,B}(\sigma) &= \sigma : (A:\omega\mid B:\omega), &
i_{x,y}(\alpha) &=\alpha.
\end{align*}
\end{corollary}
As for opfibrational theories, the syntactic models give the left adjoint to the forgetful functors from models to theories.
\begin{theorem}\label{thm:deflational-adjoint}
The construction of deflational models from deflational theories extends to a functor $\DeflTh\rightarrow\DeflThMod$, which is moreover the left adjoint to the forgetful functor.
\end{theorem}
\begin{proof}
Any morphism of deflational theories $F:\mathcal T\rightarrow\mathcal K$ recursively extends to types, terms and 2-terms, as shown in Section~\ref{sec:layered-theories}, giving the 2-functor between the total categories $T_{\mathsf{defl}}(\mathcal T)\rightarrow T_{\mathsf{defl}}(\mathcal K)$. The induced functor between the base categories $\Zg(B(\mathcal T))\rightarrow\Zg(B(\mathcal K))$ is given by extending the action of $F$ on the layers and the generators to a morphism of indexed monoids, and then further extending to a morphism of zigzag 2-categories. This gives a morphism of deflations compatible with the morphism of theories, and hence a morphism in $\DeflThMod$.

By Proposition~\ref{prop:deflational-model-preserves-structural}, the syntactic model of $\mathcal T$ is the initial object of the fibre $\OpFThMod(\mathcal T)$. Moreover, any morphism out of the term model of $\mathcal T$ into any model of $\mathcal K$ factors through the term model of $\mathcal K$, thus exhibiting the desired adjunction.
\end{proof}
\begin{corollary}\label{cor:deflational-completeness}
For any deflational theory $\mathcal T$ and layered signature $\mathcal L$, we have the following equivalences:
\begin{align*}
\DeflThMod(\mathcal T) &\simeq \quot{F(\mathcal T)}{\MonDefl_{\mathsf{sp}}}, \\
\DeflMod(\mathcal L) &\simeq \quot{F(\mathcal L)}{\MonDefl_{\mathsf{sp}}}.
\end{align*}
\end{corollary}

\subsection{Deflational theories and (op)indexed monoidal categories}\label{subsec:defl-th-opin-moncat}

Given a deflational theory, there are two inequivalent ways to obtain an opindexed monoidal category from it\footnote{Note that diagram~\ref{eq:deflational-thy-to-indexed-monoidal} does not commute!}:
\begin{equation}\label{eq:deflational-thy-to-indexed-monoidal}
\scalebox{1}{\input{journal-figures/deflational-thy-to-indexed-monoidal.tikz}}.
\end{equation}
First, one can take the underlying opfibrational theory, form its syntactic category, and apply the equivalence of opfibrations with indexed monoids and opindexed monoidal categories (the down-right path in diagram~\eqref{eq:deflational-thy-to-indexed-monoidal}). Second, one can form the syntactic category of the deflational theory, take its opfibrational restriction and then apply the equivalence (the right-down path). The resulting opindexed monoidal categories will not, in general, be equivalent, as the first approach ignores information (the 2-cells). Here, however, we focus on the case when the two approaches {\em are} equivalent.

A deflational theory is {\em minimal} if it has no generating 2-cells. Thus, a minimal deflational theory only has the structural 2-cells. Clearly, for a minimal deflational theory the two ways of obtaining an indexed monoidal category are equivalent.

Now consider a (not necessarily minimal) deflational theory $\mathcal T$. Let us write $\mathcal T_{\varepsilon}$ for the theory obtained from $\mathcal T$ by adding a section to each counit 2-cell in Figure~\ref{fig:structural-twocells-adjoints}, i.e.~$\kappa$ such that $\kappa;\varepsilon=\id$.
\begin{proposition}
For any deflational theory $\mathcal T$, both $\mathcal T$ and $\mathcal T_{\varepsilon}$ result in the same (op)indexed monoidal category via first constructing the syntactic deflational category, then restricting to an im-opfibration (the right-down path in diagram~\eqref{eq:deflational-thy-to-indexed-monoidal}). Moreover, if $\mathcal T$ is minimal, then $\mathcal T_{\varepsilon}$ gives rise to the same indexed monoidal category via both paths in diagram~\eqref{eq:deflational-thy-to-indexed-monoidal}.
\end{proposition}
\begin{proof}
Requiring the generating counits in the total category to be surjective corresponds to the components $\varepsilon_{fa,fa}$ of the counits as defined in Proposition~\ref{prop:prof-embedding-adjoints} having sections. But such components always have a section by Remark~\ref{rem:counit-surjective-twocells}, so the requirement is already satisfied.
\end{proof}

The discussion in this subsection shows that (op)indexed monoidal categories are inadequate as a semantics for deflational theories, in the sense that they fail to distinguish between distinct theories. This is one of the reasons we chose to interpret our theories in monoidal deflations rather than indexed monoidal categories. An interesting question thus arises: what are the appropriate 2-cells between (op)indexed monoidal categories (or their embedding into $\Prof$) that would remedy this (see Remark~\ref{rem:minimal-deflation})?

\section{Discussion and future work}\label{sec:discussion}
Our paper is devoted to providing a semantic interpretation of layered monoidal theories. We have indeed identified three kinds of closely related semantics: split opfibrations with indexed monoids (Definition~\ref{def:opfib-indexed-mon}), split fibrations with indexed comonoids, and split monoidal deflations (Definition~\ref{def:monoidal-deflation}). In each case, we have showed soundness and completeness with respect to respective class of theories by constructing a free-forgetful adjunction. In order to represent bidirectional functor boundaries, we have introduced (monoidal) deflations (Definitions~\ref{def:deflation} and~\ref{def:monoidal-deflation}). This provides a precise mathematical domain for interpreting deflational theories, while staying as close as possible to (op)fibrations. We conclude by discussing related work and by sketching directions for future theoretical developments.

\subsection{Related work}\label{subsec:related-work}

We remark that the work related to the syntax of layered monoidal theories has already been discussed in Part~I~\cite[Subsection~6.1]{lmt-part1}. Here we, therefore, focus on the work related to the semantics -- (op)indexed monoidal categories and deflations.

\paragraph{Ponto-Shulman string diagrams for monoidal fibrations.}
Ponto and Shulman~\cite{ponto-shulman12} have introduced string diagrams to reason about the situation where every category in the image of an $\cat X$-indexed monoidal category (Definition~\ref{def:opindexed-monoidal-category}) is symmetric, and the indexing category $\cat X$ has finite products. Moeller and Vasilakopoulou~\cite{monoidal-gro} have shown that in such case the fiberwise and the ``global'' monoidal structure coincide. It would be interesting to see whether such diagrams can be obtained as a layered monoidal theory.

\paragraph{Free adjoint construction.}
A zigzag 2-category (Definition~\ref{def:zigzag-category}) is obtained from a 1-category by freely adding a right adjoint to each morphism: for each morphism, we add a new morphism in the opposite direction, as well as a minimal amount of 2-cells to make the newly added morphism the right adjoint of the original one. This is somewhat brute force, as the 2-cells are simply postulated by {\em fiat}. A very similar construction is given by Dawson, Paré and Pronk~\cite{adjoining-adjoints03}, where the 2-cells, rather than being additional structure, are diagrams of a particular shape in the original 1-category called {\em fences}. The authors show that this construction indeed makes each arrow of the original 1-category (seen as embedded into the constructed 2-category) a left adjoint, and that the construction is, moreover, universal with respect to functors that send morphisms to left adjoints. It would, therefore, be interesting to investigate whether zigzag 2-categories defined here are equivalent to the construction of~\cite{adjoining-adjoints03}.

\subsection{Further developments}\label{subsec:further-dev}

Layered monoidal theories have many promising applications, which we discuss in Part~I~\cite[Subsection~6.2]{lmt-part1}. Here we discuss technical developments pertaining to the semantics.

\paragraph{2-cells of deflations.}
The 2-cells of deflations do not have a clear counterpart in (op)indexed categories, or their embeddings into $\Prof$ (see Remark~\ref{rem:minimal-deflation} and Subsection~\ref{subsec:defl-th-opin-moncat}). What they correspond to remains to be investigated.

\paragraph{Extending the definition of a deflation.}
Currently, the definition of a deflation requires the codomain be a zigzag 2-category $\Zg(\cat X)$, which seems like an {\em ad hoc} restriction. The improved definition should, instead, require existence, preservation and reflection of finite products, as well as that every 1-cell in the codomain has a right adjoint, akin to various flavours of definitions hyperdoctrines. Closely related to this is investigating the case of global monoidal products, which we address next.

\paragraph{Global monoidal products.}
When defining string diagrams for fiberwise monoidal products, we noticed that it is convenient to freely add cartesian products to the base category, and impose the condition on the opfibration that it preserves and reflects products. This gives a hint on how the string diagrams for the ``global'' monoidal structure should look like. In this case, both total and the base category are monoidal, the (op)fibration is a strict monoidal functor, while the monoidal product of the total category preserves (op)cartesian liftings~\cite{monoidal-gro}.

\paragraph{Non-strict theories.}
As mentioned in the Preliminaries (Subsection~\ref{subsec:opfibrations-indexed-monoidal}), it would be interesting to study graphically the non-strict case of the equivalence between (op)fibrations with monoids and (op)indexed monoidal categories (Theorem~\ref{thm:moeller-and-vasilakopoulou}). Layered monoidal theories provide enough generality to make this possible: one would have to define a theory where both 0-equations and 1-equations are empty, and account for all identifications by isomorphic 1-cells.

\phantomsection
\addcontentsline{toc}{section}{References}
\bibliographystyle{plainnat}
\bibliography{bibliography}

\onecolumn\newpage

\end{document}